\documentclass[10pt]{amsart}
\usepackage{amsmath}
\usepackage{amscd,amsthm,amssymb,amsfonts}
\usepackage{mathrsfs}
\usepackage{dsfont}
\usepackage{stmaryrd}
\usepackage{euscript}
\usepackage{expdlist}
\usepackage{enumerate}

\input xy
\xyoption{all}
\usepackage[OT2,T1]{fontenc}

\theoremstyle{plain}
\newtheorem{thm}{Theorem}[section]
\newtheorem{thmA}{Theorem}

\newtheorem*{thm*}{Theorem}

\newtheorem{lm}[thm]{Lemma}
\newtheorem{cor}[thm]{Corollary}
\newtheorem*{cor*}{Corollary}
\newtheorem{prop}[thm]{Proposition}
\newtheorem*{conj*}{Conjecture}

\theoremstyle{remark}
\newtheorem*{remark}{Remark}
\newtheorem*{thank}{Acknowledgments}

\theoremstyle{definition}
\newtheorem*{defn*}{Definition}
\newtheorem{Remark}[thm]{Remark}
\newtheorem{defn}[thm]{Definition}

\newcommand{\nc}{\newcommand}

\newcommand{\beq}{\begin{equation}}
\newcommand{\eeq}{\end{equation}}
\newcommand{\bpmx}{\begin{pmatrix}}
\newcommand{\epmx}{\end{pmatrix}}
\newcommand{\bbmx}{\begin{bmatrix}}
\newcommand{\ebmx}{\end{bmatrix}}
\newcommand{\wh}{\widehat}
\newcommand{\wtd}{\widetilde}

\newcommand{\beqcd}[1]{\begin{equation*}\label{#1}\tag{#1}}
\newcommand{\eeqcd}{\end{equation*}}

\numberwithin{equation}{section}
\newenvironment{mylist}{
  \begin{enumerate}{}{%
      \setlength{\itemsep}{5pt} \setlength{\parsep}{0in}
      \setlength{\parskip}{0in} \setlength{\topsep}{0in}
      \setlength{\partopsep}{0in}
      \setlength{\leftmargin}{0.17in}}}{\end{enumerate}}

\def\parref#1{\ref{#1}}
\def\thmref#1{Theorem~\parref{#1}}
\def\propref#1{Prop.~\parref{#1}}
\def\corref#1{Cor.~\parref{#1}}     \def\remref#1{Remark~\parref{#1}}
\def\secref#1{\S\parref{#1}}
\def\chref#1{Chapter~\parref{#1}}
\def\lmref#1{Lemma~\parref{#1}}
\def\subsecref#1{\S\parref{#1}}
\def\defref#1{Def.~\parref{#1}}
\def\assref#1{Assumption~\parref{#1}}
\def\hypref#1{Hypothesis~\parref{#1}}
\def\subsubsecref#1{\S\parref{#1}}
\def\egref#1{Example~\parref{#1}}
\def\conjref#1{Conjecture~\parref{#1}}
\def\makeop#1{\expandafter\def\csname#1\endcsname
  {\mathop{\rm #1}\nolimits}\ignorespaces}
\makeop{Hom}   \makeop{End}   \makeop{Aut}   %\makeop{Isom}
\makeop{Pic} \makeop{Gal}       \makeop{Div} \makeop{Lie}
\makeop{PGL}   \makeop{Corr} \makeop{PSL} \makeop{sgn} \makeop{Spf}
 \makeop{Tr} \makeop{Nr} \makeop{Fr} \makeop{disc}
\makeop{Proj} \makeop{supp} \makeop{ker}   \makeop{Im} \makeop{dom}
\makeop{coker} \makeop{Stab} \makeop{SO} \makeop{SL} \makeop{SL}
\makeop{Cl}    \makeop{cond} \makeop{Br} \makeop{inv} \makeop{rank}
\makeop{id}    \makeop{Fil} \makeop{Frac}  \makeop{GL} \makeop{SU}
\makeop{Trd}   \makeop{Sp} \makeop{Tr}    \makeop{Trd} \makeop{Res}
\makeop{ind} \makeop{depth} \makeop{Tr} \makeop{st} \makeop{Ad}
\makeop{Int} \makeop{tr}    \makeop{Sym} \makeop{can} \makeop{SO}
\makeop{torsion} \makeop{GSp} \makeop{Tor}\makeop{Ker} \makeop{rec}
\makeop{Ind} \makeop{Coker}
 \makeop{vol} \makeop{Ext} \makeop{gr} \makeop{ad}
 \makeop{Gr}\makeop{corank} \makeop{Ann}
\makeop{Hol} %Holomorphic
\makeop{Fitt} \makeop{Mp} \makeop{CAP}

\def\Sel{Sel}

\def\Ord{{\mathrm{ord}}}

\DeclareMathOperator{\Eis}{Eis}

\DeclareMathAlphabet{\mathpzc}{OT1}{pzc}{m}{it}
\DeclareSymbolFont{cyrletters}{OT2}{wncyr}{m}{n}
\DeclareMathSymbol{\SHA}{\mathalpha}{cyrletters}{"58}

\def\makebb#1{\expandafter\def
  \csname bb#1\endcsname{{\mathbb{#1}}}\ignorespaces}
\def\makebf#1{\expandafter\def\csname bf#1\endcsname{{\bf
      #1}}\ignorespaces}
\def\makegr#1{\expandafter\def
  \csname gr#1\endcsname{{\mathfrak{#1}}}\ignorespaces}
\def\makescr#1{\expandafter\def
  \csname scr#1\endcsname{{\EuScript{#1}}}\ignorespaces}
\def\makecal#1{\expandafter\def\csname cal#1\endcsname{{\mathcal
      #1}}\ignorespaces}
\def\doLetters#1{#1A #1B #1C #1D #1E #1F #1G #1H #1I #1J #1K #1L #1M
                 #1N #1O #1P #1Q #1R #1S #1T #1U #1V #1W #1X #1Y #1Z}
\def\doletters#1{#1a #1b #1c #1d #1e #1f #1g #1h #1i #1j #1k #1l #1m
                 #1n #1o #1p #1q #1r #1s #1t #1u #1v #1w #1x #1y #1z}
\doLetters\makebb   \doLetters\makecal  \doLetters\makebf
\doLetters\makescr
\doletters\makebf   \doLetters\makegr   \doletters\makegr

\normalsize

\makeop{Ram} \makeop{Rep} \makeop{mass}

\makeop{Bl}
\def\abs#1{\left|#1\right|}
\def\norm#1{\lVert#1\rVert}
\def\Fpbar{\bar\bbF_p}

\def\Qp{\Q_p}
\def\Qbar{\bar\Q}
\def\Zbar{\bar{\Z}}

\def\Zp{\Z_p}

\def\rhobar{\bar{\rho}}
\def\rmT{{\mathrm T}}
\def\rmN{{\mathrm N}}

\def\cA{{\mathcal A}}  %automorphic forms

\def\cD{\mathcal D}

\def\cG{{\mathcal G}}
\def\cL{{\mathcal L}}
\def\cH{{\mathcal H}}
\def\cI{\mathcal I}

\def\cK{{\mathcal K}}  %imaginary quadratic field
\def\cM{\mathcal M}
\def\cR{{\mathcal R}}
\def\cO{\mathcal O}
\def\cS{{\mathcal S}}
\def\cf{{\mathcal f}}
\def\cW{{\mathcal W}}

\def\cZ{\mathcal Z}
\def\cX{\mathcal X}

\def\cP{{\mathcal P}}

\def\cU{\mathcal U}

\def\bfc{\mathbf c}

\def\bfM{\mathbf M}

\def\bff{\mathbf f}

\def\bfw{\mathbf w}

\def\bftheta{\boldsymbol{\theta}}

\def\bfal{\boldsymbol{\alpha}}

\def\bbI{\mathbb I}

\newcommand{\Z}{\mathbf Z}
\newcommand{\Q}{\mathbf Q}
\newcommand{\R}{\mathbf R}
\newcommand{\C}{\mathbf C}
\newcommand{\A}{\mathbf A}    % for adele

\def\bbmu{\boldsymbol{\mu}}

\def\frakp{{\mathfrak p}}
\def\frakP{\mathfak P}
\def\frakq{\mathfrak q}

\def\frakm{\mathfrak m}

\def\frakl{\mathfrak l}
\def\frakP{\mathfrak P}

\def\frakD{\mathfrak D}

\def\frakX{\mathfrak X}

\def\frakN{\mathfrak N}

\def\bfone{{\mathbf 1}}

\def\wbar{\bar{w}}

\def\wbar{\bar{w}}

\def\padic{\text{$p$-adic }}

\def\Frob{\mathrm{Frob}}
\newcommand{\<}{\langle}   %\< is not defined yet.
\renewcommand{\>}{\rangle} %\> is already defined.

\def\isoto{\stackrel{\sim}{\to}}

\def\surjto{\twoheadrightarrow}

\def\ot{\otimes}

\def\hookto{\hookrightarrow}
\def\longto{\longrightarrow}
\def\ol{\overline}  \nc{\opp}{\mathrm{opp}} \nc{\ul}{\underline}

\newcommand{\pair}[2]{\< #1, #2\>}

\newcommand{\pairing}{\pair{\,}{\,}}

\def\XYmatrix{\xymatrix@M=8pt} % make \xymatrix not too cluttered
\def\ncmd{\newcommand}
\ncmd{\xysubset}[1][r]{\ar@<-2.5pt>@{^(-}[#1]\ar@<2.5pt>@{_(-}[#1]}
\ncmd{\XYmatrixc}[1]{\vcenter{\XYmatrix{#1}}}
\ncmd{\xyto}[1][r]{\ar@{->}[#1]}
\ncmd{\xyinj}[1][r]{\ar@{^(->}[#1]}
\ncmd{\xysurj}[1][r]{\ar@{->>}[#1]}
\ncmd{\xyline}[1][r]{\ar@{-}[#1]}
\ncmd{\xydotsto}[1][r]{\ar@{.>}[#1]}
\ncmd{\xydots}[1][r]{\ar@{.}[#1]}
\ncmd{\xyleadsto}[1][r]{\ar@{~>}[#1]}
\ncmd{\xyeq}[1][r]{\ar@{=}[#1]} \ncmd{\xyequal}[1][r]{\ar@{=}[#1]}
\ncmd{\xyequals}[1][r]{\ar@{=}[#1]}
\ncmd{\xymapsto}[1][r]{l\ar@{|->}[#1]}\ncmd{\xyimplies}[1][r]{\ar@{=>}[#1]}
\ncmd{\xyiso}{\ar[r]_-{\sim}}
\def\injxy{\ar@{^(->}}

\newcommand{\pMX}[4]{\begin{pmatrix}
{#1}& {#2}\\
{#3}&{#4}\end{pmatrix} }

 \newcommand{\DII}[2]{\begin{pmatrix}{#1}&0
 \\0&{#2}\end{pmatrix}}

\newcommand{\seesaw}[4]{{#1}\ar@{-}[rd]\ar@{-}[d]&{#2}\ar@{-}[d]\\
{#3}\ar@{-}[ru]&{#4}}

\def\ie{i.e. }

\def\cf{\mbox{{\it cf.} }}
\def\loccit{\mbox{{\it loc.cit.} }}

\def\mapR{\smash{\mathop{\longrightarrow}}}

\newcommand{\exact}[3]{
0\mapR{#1}\mapR{#2}\mapR{#3}\mapR 0 }

\def\uf{\varpi} %uniformizer
\def\Abs{{|\!\cdot\!|}} %adelic absolute value
\def\ndivides{\nmid}

\def\x{{\times}}

\def\onehalf{{\frac{1}{2}}}
\def\e{\varepsilon} % episilon factor
\def\al{\alpha}

\def\om{\omega}
\def\dirlim{\varinjlim}
\def\prolim{\varprojlim}
\def\iso{\simeq}
\def\con{\equiv}
\def\bksl{\backslash}
\newcommand\stt[1]{\left\{#1\right\}}
\def\ep{\epsilon}

\def\lam{\lambda}
\def\pii{\pi i}

\def\sg{\sigma}
\def\vp{\varphi}
\def\disjoint{\bigsqcup}

\def\dx{d^\x}

\def\AKf{\A_{\cK,f}}

\def\AK{\A_\cK}
\def\setp{{(p)}}

\newcommand{\powerseries}[1]{\llbracket{#1}\rrbracket}

\renewcommand\pmod[1]{\,(\mbox{mod }{#1})}

\renewcommand\Re{\text{Re}\,}
\newcommand\Dmd[1]{\left<{#1}\right>} %Diamond operator

\def\Cp{\C_p}

\def\ram{{\mathrm{ram}}}

\def\alg{\mathrm{alg}} 
\usepackage{hyperref}
\def\CMP{\bftheta}
\def\cmpt{\varsigma}
\def\pme{q}
\def\cK{K}
\def\setn{{(n)}}

\def\cmptv{\cmpt_\pme}
\def\w{\frakq}
\def\wbar{\ol{\frakq}}
\def\frakp{p}
\def\frakP{{\mathfrak p}}
\def\Af{\A_f}
\newcommand\localW[1]{W_{#1}}
\def\addchar{\psi}
\def\Prd{\boldsymbol{P}}
\def\cmJ{J}
\def\OKbasis{\bftheta}

\def\frakPbar{\ol{\frakP}}
\def\Eichler{R}
\def\hatR{\wh R}
\newcommand\localP[2]{\Prd(#1,#2)}
\newcommand\Lpair[2]{\pair{#1}{#2}_k}
\newcommand\PetG[2]{\pair{#1}{#2}_{\GL_2}}
\newcommand\PetB[2]{\pair{#1}{#2}_{\Eichler}}
\def\lnew{\varphi_\pme}
\def\OFp{\Zp}

\def\wt{k}

\def\bfv{{\mathbf v}}
\def\p{p}
\def\vform{f}
\def\vformB{f_{\pi'}}
\def\Form{\varphi_{\pi'}}
\def\ALI{\tau^{N_B}}
\def\ALIG{\boldsymbol{\tau}^{N}}
\def\Qq{\Q_q}
\def\newForm{\varphi_\pi}
\def\newW{W}
\def\MF{\bfM}
\def\padicMF{\cM}
\def\#{\sharp}
\def\GrossC{\cX}
\def\CMspace{\mathrm{CM}}
\def\Red{\mathrm{Red}}

\def\hatB{\wh B^\x}
\def\hatK{\wh K^\x}
\def\hatQ{\wh \Q^\x}

\def\Sgsc{Q_2}
\def\Adjrep{W_\rho}
\def\bfal{A}
\def\Diff{\delta}
\thanks{The first author is partly supported by
Grant-in-Aid for Young Scientists (B) No. 23740015  and Hakubi project
of Kyoto University. The second author is partly supported by National Science Council grant 101-2115-M-002-010-MY2}
\title[Special values of anticyclotomic L-functions]{Special values of anticyclotomic L-functions for modular forms} %on definite quaternion algebras}
\author[M. Chida]{Masataka Chida}
\address{
Department of Mathematics, Graduate School of Science, Kyoto University,
Kitashirakawa-Oiwakecho, Sakyo-ku, Kyoto,  606-8502, Japan.
}
\address{
The Hakubi Center for Advanced Research, Kyoto University,
Yoshida-Ushinomiya-cho, Sakyo-ku, Kyoto, 606-8302, Japan
}
\email{chida@math.kyoto-u.ac.jp}
\author[M.-L. Hsieh]{Ming-Lun Hsieh}
\address{ Department of Mathematics~\\National Taiwan University ~ \\
No. 1, Sec. 4, Roosevelt Road, Taipei 10617, Taiwan~
}
\email{mlhsieh@math.ntu.edu.tw}
\date{\today}
\begin{document}
\begin{abstract}In this article, we generalize some works of Bertolini-Darmon and Vatsal on anticyclotomic $L$-functions attached to modular forms of weight two to higher weight case. We construct a class of anticyclotomic \padic $L$-functions for ordinary modular forms and derive the functional equation and the interpolation formula at all critical specializations. Moreover, we prove results on the vanishing of $\mu$-invariant of these \padic $L$-functions and the non-vanishing of central $L$-values with anticyclotomic twists.
\end{abstract}

%\begin{abstract}
%\end{abstract}

\maketitle
\tableofcontents
%!TEX root = central.tex

\def\wtsp{\frakX_p^{\mathrm{crit}}}
\def\bfa{\bfc}
\def\thmref#1{Theorem~\parref{#1}}
\def\propref#1{Proposition~\parref{#1}}
\def\corref#1{Corollary~\parref{#1}}     \def\remref#1{Remark~\parref{#1}}
\def\secref#1{\S\parref{#1}}
\def\chref#1{Chapter~\parref{#1}}
\def\lmref#1{Lemma~\parref{#1}}
\def\subsecref#1{\S\parref{#1}}
\def\defref#1{Defefinition~\parref{#1}}
\def\assref#1{Assumption~\parref{#1}}
\def\hypref#1{Hypothesis~\parref{#1}}
\def\subsubsecref#1{\S\parref{#1}}
\def\egref#1{Example~\parref{#1}}
\def\conjref#1{Conjecture~\parref{#1}}
\section*{Introduction}
The purpose of this article is to generalize works of Bertolini, Darmon and Vatsal on anticyclotomic \padic $L$-functions attached to modular forms of weight two to higher weight. We construct anticyclotomic \padic $L$-functions for ordinary modular forms and prove the interpolation property at all critical specializations and the functional equation. In addition, following the ideas of Vatsal, we prove results on the vanishing of the $\mu$-invariant of this anticyclotomic \padic $L$-functions and the non-vanishing modulo $\ell$ of central $L$-values with anticyclotomic twists. To state our results precisely, we introduce some notation. Let $f\in S_k(\Gamma_0(N))$ be an elliptic new form of weight $k$ (even) and conductor $N$. Let $p$ be a rational prime. Let $\cK$ be an imaginary quadratic field of discriminant $-D_\cK<0$.  Then $K$ determines a factorization
\[N=p^{n_p}N^+N^-\quad (p,N^+N^-)=1,\]
where $N^+$ (resp. $N^-$) is only divisible by primes that are split (resp. inert or ramified) in $\cK$. We assume that
\[\text{ $N^-$ is the square-free product of an odd number of primes.}\]
Fix a decomposition $N^+\cO_\cK=\frakN^+\ol{\frakN^+}$. For each finite prime, let $\ep_\pme(f)\in \stt{\pm 1}$ denote the local root number of $f$, so $\ep_\pme(f)=1$ if $\pme\ndivides N$, and $\ep_\pme(f)$ is the eigenvalue of Atkin-Lehner involution at $\pme$ if $\pme|N$. The global root number $\ep(f):=(-1)^\frac{k}{2}\prod_{\pme}\ep_\pme(f)$. Let
\[f(q)=\sum_{n=1}^\infty \bfa_n(f)q^n\] be the $q$-expansion of $f$ at the infinity cusp. It is known that $\ep_\pme(f)=-\pme^\frac{2-k}{2}\bfa_\pme(f)$ if $\pme\parallel N$. Let $\bfal_p$ be a complex root of $X^2-\bfa_p(f)X+p^{k-1}$ if $p\ndivides N$ (or $X^2-\bfa_p(f)X$ if $p|N$). Let $\cK_\infty^-$ be the anticyclotomic $\Zp$-extension over $\cK$ and let $\Gamma^-=\Gal(\cK_\infty^-/\cK)$ be the Galois group. Denote by $\rec_\cK:\AK^\x\to G_\cK^{ab}$ the geometrically normalized reciprocity law map. Fix embeddings $\iota_\infty:\Qbar\hookto\C$ and $\iota_p:\Qbar\hookto\Cp$. To each locally algebraic \padic character $\wh\chi:\Gamma^-\to\Cp^\x$ of ($p$-adic) weight $(m,-m)\in\Z^2$, we can associate a Hecke character $\chi:\AK^\x/\cK^\x\to\C^\x$ of (archimedean) weight $(m,-m)$ defined by
\[\chi(a):=\iota_\infty\iota_p^{-1}(\wh\chi(\rec_\cK(a))(\ol{a}_p/a_p)^{m})(a_\infty/\ol{a}_\infty)^{m},\]
where $a_p\in (\cK\ot_\Q\Qp)^\x$ and $a_\infty\in (\cK\ot_\Q\R)^\x$ are the $p$-component and $\infty$-component of $a$. We call $\wh\chi$ the \emph{\padic avatar} of $\chi$. Let $\wtsp$ be the set of \emph{critical specializations} consisting of locally algebraic \padic characters $\wh\chi:\Gamma^-\to\Cp^\x \text{ of weight }(m,-m)$ with \[-k/2<m<k/2.\]

Let $\cO_{f}$ be the ring of integers of the Hecke field of $f$. Fix a prime $\lam$ of $\Qbar$ and let $\cO_{f,\lam}$ be the completion of $\cO_f$ with respect to $\lam$. Suppose that $\lam$ is induced by $\iota_p$ (so $\lam$ has residue characteristic $p$). Denote by  $L(f/\cK,\chi,s)$ the Rankin $L$-series associated with $f$ and the theta series attached to $\chi$. Our first theorem is the construction of the anticyclotomic \padic $L$-function attached to $f$ over $\cK$ with the explicit evaluation formula at critical specializations.
\begin{thmA}\label{T:A}Suppose that \begin{itemize}\item[(a)] $p>k-2$,
\item[(b)] $\bfal_p$ is a $\lam$-adic unit.
\end{itemize} Then there exist an element $\Theta_\infty\in\cO_{f,\lam}\powerseries{\Gamma^-}$ and a complex number $\Omega_{f,N^-}\in\C^\x$ such that for every $\wh\chi\in \wtsp$ of weight $(m,-m)$ and conductor $p^n$, we have the following interpolation formula:
\[\begin{aligned}\wh\chi(\Theta_\infty^2)=&\Gamma(k/2+m)\Gamma(k/2-m)\cdot \frac{L(f/\cK,\chi,k/2)}{\Omega_{f,N^-}} \cdot e_p(f,\chi)^{2-\Ord_p(N)}\cdot p^n\bfal_p^{-2n}(p^{n}D_\cK)^{k-2}\\
&\times u_\cK^2\sqrt{D_\cK}\cdot \ep_p(f)(-1)^m\prod_{\pme|(D_\cK,N^-)}(1-\ep_\pme(f))\cdot \chi(\frakN^+),\end{aligned}\]
where $u_\cK=\#(\cO_\cK^\x)/2$ and $e_p(f,\chi)$ is the \padic multiplier given by  \[e_p(f,\chi)=\begin{cases}1&\text{if $n>0$,}\\
(1-\chi(\frakP)p^\frac{k-2}{2}\bfal_\frakp^{-1})(1-\chi(\ol{\frakP})p^\frac{k-2}{2}\bfal_\frakp^{-1})&\text{if $n=0$ and $\frakp=\frakP\ol{\frakP}$ is split},\\
1-p^{k-2}\bfal_\frakp^{-2}&\text{if $n=0$ and $\frakp=\frakP$ is inert},\\
1-\chi(\frakP)p^\frac{k-2}{2}\bfal_\frakp^{-1}&\text{if $n=0$ and $\frakp=\frakP^2$ is ramified}.
\end{cases}\]
\end{thmA}
\begin{remark}\begin{mylist}
\item The existence of $\bfal_p$ satisfying the assumption (b) is usually referred to the $p$-ordinary hypothesis for $f$, \ie the $p$-th Fourier coefficient $\bfa_p(f)$ is a $\lam$-adic unit.
\item The complex number $\Omega_{f,N^-}$ is given by
\[\Omega_{f,N^-}=\frac{(4\pi)^k\norm{f}_{\Gamma_0(N)}}{\xi_{f}(N^+,N^-)},\]
where $\norm{f}_{\Gamma_0(N)}$ is the Petersson norm of $f$ and $\xi_{f}(N^+,N^-)\in\cO_{f,\lam}$ is an integer connected with certain congruence number of $f$. The precise definition is given in \eqref{E:periodV.W}. It is interesting and important to make a comparison between $\Omega_{f,N^-}$ and Hida's canonical period $\Omega_f$ attached to $f$. In general, we have $\Omega_{f,N^-}/\Omega_f\in\cO_{f,\lam}$. If $k=2$, then under a mild hypothesis, Pollack and Weston \cite{Pollack_Weston:AMU} have shown that this ratio is a product of local Tamagawa numbers at primes dividing $N^-$ modulo a unit in $\cO_{f,\lam}$. We will investigate this subtle problem in \secref{S:periods} for general weight $k$.
\item \thmref{T:A} indeed gives the construction of the anticyclotomic \padic $L$-function that interpolates \emph{square root} of central $L$-values. In the case $k=2$, $\Theta_\infty$ is precisely the theta element $\theta_\infty$ (with trivial tame branch character) given by Bertolini and Darmon \cite[p.436]{BD:Heegner_Mumford}. Therefore, combined with the anticyclotomic Iwasawa main conjecture for elliptic curves \cite{Bertolini_Darmon:IMC_anti}, the usual control theorem and the comparison between periods \cite[Proposition\,3.7,\,Theorem\,6.8]{Pollack_Weston:AMU}, the evaluation formula of $\Theta_\infty$ at the trivial character yields the optimal upper bound of the size of $p$-primary Selmer groups of certain elliptic curves over $\cK$ in terms of central $L$-values as predicted by Birch and Swinnerton-Dyer conjecture.

\end{mylist}
\end{remark}
Let $*:\cO_{f,\lam}\powerseries{\Gamma^-}\to\cO_{f,\lam}\powerseries{\Gamma^-}$ be the involution defined by $\sg\mapsto \sg^{-1}$. Then $\Theta_\infty$ satisfies the following functional equation.
\begin{thmA}\label{T:intro_functional}%Let $r$ be the number of prime divisors of $N^-$ that are unramified in $\cK$.
Let $\sg_{\frakN^+}$ be the image of $\frakN^+$ in $\Gamma^-$ under the reciprocity law map $\rec_\cK$. % and let \[\ep':=%(-1)^{\frac{k}{2}}\prod_{\pme\not =p}\ep_\pme(f)=
%\ep_p(f)\ep(f)\in\stt{\pm 1}.\]
Then we have the functional equation:
\[\Theta_\infty^*=\ep_p(f)\ep(f)\cdot \Theta_\infty\cdot \sg_{\frakN^+}^{-1}.\]
\end{thmA}
In the case $k=2$ and $(D_\cK,N^-)=1$, this theorem is proved in \cite[Proposition\,2.13]{BD:Heegner_Mumford} up to an element in $\Gamma^-$. The above functional equation suggests that the product $\abs{\Theta_\infty}^2:=\Theta_\infty\cdot \Theta_\infty^*$ be intrinsic. Namely, $\abs{\Theta_\infty}^2$ does not depend on the choice of an auxiliary decomposition $N^+=\frakN^+\ol{\frakN^+}$. Combined with \thmref{T:A}, this verifies the formula of $\chi(\abs{\Theta_\infty}^2)$ in \cite[Conjecture 2.12]{BD:Heegner_Mumford}. If $\chi$ is ramified character of finite order and $p\ndivides D_\cK N$, then the formula of $\chi(\abs{\Theta_\infty}^2)$ actually is a consequence of Gross special
value formulae first obtained by B. Gross in a special case $k=2$ and $N$ is a prime and generalized by Shou-Wu Zhang \cite[Theorem\,7.1]{Zhang:GZII} for $k=2$ and Haiping Yuan \cite{YuanHP:Thesis} for $k>2$.

In \cite[Theorem\,1.1]{Vatsal:nonvanishing}, Vatsal determines the $\mu$-invariant of anticyclotomic \padic $L$-functions for modular forms of weight two. Our second theorem provides a partial generalization of his result to modular forms of higher weight.
\begin{thmA}\label{T:B}Let $\rho_{f,\lam}:\Gal(\Qbar/\Q)\to\GL_2(\cO_{f,\lam})$ be the Galois representation associated to $f$. With the assumptions in \thmref{T:A}, suppose further
\begin{mylist}
\item $(D_\cK,N^-)=1$,
\item the residual representation $\ol{\rho}_{f,\lam}$ is absolutely irreducible.\end{mylist} Then the Iwasawa $\mu$-invariant of $\Theta_\infty$ vanishes.
\end{thmA}
\newtheorem{I_Remark*}{Remark}
\begin{I_Remark*}\thmref{T:B} has an important application to Iwasawa main conjecture for $\GL(2)$. Vatsal's theorem on the vanishing of $\mu$-invariant plays a key role in the proof of Iwasawa main conjecture for elliptic curves in the recent work of Skinner-Urban \cite{SU:IMC_GL(2)}. They actually prove the main conjecture for modular forms on $\Gamma_0(N)$ of weight $k\con 2\pmod{p-1}$ \cite[Theorem\,3.6.4]{SU:IMC_GL(2)}, and \thmref{T:B} enables us to lift their assumption $k\con 2\pmod{p-1}$ .
\end{I_Remark*}

Now we suppose that $\lam$ has residue characteristic $\ell\not =p$ and consider the problem of non-vanishing modulo $\lam$ of central $L$-values with anticyclotomic twists. We obtain the following result, which is a generalization and an improvement of \cite[Theorem\,1.2]{Vatsal:nonvanishing} in the weight two case.
\begin{thmA}\label{T:C}Suppose that $p^2\ndivides N$ and $(D_\cK,N^-)=1$. Let $\ell$ be a rational prime such that
\begin{mylist}
\item $\ell\ndivides pND_\cK$ and $\ell>k-2$,
 \item $\ol{\rho}_{f,\lam}$ is absolutely irreducible.\end{mylist}  Then for all but finitely many characters $\chi:\Gamma^-\to\mu_{p^\infty}$, we have
\[\frac{L(f/\cK,\chi,k/2)}{\Omega_{f,N^-}}\not\con 0\pmod{\lam}.\]
\end{thmA}
\begin{I_Remark*}\thmref{T:C} has several consequences in number theory and representation theory. In number theory, this theorem removes the assumptions on $p\ndivides D_\cK$ and the $p$-indivisibility of the class number of $\cK$ in \cite[Theorem\,1.4]{Vatsal:Uniform_CM}, and shows the finiteness of the $\ell$-primary Selmer groups of elliptic curves over $\cK_\infty^-$ in virtue of \cite{Longo_Vigni:Selmer_groups}. From representation theoretic point of view, this theorem  provides a simultaneous non-vanishing result of central $L$-values with anticyclotomic twist, and hence has application to the non-vanishing of Bessel models of certain theta lifting on $\GSp(4)$ by \cite[Theorem\,3]{Prasad:Bessel}.   \end{I_Remark*}

The construction of the theta element $\Theta_\infty$ is based on an adelic formulation of the method of Bertolini and Darmon, with which one can borrow tools from representation theory (Such kind of adelic formulation was also used by Van Order \cite{VanOrder} in the case of Hilbert modular forms of parallel weight two).
The interpolation formula is the elaboration of an explicit Waldspurger formula combined with a \padic congruence argument. We briefly describe these ideas in what follows. Let $B$ be the definite quaternion algebra over $\Q$ of the absolute discriminant $N^-$ and let $R$ be an Eichler order of level $N/N^-$. Let $\vp_f:B^\x\bksl B^\x_\A/\wh R^\x\to\Sym^{k-2}(\C^2)$ be a vector-valued automorphic new form on $B$ attached to $f$ via Jacquet-Langlands correspondence. For each positive integer $n$, let $\cO_n=\Z+p^n\cO_\cK$ be the order of $\cK$ of conductor $p^n$ and let $\cG_n=\Gal(H_n/\cK)$ be the Galois group of the ring-class field of $\cK$ of conductor $p^n$. Then the Picard group $\Pic\cO_n$ will be called \emph{Gross points} of level $p^n$, which is a homogeneous space of $\cG_n$. Our hypothesis on $N$ assures that there exists an optimal embedding $\iota_n:\cK\to B$ with respect to $(\cO_n,R)$, which in turn induces a map $\iota_n:\Pic\cO_n\to B^\x\bksl B^\x_\A/\wh R^\x$. Fix a distinguished point $P_n\in\Pic\cO_n$ and let $P_n^\dagger$ be the \emph{regularized} Gross point (See \subsecref{SS:thetalets}). The module $\Sym^{k-2}(\cO_{f,\lam}^2)$ has a natural $\cO_{f,\lam}$-basis $\stt{\bfv_j}$ indexed by integers $-k/2<j<k/2$ (See \eqref{E:basis.W}), and we can write \[\vp_f=\sum_{-k/2<j<k/2}\vp_f^{[j]}\ot\bfv_j.\]
Here $\vp_f^{[j]}:B^\x\bksl B_\A^\x\to\C$ are automorphic forms on $B^\x_\A$. One can take $p$-adically optimal normalization of $\vp_f^{[0]}$ using the integral structure $\Sym^{k-2}(\cO_{f,\lam}^2)$ (see \subsecref{SS:l_adic_modular_forms}). Moreover, under the ordinary assumption, it can be shown that the restriction of the normalized $\vp_f^{[0]}$ to regularized Gross points does take value in $\cO_{f,\lam}$. We can thus define
\[\wtd\Theta_n=\sum_{\sg\in \cG_n}\sg\ot\vp_f^{[0]}(\iota_n(\sg(P_n^\dagger)))\in\cO_{f,\lam}[\cG_n]. \]
Then $\stt{\wtd\Theta_n}_n$ is compatible with respect to the natural quotient $\cG_{n+1}\to \cG_n$. Then we obtain $\wtd{\Theta}_\infty$ by taking the limit $\stt{\wtd\Theta_n}_n\in\cO_{f,\lam}\powerseries{\cG_\infty}$, where $\cG_\infty=\prolim_n \cG_n$. The Galois group $\Gamma^-=\Gal(\cK^-_\infty/\cK)$ is the maximal $\Zp$-free quotient of $\cG_\infty$. The theta element $\Theta_\infty$in \thmref{T:A} is defined to be the projection of $\wtd\Theta_\infty$ obtained by the quotient map $\cG_\infty\to\Gamma^-$.
%and we can write
%\[\vec{\Theta}_\infty^{\bfone}=\sum_{-k/2<j<k/2}\Theta_\infty^{[j]}\ot\bfv_j\quad(\Theta_\infty^{[j]}\in\Lam=\cO_{f,\lam}\powerseries{\Gamma^-}).\]
If $\chi$ is a finite order character of conductor $p^n$, the evaluation of $\wh\chi(\Theta_\infty)^2$ indeed can be translated into an explicit Waldspurger's formula.  Let $\vp_f^\dagger$ be the \emph{$p$-stabilization} of $\vp_f^{[0]}$ with respect to $A_p$. Then $\wh\chi(\Theta_\infty)$ is essentially the global toric period given by
\[P(\vp^\dagger_f,\chi)=\int_{\cK^\x\A_\Q^\x\bksl \AK^\x}\vp_f^\dagger(\iota_n(t))\chi(t)dt.\]
The value $P(\vp_f^\dagger,\chi)^2$ is a product of local toric period integrals by the \emph{fundamental} formula of Waldspurger \cite[Proposition\,7]{Waldsupurger:Central_value}. We make an explicit calculation of these local integrals. The new input is the calculation of the local toric integral of the $p$-stabilized local new vector at $p$. It is no surprise that the \padic multiplier $e_p(f,\chi)$ is contributed by this local integral. Note that Waldspurger's formula only computes $\wh\chi(\Theta_\infty)$ for \emph{finite order} characters $\chi$. We obtain the formula of $\wh\chi(\Theta_\infty)$ for characters $\wh\chi\in\wtsp$ of infinite order by a congruence trick  (\corref{C:congruence.W}).

The proof of \thmref{T:B} is based on the uniform distribution of CM points in the zero dimensional Shimura variety attached to the definite quaternion algebra $B$, which is the idea of Vatsal in his study on the non-vanishing of anticyclotomic central $L$-values of weight two modular forms. In the higher weight situation, the new idea is to use the congruences among modular forms. Roughly speaking, we construct a weight two $\Fpbar$-valued modular form $\bff_p$ such that the evaluations of $\bff_p$ and $\vp^{[0]}_f$ at Gross points are congruent to each other. We thus reduce the problem to $\bff_p$, for which the approach of Vatsal can be applied. Since the form $\bff_p$ is not a new form in general, we have to use a stronger uniform distribution result \cite[Proposition\,2.10]{Vatsal_Cornut:Documenta} and slightly generalized Ihara's lemma (\lmref{L:Ihara2.W}). The proof of \thmref{T:C} is based on the same idea combined with a Galois average
trick.

This paper is organized as follows. After fixing basic notation and definitions in \secref{S:notation}, we give a brief review of modular forms on definite quaternion algebras and an adelic description of Gross points in \secref{S:Gross_modular_forms}.
In \secref{S:specialvalue}, we give the explicit calculation of the toric periods of $p$-stabilized modular forms based on Waldspurger's formula (\propref{P:Waldformula.W}). The calculation of the local toric integral at $p$ is carried out in \propref{P:5.W}, and the final formula is summarized in \thmref{T:central.W}. In \secref{S:ThetaElment}, we give the construction of theta elements (\defref{D:Theta.W}). The functional equation is proved in \thmref{T:functionaleq}, and the evaluation formula \thmref{T:Thetaevaluation.W} (\thmref{T:A}) is obtained by combining \propref{P:evaluation.W} and the congruence property \corref{C:congruence.W} among theta elements. In \secref{S:NV}, after preparing a key result of Vatsal-Cornut on the uniform distribution of CM points and Ihara's Lemma, we prove \thmref{T:B} (\thmref{T:1.W}) and \thmref{T:C} (\thmref{T:2.W}). Finally, in \secref{S:periods} we give a sufficient condition (\propref{P:congruence}) under which the complex number $\Omega_{f,N^-}$ equals Hida's canonical periods $\Omega_f$ up to a unit in $\cO_{f,\lam}$, applying techniques of Wiles, Taylor-Wiles and Diamond in their proofs of modularity lifting theorems.
\begin{thank}This project was initiated when the first author visited Taida Institute of Mathematical Science. He would like to thank for their hospitality.
\end{thank} 
%!TEX root = central.tex

\def\Iwahori{\cI_p}
\def\sqrtb{\sqrt{\beta}}
\def\mForm{\Form^{[m]}}
\section{Notation and definitions}\label{S:notation}

\subsection{}
If $L$ is a number field, $\cO_L$ is the ring of integers of $L$, $\A_L$ is the adele of $L$ and $\A_{L,f}$
is the finite part of $\A_L$. Let $\A=\A_\Q$.
%For $a\in\A_L$, we put \[\il_L(a):=a(\cO_L\ot\Zhat)\cap L.\]
%Denote by $G_L$ the absolute Galois group and by $\rec_L:\A_L^\x\to G^{ab}_{L}$ the geometrically normalized reciprocity law map. %If $L'/L$ is an abelian extension, denote by $\rec_{L'/L}=\pi_{L'/L}\circ\rec_L:\A_L^\x\to\Gal(L'/L)$ the composition of $\rec_L$ and the quotient map $\pi_{L'/L}:G_L^{ab}\to \Gal(L'/L)$.
Let $\addchar=\prod \addchar_\pme$ be the standard additive character
of $\A/\Q$ such that $\addchar(x_\infty)=\exp(2\pii x_\infty),\,x_\infty\in\R$. %We define $\addchar_L:\A_L/L\to\C^\x$ by
%$\addchar_L(x)=\addchar_\Q\circ\Tr_{L/\Q}(x)$.
% For $\beta\in L$, $\addchar_{L,\beta}(x)=\addchar_L(\beta x)$. If $L=\cF$, we write $\addchar$ for $\addchar_\cF$.

We fix once and for all an embedding
$\iota_\infty:\Qbar\hookto\C$ and an isomorphism
$\iota:\C\iso\C_\ell$ for each rational prime $\ell$, where $\C_\ell$ is the completion of an
algebraic closure of $\Q_\ell$. Let
$\iota_\ell=\iota\iota_\infty:\Qbar\hookto\C_\ell$ be their composition. Let $\Ord_\ell:\C_\ell\to\Q\cup\stt{\infty}$ be the $\ell$-adic valuation on $\C_\ell$ normalized so that $\Ord_\ell(\ell)=1$. We regard $L$ as a subfield in $\C$ (resp. $\C_\ell$) via
$\iota_\infty$ (resp. $\iota_\ell$) and $\Hom(L,\Qbar)=\Hom(L,\C_\ell)$.

Let $\Zbar$ be the ring of algebraic integers of $\Qbar$ and let $\Zbar_\ell$ be the $\ell$-adic completion of $\Zbar$ in $\C_\ell$. Denote by $\wh\Z$ the finite completion of $\Z$. For an abelian group $M$, let $\wh M:=M\ot_\Z\wh\Z$.

\subsection{Measures on local fields}We fix some general notation and conventions on local fields used in \secref{S:specialvalue}. Let $q$ be a place of $\Q$ and $\Abs_{\Q_q}$ be the standard absolute value on $\Q_q$. Let $F$ be a finite extension of $\Q_q$.
If $F$ is non-archimedean, we usually denote by $\uf_F$ a uniformizer of $F$. Denote by $\cO_F$ the ring of integers of $F$. Let $D_F$ be the discriminant of $F/\Q_q$. Let $\Abs_F$ be the absolute value of $F$ normalized by $\abs{x}_F=\abs{\rmN_{F/\Q_q}(x)}_{\Q_q}$. We often simply write $\abs{x}=\abs{x}_F$ for $x\in F$ if its meaning is clear from the context without possible confusion. Let $\addchar:\A/\Q\to \C^\x$ be the additive character such that $\addchar(x_\infty)=\exp(2\pii x)$.
Let $\addchar_q$ be the local component of $\addchar$ at $q$ and let $\addchar_F:=\addchar_q\circ\rmT_{F/\Q_q}$, where $\rmT_{F/\Q_q}$ is the trace from $F$ to $\Q_q$.

Let $dx$ be the Haar measure on $F$ self-dual with respect to the pairing $(x,x')\mapsto \addchar_F(xx')$. If $F$ is non-archimedean, then $\vol(\cO_F,dx)=\abs{D_F}_{\Q_q}^\onehalf$. We recall the definition of the local zeta function $\zeta_F(s)$. If $F$ is non-archimedean, then \[\zeta_F(s)=\frac{1}{1-\abs{\uf_F}_F^s}.\] If $F$ is archimedean, then \[\zeta_\R(s)=\Gamma_\R(s):=\pi^{-s/2}\Gamma(s/2);\,\zeta_\C(s)=\Gamma_\C(s):=2(2\pi)^{-s}\Gamma(s).\]
The Haar measure $\dx x$ on $F^\x$ is normalized by
\[\dx x=\zeta_{F}(1)\abs{x}_F^{-1}dx.\]
In particular, if $F=\R$, then $dx$ is the Lebesgue measure and $\dx x=\abs{x}_\R^{-1}dx$, and if $F=\C$, then $dx$ is twice the Lebesgue measure on $\C$ and $\dx x=2\pi^{-1}r^{-1}drd\theta\,\,(x=re^{i\theta})$.

%If $\mu:F^\x\to\C^\x$ is a character of $F^\x$, define the local conductor $a(\mu)$ by
%\[a(\mu)=\inf\stt{n\in \Z_{\geq 0}\mid
%\mu|_{1+\uf_F^n \cO_F}=1}.\]

\subsection{$L$-functions}
Let $F$ be a non-archimedean local field. Let $\pi$ be an irreducible admissible representation of $\GL_2(F)$. Let $L(s,\pi)$ and $\ep(s,\pi,\psi_F)$ be the associated local $L$-function and local epsilon factor respectively (\cite[Theorem\,2.18 (iv)]{Jacquet_Langlands:GLtwo}).

Let $E$ be a quadratic extension of $F$. We write $\pi_E$ for the base change of $\pi$. Let $\mu,\nu:F^\x\to\C^\x$ be two characters of $F^\x$. Suppose that either $\pi=\pi(\mu,\nu)$ is a principal series if $\mu\nu^{-1}\not =\Abs^{\pm 1}$ or $\pi=\sg(\mu,\nu)$ is a special representation if $\mu\nu^{-1}=\Abs$. Let $\chi:E^\x\to\C^\x$ be a character.  We recall the definition of local $L$-functions $L(s,\pi_{E}\ot\chi)$ (\cite[\S 20]{Jacquet:GLtwoPartII}). If $E=F\oplus F$, then we write $\chi=(\chi_1,\chi_2):F^\x\oplus F^\x\to\C^\x$ and put
\begin{align*}L(s,\pi_{E}\ot\chi)=&\begin{cases}L(s,\pi\ot\chi_1)L(s,\pi\ot\chi_2)&\text{ if $\mu\nu^{-1}\not =\Abs^{\pm 1}$},\\
L(s,\mu\chi_1)L(s,\mu\chi_2)&\text{ if $\mu\nu^{-1}=\Abs$}.\end{cases}
\intertext{If $E$ is a field, then} L(s,\pi_{E}\ot\chi)=&\begin{cases}L(s,\mu'\chi )L(s,\nu'\chi)&\text{ if $\mu\nu^{-1}\not =\Abs^{\pm 1}$},\\
L(s,\mu'\chi )&\text{ if $\mu\nu^{-1}=\Abs$}.\end{cases}\end{align*}
Here $\mu'=\mu\circ\rmN_{E/F},\,\nu'=\nu\circ\rmN_{E/F}$ are characters of $E^\x$.
\subsection{Whittaker functions on $\GL_2(\Qq)$}\label{SS:Whittaker.W}
Let $\pme$ be a place of $\Q$ and let $\pi$ be an admissible irreducible representation of $\GL_2(\Qq)$ with the trivial central character. Let $\cW(\pi,\addchar)$ be the Whittaker model of $\pi$ attached to the additive character $\psi=\addchar_\pme:\Qq\to\C^\x$. Recall that $\cW(\pi,\addchar)$ is a subspace of smooth functions $W:\GL_2(F)\to\C$ such that
\begin{mylist}
\item $W(\pMX{1}{x}{0}{1}g)=\addchar(x)W(g)$ for all $x\in\Qq$.
\item If $\pme$ is the archimedean place, $W(\DII{a}{1})=O(\abs{a}^M)$ for some positive number $M$.
\end{mylist}
If $\Qq$ is non-archimedean and $N$ is a positive integer, we put
\[U_0(N)_\pme=\stt{g=\pMX{a}{b}{c}{d}\in\GL_2(\Z_\pme)\mid c\in N\Z_\pme}.\]
 Let $c$ be the conductor of $\pi$. Let $\newW_{\pi}$ be the \emph{normalized Whittaker new form} characterized by $\newW_{\pi}(1)=1$ and $\newW_{\pi}(gu)=\newW_{\pi}(g)$ for all $u\in U_0(c)_\pme$. If $\Qq=\R$ and $\pi$ is a discrete series of weight $k$, then the normalized Whittaker new form $\newW_{\pi}\in\cW(\pi,\psi)$ is defined by
\[\begin{aligned}\newW_{\pi}(z\pMX{a}{x}{0}{1}\pMX{\cos\theta}{\sin\theta}{-\sin\theta}{\cos\theta})&=a^\frac{k}{2}e^{-2\pi a}\bbI_{\R_+}(a)\cdot\sgn(z)^k\psi(x)e^{ik\theta}\\
&\quad(a,z\in\R^\x,\,x,\theta\in\R).\end{aligned}
\]
Here $\bbI_{\R_+}(a)$ denotes the characteristic function of the set of positive real numbers. Recall that the zeta integral $\Psi(s,W,\chi)$ for $W\in\cW(\pi,\psi)$ and a character $\chi:\Qq^\x\to\C^\x$ is defined by
\[\Psi(s,W,\chi)=\int_{\Qq^\x}W(\DII{a}{1})\chi(a)\abs{a}^{s-\onehalf}\dx a\quad(s\in\C).\]
Then $\Psi(s,W,\chi)$ converges absolutely for $\Re s\gg 0$ and has meromorphic continuation to the whole $s\in\C$. 

Let $\mathcal K(\pi,\psi)$ be the Kirillov model of $\pi$ with respective to $\psi$. Then $\mathcal K(\pi,\psi)$ is a subspace of smooth functions $\phi:\Qq^\x\to\C$, and there is an isomorphism $\cW(\pi,\psi)\isoto \mathcal K(\pi,\psi)$ given by \[W\mapsto \phi_W(a):=W(\DII{a}{1}).\] We call $\phi_{W_\pi}$ the normalized Kirillow new form. By the list of Kirillov new forms \cite[\S 2.4]{Schmidt:newform}, we can verify that
\beq\label{E:4.W}\Psi(s,W_\pi,\chi)=L(s,\pi\ot\chi)\text{ for unramified character }\chi.\eeq
\section{Gross points and modular forms on definite quaternion algebras}\label{S:Gross_modular_forms}
\subsection{}\label{SS:setup}Let $\cK$ be an imaginary quadratic field with the discriminant $-D_\cK<0$ and let $\Diff=\sqrt{-D_\cK}$. Write $z\mapsto \ol{z}$ for the complex conjugation on $\cK$. Define $\CMP\in\cK$ by
%\[\CMP=\frac{\delta}{2}\text{ if }4\mid D_\cK\,;\, \CMP=\frac{\delta^2-\delta}{2}\text{ if } 4\ndivides D_\cK.\]
\[\CMP=\frac{D'+\delta}{2},\,D'=\begin{cases}D_\cK&\text{ if }2\ndivides D_\cK,\\
D_\cK/2&\text{ if } 2\mid D_\cK.
\end{cases}\]
Then $\cO_\cK=\Z+\Z\cdot\bftheta$ and $\CMP\ol{\CMP}$ is a local uniformizer of primes that are ramified in $\cK$. Fix positive integers $N^+$ that are only divisible by prime split in $\cK$ and $N^-$ that are only divisible by primes inert or ramified in $\cK$.

We assume that
\[\text{ $N^-$ is the square-free product of an odd number of primes}.\]
Let $B$ be the definite quaternion over $\Q$ which is ramified precisely at the prime factors of $N^-$ and the archimedean place. We can regard $\cK$ as a subalgebra of $B$.
Write $\rmT$ and $\rmN$ for the reduced trace and norm of $B$ respectively. Let $G=B^\x$ be the algebraic group over $\Q$ and let $Z=\Q^\x$ be the center of $G$.
Fix a distinguished rational prime $p$ such that
\[p\ndivides N^+N^-.\]Let $\frakP$ be the prime of $\cK$ above $p$ induced by $\iota_p:\cK\hookto\Cp$. Let $\ell\ndivides N^-$ be a rational prime ($\ell$ can be $p$). We choose a basis of $B=\cK\oplus \cK\cdot\cmJ$ over $\cK$ such that
\begin{itemize}
\item $\cmJ^2=\beta\in\Q^\x$ with $\beta<0$ and $\cmJ t=\ol{t}\cmJ$ for all $t\in\cK$.
\item $\beta\in(\Z_\pme^\x)^2$ for all $\pme\mid p\ell N^+$ and $\beta\in\Z_\pme^\x$ for $\pme|D_\cK$.
\end{itemize}
Fix a square root $\sqrtb\in\Qbar$ of $\beta$. We fix an isomorphism $i=\prod i_\pme: \wh B^{(N^-)}\iso M_2(\A_f^{(N^-)})$ as follows. For each finite place $\pme |p\ell N^+$, the isomorphism $i_\pme:B_\pme\iso M_2(\Qq)$ is defined by
\beq\label{E:embedding.W}i_\pme(\bftheta)=\pMX{\rmT(\bftheta)}{-\rmN(\bftheta)}{1}{0};\quad i_\pme(\cmJ)=\sqrtb\cdot \pMX{-1}{\rmT(\bftheta)}{0}{1}\quad(\sqrtb\in\Z_\pme^\x).\eeq
For each finite place $\pme\ndivides p\ell N$, the isomorphism $i_\pme:B_\pme:=B\ot_\Q\Qq\iso M_2(\Qq)$ is chosen so that
 \beq\label{E:21.W}i_{\pme}(\cO_\cK\ot\Z_\pme)\subset M_2(\Z_\pme).\eeq
Hereafter, we shall identify $B_\pme$ and $G(\Qq)$ with $M_2(\Qq)$ and $\GL_2(\Qq)$ via $i_\pme$ for finite $\pme\ndivides N^-$. Finally, we define
\beq\label{E:2.W}\begin{aligned}i_\cK: B&\hookto M_2(\cK)\\
a+b\cmJ&\mapsto i_\cK(a+ b\cmJ):=\pMX{a}{b\beta}{\ol{b}}{\ol{a}}\quad(a,b\in\cK)\end{aligned}\eeq
and let $i_\C:B\to M_2(\C)$ be the composition $i_\C=\iota_\infty\circ i_\cK$
\subsection{Optimal embeddings and Gross points}
Fix a decomposition $N^+\cO_\cK=\frakN^+\ol{\frakN^+}$ once and for all. For each finite place $\pme\ndivides p$, we define $\cmptv\in G(\Qq)$ as follows:
\beq\label{E:cmptv.W}\begin{aligned}
\cmptv=&1\text{ if $\pme\ndivides pN^+$,}\\
\cmptv=&\Diff^{-1}\pMX{\CMP}{\ol{\CMP}}{1}{1}\in\GL_2(\cK_\w)=\GL_2(\Qq)\text{ if $\pme=\w\wbar$ is split with $\w|\frakN^+$.}
\end{aligned}\eeq
For $g\in B$, we put
\[\iota_{\cmptv}(g):=i_\pme^{-1}(\cmptv^{-1}i_\pme(g)\cmptv).\]
If $\pme|N^+$ and $t=(t_1,t_2)\in \cK_\pme:=\cK\ot_\Q\Q_\pme=\cK_\w\oplus\cK_{\wbar}$, then
\beq\label{E:cm2.W}\iota_{\cmptv}(t)=\DII{t_1}{t_2}.
\eeq
For each non-negative integer $n$, we choose $\cmpt_\frakp^\setn\in G(\Qp)$ as follows.
If $\frakp=\frakP\ol{\frakP}$ splits in $\cK$, we put
\begin{align}\label{E:op1.W}\cmpt_\frakp^\setn=&\pMX{\CMP}{-1}{1}{0}\DII{\p^n}{1}\in\GL_2(\cK_\frakP)=\GL_2(\Qp).
\intertext{If $\frakp$ is inert or ramified in $\cK$, then we put}
\label{E:op2.B}\cmpt_\frakp^\setn=&\pMX{0}{1}{-1}{0}\DII{\p^n}{1}.\end{align}
Define $x_n:\AK^\x\to G(\A)$ by
\beq\label{E:op3.W}x_n(a):=a\cdot \cmpt^\setn\quad(\cmpt^\setn:=\cmpt^\setn_\frakp\prod_{\pme\not=\frakp}\cmptv).\eeq
This collection $\stt{x_n(a)}_{a\in\AK^\x}$ of points is called \emph{Gross points} of conductor $\frakp^n$ associated to $\cK$.

Let $\cO_n=\Z+p^n\cO_\cK$ be the order of $\cK$ of conductor $p^n$. For each positive integer $M$ prime to $N^-$, we denote by $R_{M}$ the Eichler order of level $M$ with respect to the isomorphisms $\stt{i_\pme:B_\pme\iso M_2(\Qq)}_{\pme\ndivides N^-}$. It is not difficult to verify immediately that the inclusion map $K\hookto B$ is an optimal embedding of $\cO_n$ into the Eichler order $B\cap \cmpt^\setn\wh R_{M}(\cmpt^\setn)^{-1}$ if $\Ord_p(M)\leq n$. In other words,
\beq\label{E:EichlerOrder.W}(B\cap \cmpt^\setn\wh R_{M}(\cmpt^\setn)^{-1})\cap \cK=\cO_n. \eeq

\subsection{Modular forms on definite quaternion algebras}\label{SS:modular_forms}
Let $k\geq 2$ be an even integer. For a ring $A$, we denote by $L_k(A)=\Sym^{k-2}(A^2)$ the set of the set of homogeneous polynomials of degree $k-2$ with coefficients in $A$. We write
 \beq\label{E:basis.W}L_{k}(A)=\bigoplus_{-\frac{k}{2}<m<\frac{k}{2}} A\cdot\bfv_m\quad(\bfv_m:=X^{\frac{k-2}{2}-m}Y^{\frac{k-2}{2}+m}).\eeq
 We let $\rho_k:\GL_2(A)\to \Aut_AL_k(A)$ be the unitary representation defined by
 \[\rho_k(g)P(X,Y)=\det(g)^{-\frac{k-2}{2}}\cdot P((X,Y)g)\quad(P(X,Y)\in L_k(A)).\]
If $A$ is a $\Z_\setp$-algebra with $p>k-2$, we define a perfect pairing $\pairing_k:L_k(A)\x L_k(A)\to A$ by
\[\Lpair{\sum_{i} a_i \bfv_i}{\sum_{j} b_j \bfv_j}=\sum_{-k/2<m<k/2}a_{m}b_{-m}\cdot (-1)^{\frac{k-2}{2}+m}\frac{\Gamma(k/2+m)\Gamma(k/2-m)}{\Gamma(k-1)}.\]
This pairing is $\GL_2(A)$-equivariant, \ie For $P,P'\in L_k(A)$, we have
\[\Lpair{\rho_k(g)P}{\rho_k(g)P'}=\Lpair{P}{P'}.\]
Via the embedding $i_\C$ in \eqref{E:2.W}, we obtain a representation \[\rho_{k,\infty}:G(\R)=(B\ot_\Q\R)^\x\stackrel{i_\C}\longto \GL_2(\C)\to \Aut_\C L_\wt(\C).\]
Then $\C\cdot \bfv_m$ is the eigenspace on which $\rho_{k,\infty}(t)$ acts by $(\ol{t}/t)^m$ for $t\in(\cK\ot_\Q\R)^\x$.
If $A$ is a $\cK$-algebra and $U\subset G(\Af)$ is an open compact subgroup, we denote by $\MF_{\wt}(U,A)$ be the space of modular forms of weight $k$ defined over $A$, consisting of functions $f:G(\Af)\to L_k(A)$ such that
\[f(\al g u)=\rho_{k,\infty}(\al)f(g)\text{ for all }\al\in G(\Q),\,u\in U.\]
The right translation makes $\MF_{\wt}(A):=\dirlim_U\MF_{\wt}(U,A)$ an admissible $G(\Af)$-representation.

Let $\cA(G)$ be the space of automorphic forms on $G(\A)$. For $\bfv\in L_k(\C)$ and $f\in\MF_{\wt}(\C)$, we define a function $\Psi(\bfv\ot f):G(\Q)\bksl G(\A)\to\C$ by \beq\label{E:1.W}\Psi(\bfv\ot f)(g):=\Lpair{\rho_{k,\infty}(g_\infty)\bfv}{f(g_f)}.\eeq
Then the map $\bfv\ot f\mapsto \Psi(\bfv\ot f)$ gives rise to $G(\A)$-equivariant morphism $L_k(\C)\ot \MF_{\wt}(\C)\to \cA(G)$.
Let $\om$ be a unitary Hecke character of $\Q$. We let \[\MF_{\wt}(U,\om,\C)=\stt{f\in \MF_{\wt}(U,\C)\mid f(zg)=\om(z)f(g)\text{ for all }z\in Z(\A)}.\]
Let $\cA_{\wt}(U,\om,\C)$ be the space of automorphic forms on $G(\A)$ of weight $\wt$ and central character $\om$, consisting of functions $\Psi(f\ot\bfv):G(\A)\to\C$ for $f\in S_{\wt}(U,\om,\C)$ and $\bfv\in L_k(\C)$. Denote by $\bfone$ the trivial character. For each positive integer $M$, we put
\begin{align*}\MF_{\wt}(M,\C)=&\MF_{\wt}(\wh R_{M}^\x,\bfone,\C),\\
\cA_{\wt}(M,\C)=&\cA_{\wt}(\wh R_{M}^\x,\bfone,\C).
\end{align*}

\section{Special value formula}\label{S:specialvalue}
\subsection{Global setting}
Let $\pi$ be an unitary irreducible cuspidal automorphic representation on $\GL_2(\A)$ with trivial central character. Henceforth, we make the following assumptions:
\begin{itemize}
\item The archimedean constituent $\pi_\infty$ is a discrete series of weight $k$;
\item The conductor of $\pi$ is $N=p^{n_p}N^+N^-$;
\item $\Ord_p(N)=n_p\leq 1\iff p^2\ndivides N$.
\end{itemize}
Let $\pi'=\ot\pi'_\pme$ be the unitary irreducible cuspidal automorphic representation on $G(\A)$ with trivial central character attached to $\pi$ via Jacquet-Langlands correspondence. Then we have%Let $\cA(\pi')$ be the automorphic realizations of $\pi'$ in $\cA(G)$.
\begin{mylist}
%\item The conductor of $\pi$ is $N$, $p\ndivides D_\cK$
\item The archimedean constituent $\pi'_\infty\iso (\rho_{k,\infty}, L_{\wt}(\C))$ as $G(\R)$-modules, and $\pi'_\pme$ is a unramified one dimensional representation for $\pme\mid N^-$.
\item The local constituent $\pi'_p=\pi_p$ is either an unramified principal series $\pi(\mu_\frakp,\nu_\frakp)$ or an unramified special representation $\sg(\mu_\frakp,\nu_\frakp)$ with $\mu_\frakp\nu_\frakp^{-1}=\Abs_{\Qp}$.
\end{mylist}

\subsection{$p$-stabilization of new forms}\label{SS:pstabilization} Let $\pi'_f$ denote the finite constituent of $\pi'$. Let
\[N_B=p^{n_p}N^+=N/N^-\]
and let $R:=R_{N_B}$ be the Eichler order of level $N_B$. The multiplicity one theorem together with our assumptions in particular imply that $\pi'_f$ can be realized as a \emph{unique} $G(\Af)$-submodule $\MF_\wt(\pi'_f)$ of $\MF_{\wt}(\C)$ and $\MF_\wt(N_B,\C)[\pi'_f]:=\MF_\wt(\pi'_f)\cap \MF_{\wt}(N_B,\C)$ is one dimensional. We fix a nonzero new form $\vformB\in\MF_{\wt}(N_B,\C)[\pi'_f]$. Throughout this section, we fix an integer $m$ such that
\[-k/2<m< k/2.\]
Define the automorphic form $\mForm\in\cA_{\wt}(N_B,\C)$ by
\beq\label{E:defn.W}\mForm:=\Psi(\bfv_m^*\ot\vformB)\quad(\bfv^*_m=\sqrtb^{-m}D_\cK^\frac{k-2}{2}\cdot\bfv_m).\eeq
We shall simply write $\Form$ for $\mForm$ for brevity. Set \[\al_\frakp:=\mu_\frakp(p)\abs{p}_p^{-\onehalf}.\] Define the \emph{$p$-stabilization} $\vformB^\dagger$ with respect to $\al_\frakp$ as follows: If $p\mid N$, let $\vformB^\dagger=\vformB$, and if $p\ndivides N$, let
\[\vformB^\dagger=\vformB-\frac{1}{\al_\frakp}\cdot \pi'(\DII{1}{\p})\vformB.\]
For $f\in\MF_{\wt}(N_B,\C)$, recall that the $U_p$-operator on $f$ is defined by
\[ f\mid U_p(g)=\sum_{x\in\Z/p\Z} f(g\pMX{p}{x}{0}{1}).\]
Thus $\vformB^\dagger$ is an $U_p$-eigenform with the eigenvalue $\al_\frakp$. Let $\Form^\dagger$ be the $p$-stabilization of $\Form$ given by \beq\label{E:defnP.W}\Form^\dagger:=\Psi(\bfv_m^*\ot \vformB^\dagger).\eeq
By definition, one can verify that
\beq\label{E:8.W}\begin{aligned}\Form^\dagger(x_n(\gamma au))=&\Form^\dagger(x_n(a_f))(\ol{a}_\infty/a_\infty)^{m}\\
(\gamma\in\cK^\x,&a=(a_\infty,a_f)\in\C^\x\x\hatK,\,u\in \wh\cO_n^\x).\end{aligned}\eeq
\subsection{The Petersson inner product of new forms on $\GL_2(\A)$}
For each place of $\Q$, recall that $\newW_{\pi_\pme}$ is the Whittaker new form normalized so that $\newW_{\pi_\pme}(1)=1$. Let $\newForm$ be the normalized new form in $\pi$. In other words,
\[\newForm(g):=\sum_{\al\in\Q}\newW_{\pi}(\DII{\al}{1}g)\quad(\newW_{\pi}=\prod_\pme \newW_{\pi_\pme}).\]
Let $\ALIG=\prod_\pme\ALIG_\pme\in\GL_2(\A)$ be the \emph{Atkin-Lehner} element defined by $\ALIG_\infty=\DII{1}{-1}$ and $\ALIG_\pme=\pMX{0}{1}{-N}{0}$ if $\pme\not =\infty$. Let $d^tg$ be the Tamagawa measure on $\GL_2$. We put
\[\PetG{\newForm}{\newForm}:=\int\limits_{\A^\x\GL_2(\Q)\bksl \GL_2(\A)}\newForm(g)\newForm(g\ALIG)d^tg.\]
 To give a formula of $\PetG{\newForm}{\newForm}$, we define the $\GL_2(\Qq)$-equivariant pairing $\bfb_\pme:\cW(\pi_\pme,\psi_\pme)\x\cW(\pi_\pme,\psi_\pme)\to\C$ by
\beq\label{E:7.W}\bfb_\pme(W_1,W_2):=\int_{\Qq^\x}W_1(\DII{a}{1})W_2(\DII{-a}{1})\dx a.\eeq
The convergence of this integral follows from the fact that $\pi_\pme$ is the local constituent of a unitary cuspidal automorphic representation. Let $\norm{\newForm}_\pme$ be the \emph{local norm} of $\newForm$ at $\pme$ defined by
\beq\label{E:lcoalnorm.W}\norm{\newForm}_\pme:=\frac{\zeta_{\Qq}(2)}{\zeta_{\Qq}(1)L(1,\Ad\pi_\pme)}\cdot\bfb_\pme(\newW_{\pi_\pme},\pi(\ALIG_\pme)\newW_{\pi_\pme}).\eeq
It is not difficult to deduce from \cite[Proposition\,5]{Waldsupurger:Central_value} that
\beq\label{E:Waldinner.W}\PetG{\newForm}{\newForm}=\frac{2L(1,\Ad\pi)}{\zeta_\Q(2)}\cdot\prod_{\pme}\norm{\newForm}_\pme.\eeq
Define the local root number $\ep(\pi_\pme)$ by
\[\ep(\pi_\pme):=\ep(\onehalf,\pi_\pme,\psi_\pme)\in\stt{\pm 1}.\]
The following lemma is well-known.
\begin{lm}\label{L:4.W}Let $\pme$ be a finite place. We have $\pi(\ALIG_\pme)\newW_{\pi_\pme}=\ep(\pi_\pme)\cdot \newW_{\pi_\pme}$.
\end{lm}
\begin{proof}We suppress the subscript $\pme$ and write $\pi=\pi_\pme$ and $\psi=\psi_\pme$ for brevity. Since $\pi(\ALIG_\pme)W_\pi $ is also a nonzero new vector of $\pi$, we find that $\pi(\ALIG_\pme)W_\pi=C\cdot W_{\pi}$ for some constant $C\in\C^\x$. Recall that we have
the local functional equation \cite[Theorem\,2.18 (iv)]{Jacquet_Langlands:GLtwo}:
%\beq\label{E:functionaleq}
\[\frac{\Psi(1-s,\wh W,\bfone)}{L(1-s,\pi)}=\ep(s,\pi,\addchar)\cdot\frac{\Psi(s,W,\bfone)}{L(s,\pi)},\]%\eeq
where $\wh W(g):=W(g\pMX{0}{1}{-1}{0})\in\cW(\pi,\psi)$. To evaluate $C$, we compute the zeta integral:
\begin{align*}
C\cdot \Psi(s,W_{\pi},\bfone)&=\Psi(s,\pi(\ALIG_\pme)W_\pi,\bfone)\\
=&\int_{F^\x}\pi(\ALIG_\pme)W_\pi(\DII{a}{1})\abs{a}^{s-\onehalf}\dx a\\
=&\abs{N}^{s-\onehalf}\cdot \int_{F^\x}\pi(\pMX{0}{1}{-1}{0})W_\pi(\DII{a}{1})\abs{a}^{s-\onehalf}\dx a\\
=&\abs{N}^{s-\onehalf}\cdot\Psi(s,\wh W_\pi,\bfone)\\
=&\abs{N}^{s-\onehalf}\cdot\ep(1-s,\pi,\psi)\frac{L(s,\pi)}{L(1-s,\pi)}\cdot\Psi(1-s,W_\pi,\bfone).
\end{align*}
It follows from \eqref{E:4.W} that
\[C=\abs{N}^{s-\onehalf}\cdot\ep(1-s,\pi,\psi)=\ep(\onehalf,\pi,\psi).\]
This completes the proof.
 \end{proof}
\begin{lm}\label{L:3.W} We have $\norm{\newForm}_\pme=1$ for finite $\pme\ndivides N$ and  $\norm{\newForm}_\infty=2^{-k-1}$. If $\pme|N^-$, then \[\norm{\newForm}_\pme=\ep(\pi_\pme)\cdot (1+\abs{\pme})^{-1}.\]
\end{lm}
\begin{proof}The assertions for $\pme\ndivides N$ and $\pme=\infty$ are straightforward. Suppose that $\pme|N^-$. Then $\pi=\pi_\pme$ is a unramified special representation, and $L(1,\Ad\pi)=\zeta_{\Qq}(2)$. By \lmref{L:4.W}, we have
\begin{align*}
\bfb_\pme(\newW_\pi,\pi(\ALIG_\pme)\newW_\pi)&=\int_{F^\x}W_\pi(\DII{a}{1})\cdot \pi(\ALIG_\pme)W_\pi(\DII{-a}{1})\dx a\\
&=\ep(\onehalf,\pi,\psi)\cdot\int_{F^\x}W_\pi(\DII{a}{1})W_{\pi}(\DII{a}{1})\dx a\\
&=\ep(\onehalf,\pi,\psi)L(1,\Ad\pi).\qedhere
\end{align*}
%It follows that
%\[\norm{\newForm}_\pme=\bfb_\pme(\newW_\pi,\pi(\ALIG_\pme)\newW_\pi)\]
\end{proof}

The Petersson inner product $\norm{\newForm}_{\Gamma_0(N)}$ of $\newForm$ is defined by
\[\norm{\newForm}_{\Gamma_0(N)}:=\vol(U_0(N),d^tg)^{-1}\cdot \int\limits_{\A^\x\GL_2(\Q)\bksl \GL_2(\A)}\abs{\newForm(g)}^2d^tg,\]
where $U_0(N)={\mathrm O}(2,\R)\x\prod_{\pme<\infty} U_0(N)_\pme$.
Note that \[\vol(U_0(N),d^tg)^{-1}=%\vol(\GL_2(\Zp),d^tg)^{-1}\cdot N\prod_{\pme|N}(1+\pme^{-1}) =
\zeta_\Q(2)N\prod_{\pme|N}(1+\pme^{-1}).\]
We have the following proposition:
\begin{prop}[Theorem\,5.1 \cite{Hida:congruence_number}]\label{P:1.W}We have
\[L(1,\Ad\pi)= \norm{\newForm}_{\Gamma_0(N)}\cdot 2^{k}N^{-1}\cdot \prod_{\pme|N_B}\frac{\ep(\pi_\pme)}{(1+\pme^{-1})\norm{\newForm}_\pme}.\]
\end{prop}
\begin{proof}By \eqref{E:Waldinner.W} and \lmref{L:3.W}, we have\[\PetG{\newForm}{\newForm}=\frac{2L(1,\Ad\pi)}{\zeta_\Q(2)}\cdot 2^{-k-1}\prod_{\pme|N^-}\frac{\ep(\pi_\pme)}{1+\pme^{-1}}\cdot \prod_{\pme|N_B}\norm{\newForm}_\pme.\]
On the other hand, it is well known that
\[\newForm(g\ALIG)=\ep(\pi_f)\cdot \ol{\newForm(g)}\quad(\ep(\pi_f):=\prod_{\pme<\infty}\ep(\pi_\pme)),\]
and hence
\begin{align*}\PetG{\newForm}{\newForm}=&\ep(\pi_f)\cdot\vol(U_0(N),d^tg)\cdot \norm{\newForm}_{\Gamma_0(N)}\\
=&\norm{\newForm}_{\Gamma_0(N)}\cdot \frac{1}{N\zeta_\Q(2)}\cdot\prod_{\pme|N}\frac{\ep(\pi_\pme)}{1+\pme^{-1}}.\end{align*}
Combining these formulae, we find that
\[L(1,\Ad\pi)=2^{k}N^{-1}\norm{\newForm}_{\Gamma_0(N)}\cdot \prod_{\pme|N_B}\frac{\ep(\pi_\pme)}{(1+\pme^{-1})\norm{\newForm}_\pme}.\qedhere\]
\end{proof}

\subsection{Local toric integrals}\label{SS:localtoric} For each place $\pme$ of $\Q$,  denote by $\pi'_\pme$ the local constituent of $\pi'$ at $\pme$.
\begin{defn}\label{D:1.W}Define the new vector $\lnew\in\pi'_\pme$ as follows:
\begin{itemize}
\item[(a)]if $\pme=\infty$, then $\lnew$ is a multiple of $\bfv_m\in L_\wt(\C)\iso \pi'_\infty$,\\
\item[(b)]if $\pme\mid N^-$, then $\lnew$ is a basis of the one dimensional representation $\pi'_\pme$ of $G(\Qq)$,\\
\item[(c)]if $\pme\ndivides N^-$, then $\lnew$ is fixed by $(R\ot_\Z\Z_\pme)^\x\iso U_0(N)_\pme$.
\end{itemize}
Let $\lnew^\dagger=\lnew$ if either $\pme\not=p$ or $\pme=p\mid N$ and let
\[\lnew^\dagger=\lnew-\frac{1}{\al_\frakp}\cdot\pi(\DII{1}{\p})\lnew\text{ if }\pme=\frakp\ndivides N.\]
\end{defn}
Define the local \emph{Atkin-Lehner} element $\ALI_\pme\in G(\Q_\pme)$ as follows: $\ALI_\pme=\cmJ$ for $\pme|\infty N^-$, $\ALI_\pme=1$ for finite place $\pme\ndivides N$ and $\ALI_\pme=\pMX{0}{1}{-N_B}{0}$ if $\pme|N_B$. Let $\ALI:=\prod\ALI_\pme\in G(\A)$. Since $\pi'$ has trivial central character, $\pi'_\pme$ is self-dual. Hence, there exists a non-degenerate $G(\Q_\pme)$-equivariant pairing $\pairing_\pme\colon\pi'_\pme\x\pi'_\pme\to\C$. This pairing is unique up to a nonzero scalar.

For $g\in G(\Qq)$ and a character $\chi:\cK_\pme^\x\to\C^\x$, we define the local toric integral for the new vector $\lnew$ by
\beq\label{E:defn_period.W}\cP(g,\lnew,\chi)=\frac{L(1,\Ad\pi_\pme)L(1,\tau_{\cK_\pme/\Qq})}{\zeta_{\Qq}(2)L(\onehalf,\pi_{\cK_\pme}\ot\chi)}\cdot
\int_{\cK_\pme^\x/\Qq^\x}\frac{\pair{\pi'(tg)\lnew^\dagger}{\pi'(\cmJ g)\lnew^\dagger}_\pme}{\pair{\lnew}{\pi'(\ALI_\pme)\lnew}_\pme}\cdot\chi(t_\pme)dt_\pme,\eeq
where $\tau_{K_q/\Q_q}$ denotes the quadratic character of $\Q_q^\x$ associated to $K_q/\Q_q$ and $dt_q$ is the quotient measure of the Haar measures of $K_q^\x$ and $\Q_q^\x$ fixed in \subsecref{S:notation}. An important observation is that the number $\cP(g,\lnew,\chi)$ does not depend on the choice of the pairing $\pairing_\pme$, depending only on $\chi$ and the line spanned by $\lnew$.

\subsection{Waldspurger's formula}
Let $\chi:\cK^\x\bksl \AK^\x\to\C^\x$ be an anticyclotomic Hecke character of archimedean weight $(m,-m)$. Namely, \beqcd{crit}\chi|_{\A^\x}=\bfone\text{ and }\chi_\infty(z)=\left(\frac{z}{\ol{z}}\right)^m\quad(-k/2<m<k/2).\eeqcd
Let $\pi_\cK$ be the automorphic representation on $\GL_2(\AK)$ via the quadratic base change of $\pi$ and let $L(s,\pi_\cK\ot\chi)$ be the automorphic $L$-function of $\pi_\cK\ot\chi$, which satisfies the functional equation
\[L(s,\pi_\cK\ot\chi)=\ep(s,\pi_\cK\ot\chi)L(1-s,\pi_\cK^\vee\ot\chi^{-1}),\]
where $\ep(s,\pi_\cK\ot\chi)=\prod_\pme\ep(s,\pi_{\cK_\pme}\ot\chi_\pme,\psi_{\cK_\pme})$ is the product of local epsilon factors. Let $\cA(\pi')$ be the automorphic realization of $\pi'$ in $\cA(G)$. For $\vp\in\cA(\pi')$ and $g\in G(\A)$, define the global toric period integral by
\[P(g,\vp,\chi):=\int\limits_{\cK^\x\A^\x\bksl \AK^\x}\vp(tg)\chi(t)dt,\]
where $dt$ is the measure of $K^\x/\Q^\x$ with the volume $\vol(K^\x\A^\x\bksl \AK^\x,dt)=2L(1,\tau_{K/\Q})$, where $L(s,\tau_{K/\Q})$ is the complete $L$-function of the quadratic character attached to $K/\Q$ (\cf \cite[p.\,180]{Waldsupurger:Central_value}).  For $\vp_1,\vp_2\in\cA(\pi')$, we define the $G(\A)$-equivariant pairing:
\[\pair{\vp_1}{\vp_2}_{G}=\int_{G(\Q)Z(\A)\bksl G(\A)}\vp_1(g)\vp_2(g)dg,\]
where $dg$ is the Tamagawa measure on $G/Z$. By the theory of new forms \cite{Casselman:Atkin-Lehner}, $\Form$ is characterized uniquely up to a scalar by the equations $\pi'(\wh R^\x)\Form=\Form$ and $\pi_\infty'(t)\Form=(\ol{t}/t)^m\Form$ for $t\in\cK^\x$, so we have $\pi'(\ALI)\Form(g)=C\cdot \ol{\Form(g)}$ for some constant $C\in\C^\x$. This in particular implies that \[\pair{\Form}{\pi'(\ALI)\Form}_G\not =0.\]
Since $\pairing_G$ is a nonzero multiple of the product $\ot_\pme\pairing_\pme$, we have \beq\label{E:5.W}\pair{\lnew}{\pi'(\ALI_\pme)\lnew}_\pme\not =0\text{ for each place $\pme$}.\eeq

 We shall make use of the following version of Waldspurger's formula, which expresses the global toric period integral as a product of local toric integrals.
\begin{prop}\label{P:Waldformula.W}We have
\[\frac{P(\cmpt^\setn,\Form^\dagger,\chi)^2}{\pair{\Form}{\pi'(\ALI)\Form}_G}=\frac{\zeta_\Q(2)}{2L(1,\Ad\pi)}\cdot L(\onehalf,\pi_\cK\ot\chi)\cdot\prod_{\pme }\cP(\cmptv^\setn,\lnew,\chi_\pme),\]
where $\pme$ runs over all places of $\Q$.
\end{prop}
\begin{proof}Fix an isomorphism $i:\pi'\iso\ot_\pme\pi'_\pme$ such that $i(\Form)=\ot_\pme\vp_\pme$ for $\vp_\pme$ chosen in \defref{D:1.W}. Set $\vp_1=\pi'(\cmpt^\setn)\Form^\dagger$, $\vp_2=\pi'(\cmJ\cmpt^\setn)\Form^\dagger$, $\vp_3=\Form$ and $\vp_4=\pi'(\ALI)\Form$. Let $i(\vp_i)=\ot_\pme\vp_{i,\pme}$. Let $\ol{\chi}$ be the character defined by $\ol{\chi}(t)=\chi(\ol{t})$. Note that $P(1,\vp_1,\chi)=P(\cmpt^\setn,\Form^\dagger,\chi)$ and $P(1,\vp_2,\ol{\chi})=P(\cmpt^\setn,\Form^\dagger,\chi)$. It follows from Waldspurger's formulae \cite[Proposition\,4,\,Proposition\,5,\,Lemme 7]{Waldsupurger:Central_value} that if $\pair{\vp_3}{\vp_4}_G\not =0$, then
\begin{align*}\frac{P(1,\vp_1,\chi)P(1,\vp_2,\ol{\chi})}{\pair{\vp_3}{\vp_4}_G}
=&\frac{\zeta_\Q(2)L(\onehalf,\pi_\cK\ot\chi)}{2L(1,\Ad\pi)}\cdot \prod_{\pme}P_\pme(\vp_1,\vp_2,\vp_3,\vp_4,\chi),\\
\intertext{where}
P_\pme(\vp_1,\vp_2,\vp_3,\vp_4,\chi)=&\frac{L(1,\Ad\pi_\pme)L(1,\tau_{\cK_\pme/\Qq})}{L(\onehalf,\pi_{\cK_\pme}\ot\chi_\pme)\zeta_{\Qq}(2)}\cdot
\int_{\cK_\pme^\x/\Qq^\x}\frac{\pair{\pi'(t_\pme)\vp_{1,\pme}}{\vp_{2,\pme}}_\pme}{\pair{\vp_{3,\pme}}{\vp_{4,\pme}}_\pme}\cdot\chi(t_\pme)dt_\pme.
\end{align*}
Moreover, $P_\pme(\vp_1,\vp_2,\vp_3,\vp_4,\chi)=1$ for all but finitely many $\pme$. The proposition follows immediately.
%We fix $f=\ot f_v\in \MF=\ot'\MF_v$ so that $j_v(f_v)=e_v\ot e_v\in \pi_v\ot \pi_v$ for all $v$ and hence $j(f)=e\ot e$. Recall that $e^\dagger=e-\al_\frakp^{-1}\pi(\DII{1}{\p})e$
\end{proof}
Let $Cl(\Eichler)$ be a set of representatives of $B^\x\bksl \hatB/\hatR^\x\hatQ$ in $\hatB=G(\Af)$. Define the inner product of $\vformB$ by
\beq\label{E:normvf.W}\PetB{\vformB}{\vformB}:=\sum_{g\in Cl(\Eichler)}\frac{1}{\#\Gamma_{g}}\cdot \Lpair{\vformB(g)}{\vformB(g\ALI)}\quad (\Gamma_g:=(B^\x\cap  g\hatR^\x g^{-1}\hatQ)/\Q^\x).\eeq Let $\ep_\pme(\pi_{\cK},\chi):=\ep(\onehalf,\pi_{\cK_\pme}\ot\chi_\pme,\psi_{\cK_\pme})\in\stt{\pm 1}$ be the local root number of $\pi_{\cK_\pme}\ot\chi_\pme$.

\begin{cor}\label{C:1.W}Suppose that $\chi$ is unramified outside $p$. Then we have
\[\begin{aligned}P(\cmpt^\setn,\Form^\dagger,\chi)^2\cdot \frac{\norm{\newForm}_{\Gamma_0(N)}}{\PetB{\vformB}{\vformB}}=&\frac{2^{2-k}(-1)^mD_\cK^{k-2}}{\sqrt{D_\cK}}\cdot L(\onehalf,\pi_\cK\ot\chi)\cdot\prod_{\pme|(D_\cK,N^-)}(1-\ep_\pme(\pi_\cK,\chi))\\
&\times\prod_{\pme|pN^+}\cP(\cmptv,\lnew,\chi_\pme)\cdot\frac{\norm{\newForm}_\pme}{\abs{D_\cK}^{\onehalf}_{\Q_\pme}\ep(\pi_\pme)}.\end{aligned}\]
\end{cor}
\begin{proof}
Let $v_R$ be the volume of $G(\R)\wh R^\x$ in $Z(\A)G(\Q)\bksl G(\A)$ and $dg_\infty$ be the Haar measure on the compact group $G(\R)/Z(\R)$ with the volume one. By Schur orthogonality relations, we have
\begin{align*}
\pair{\Form}{\pi'(\ALI)\Form}_G&=\sum_{g\in Cl(R)}\frac{1}{\#\Gamma_{g}}\cdot\int_{G(\R)/Z(\R)}\Lpair{\rho_{k,\infty}(g_\infty)\bfv^*_m}{\vformB(g)}\cdot\Lpair{\rho_{k,\infty}(g_\infty\cmJ)\bfv^*_m}{\vformB(g\ALI_f)} dg_\infty\cdot v_R\\
&=\sum_{g\in Cl(R)}\frac{1}{\#\Gamma_{g}}\cdot\Lpair{\vformB(g)}{\vform(g\ALI)}\cdot\frac{\Lpair{\bfv^*_m}{\rho_{k,\infty}(\cmJ)\bfv^*_m}}{\dim_\C L_{k-2}(\C)}\cdot v_R\\
&=\PetB{\vformB}{\vformB}\cdot\frac{(-1)^\frac{2-k}{2} D_\cK^{k-2}\Lpair{\bfv_m}{\bfv_{k-m-2}}}{(k-1)}\cdot v_R,
\end{align*}
By the Eichler mass formula:
\[\sum_{g\in Cl(R)}\frac{1}{\#\Gamma_g}=\frac{\zeta_\Q(2)}{4\pi}\cdot N\prod_{\pme|N^-}(1-\pme^{-1})\prod_{\pme|N_B}(1+\pme^{-1}),\]
and $\vol(Z(\A)G(\Q)\bksl G(\A),dg)=2$, we find that
\[v_R=\frac{8\pi}{\zeta_\Q(2)N}\cdot\prod_{\pme|N^-}\zeta_{\Qq}(1)\prod_{\pme|N_B}(1+\pme^{-1})^{-1},\]
and hence
\beq\label{E:3.W}\pair{\Form}{\pi'(\ALI)\Form}_G=
\PetB{\vformB}{\vformB}\cdot\frac{(-1)^mD_\cK^{k-2}\Gamma(k/2+m)\Gamma(k/2-m)}{\Gamma(k)}\cdot\frac{8\pi}{\zeta_\Q(2)N}\cdot\prod_{\pme|N^-}\zeta_{\Qq}(1)\prod_{\pme|N_B}(1+\pme^{-1})^{-1}.\eeq
Combining with \propref{P:1.W}, \propref{P:Waldformula.W} and \eqref{E:3.W} , we find that
\beq\label{E:midformula.W}\begin{aligned}P(\cmpt^\setn,\Form^\dagger,\chi)^2\cdot \frac{\norm{\newForm}_{\Gamma_0(N)}}{\PetB{\vformB}{\vformB}}
= & \frac{2^{2-k}D_\cK^{k-2}\pi (-1)^m\Gamma(k/2+m)\Gamma(k/2-m)}{\Gamma(k)}\cdot  L(\onehalf,\pi_\cK\ot\chi)\\
&\times\prod_{\pme\in S}\cP(\cmptv^\setn,\lnew,\chi_\pme)\cdot\prod_{\pme|N^-}\zeta_{\Qq}(1)\cdot\prod_{\pme|pN^+}\frac{\norm{\newForm}_\pme}{\ep(\pi_\pme)}.
\end{aligned}\eeq

We proceed to compute the local toric integrals $\cP(\cmptv^\setn,\lnew,\chi_\pme)$ for $\pme\ndivides pN^+$.
At the archimedean place, $\pi'_\infty=\rho_{k,\infty}$ and $\chi_\infty(t)=(t/\ol{t})^{m}$. In addition, $\ALI_\infty=\cmJ$ and $\vp_\infty=\bfv_m$ is characterized by $\rho_{k,\infty}(t)(\bfv_m)=\chi_\infty(t)^{-1}\bfv_m$ for $t\in\C^\x$, so we have
\[\frac{\Lpair{\rho_{k,\infty}(t)\bfv_m}{\rho_{k,\infty}(\cmJ)\bfv_m}}{\Lpair{\bfv_m}{\rho_{k,\infty}(\ALI_\infty)\bfv_m}}\cdot\chi_\infty(t)=1\text{ for all }t\in\C^\x=\cK^\x_\infty.\]
Recall that
\[L(1,\Ad\pi_\infty)=2^{1-k}\pi^{-(k+1)}\Gamma(k);\,L(\onehalf,\pi_{\cK}\ot\chi)=\Gamma_\C(k/2+m)\Gamma_\C(k/2-m).\]
We conclude that
\beq\label{E:period1.W}\cP(1,\bfv_m,\chi_\infty)=\frac{\Gamma(k)}{2\pi\Gamma(k/2+m)\Gamma(k/2-m)}\cdot\vol(\C^\x/\R^\x,dt_\infty)=\frac{\Gamma(k)}{\pi\Gamma(k/2+m)\Gamma(k/2-m)}.\eeq
If $\pme\ndivides pN$, then $\lnew$ is the spherical vector in $\pi_\pme$, and by \cite[Lemme 14]{Waldsupurger:Central_value} (use the fact $i_\pme(\cmJ)\in i_\pme(\cK_\pme^\x)\GL_2(\Z_\pme)$)
\beq\label{E:period2.W}\cP(\cmptv,\lnew,\chi_\pme)=\abs{D_\cK}_\pme^\onehalf. \eeq
 Suppose that $\pme\mid N^-$. Then $\pi'_\pme=\xi\circ\rmN$ is the one dimensional representation of a unramified quadratic character $\xi:\Qq^\x\to\C^\x$, and $\pi_\pme=\sg(\xi\Abs^\onehalf,\xi\Abs^{-\onehalf})$, and a simple calculation of local root numbers shows that
 \[(\chi_\pme\cdot \xi\circ\rmN)(\uf_{\cK_\pme})+\ep_\pme(\pi_{\cK},\chi)=0.\]
In particular, if $\cK_\pme/\Q_\pme$ is unramified, then $\ep_\pme(\pi_\cK,\chi)=-1$. Therefore, we find that \beq\label{E:period3.W}\begin{aligned}\cP(1,\lnew,\chi_\pme)
%=&(1-\abs{\pme})\cdot(1+(\chi_\pme\cdot\xi\circ\rmN)(\uf_{\cK_\pme}))L(1,\tau_{K_q/\Q_q}) \vol(\cO_{K_q}^\x/\Z_q^\x,dt)\\
=&\zeta_{\Q_\pme}(1)^{-1}\abs{D_\cK}_\pme^\onehalf\cdot \begin{cases}1&\cdots \cK_\pme/\Q_\pme\text{ is unramified},\\
(1-\ep_\pme(\pi_{\cK},\chi))&\cdots \cK_\pme/\Q_\pme\text{ is ramified }.
\end{cases}\end{aligned}\eeq
Combining \eqref{E:period1.W}, \eqref{E:period2.W}, \eqref{E:period3.W} and \eqref{E:midformula.W}, we obtain the desired formula.
\end{proof}
\subsection{Local toric integrals at $\pme\mid pN^+$}
 In this subsection, we carry out the computation of the local toric integral $\cP(\cmptv^\setn,\lnew,\chi_\pme)$ using the Whittaker model for $\pme\mid pN^+$. Let $F=\Qq$ and $E=\cK_\pme$ and write $\pi=\pi_\pme'\iso\pi_\pme$, $\chi=\chi_\pme$ and $\psi=\psi_\pme$ for brevity. Define the toric integral for Whittaker functions $W\in\cW(\pi,\psi)$ by
\[\localP{W}{\chi}:=\frac{\abs{D_\cK}^{-\onehalf}L(1,\tau_{E/F})}{\zeta_F(1)}\cdot \int_{E^\x/F^\x}\bfb_\pme(\pi(t_q)W,\pi(\cmJ)W)\cdot\chi(t_q)dt_q.\]
\begin{lm}\label{L:10.W}Let $\cmJ_{\pme}^\setn:=(\cmptv^\setn)^{-1}\cmJ\cmptv^\setn$. For $\pme|N^+$, we have
\[\pi(\cmJ_\pme^\setn)\vp_\pme=\ep(\pi)\pi(\DII{N^+}{1})\vp_\pme.\]
\end{lm}
\begin{proof}A straightforward computation shows that
\[\cmJ_\pme^\setn=\sqrtb\cdot\pMX{0}{1}{1}{0}\text{ if }\pme|N^+.\]
Thus, by \lmref{L:4.W} we have $\pi(\cmJ_\pme^\setn)\vp_\pme=\pi(\DII{N^+}{1})\pi(\ALIG)\vp_\pme=\ep(\pi)\pi(\DII{N^+}{1})\vp_\pme$.
\end{proof}
\begin{prop}\label{P:2.W}Let $\pme\mid N^+$. Write $\pme=\w\wbar$ in $\cK$ with $\w\mid \frakN^+$. If $\chi$ is unramified, then we have
\[\cP(\cmptv,\lnew,\chi)\cdot\frac{\norm{\newForm}_\pme}{\abs{D_\cK}^\onehalf\ep(\pi)}=\chi_\pme(\frakN^+).\]
\end{prop}
\begin{proof}Since $\pme\ndivides p$ is split in $E$, we have $L(1,\tau_{E/F})=\zeta_F(1)$, $\abs{D_\cK}=1$ and $\cmptv^\setn=\cmptv$. By definition \eqref{E:lcoalnorm.W},
\[\cP(\cmptv^\setn,\lnew,\chi)=\frac{1}{\norm{\newForm}_\pme}\cdot \frac{1}{L(\onehalf,\pi_E\ot\chi)}\cdot\localP{\pi(\cmptv)\newW_{\pi}}{\chi}.\]
Write $\chi=(\chi_\w,\chi_{\wbar})$. A straightforward computation shows that \begin{align*}
\localP{\pi(\cmptv)\newW_\pi}{\chi}&=\ep(\pi)\int_{F^\x}\int_{F^\x}\newW_\pi(\DII{at_1}{1})\newW_\pi(\DII{-aN^+}{1})\chi_\frakq(t_1)\dx a \dx t_1\quad(\text{by \lmref{L:10.W}})\\
&=\ep(\pi)\chi_\w(N^+)\int_{F^\x}\int_{F^\x}\newW_{\pi}(\DII{t_1}{1}\newW_{\pi}(\DII{a}{1})\chi_\frakq(t_1)\chi_\frakq^{-1}(a)\dx a \dx t_1\\
&=\chi_\w(N^+)\ep(\pi)\Psi(\onehalf,\newW_\pi,\chi_\frakq)\Psi(\onehalf,\newW_\pi,\chi_\frakq^{-1})\\
&= \chi_\w(N^+)\ep(\pi)\cdot L(\onehalf,\pi_E\ot\chi).
\end{align*}
The last equality follows from \eqref{E:4.W}.
\end{proof}

We continue to compute the local toric integral at the place $\pme=p$. Recall that we assume $\pi=\pi_\frakp$ is a unramified principal series $\pi(\mu_\frakp,\nu_\frakp)$ or a special representation $\sg(\mu_\frakp,\nu_\frakp)$ with unramified character $\mu_\frakp$ and $\mu_p\nu_p^{-1}=\Abs$. Let $\Iwahori$ be the Iwahori subgroup given by
\[\Iwahori=\stt{g=\pMX{a}{b}{c}{d}\in \GL_2(\Zp)\mid c\in p\Zp}.\]
By the complete description of Kirillov models in \cite{Jacquet:GLtwoPartII}, the function $\mu_\frakp\Abs^\onehalf(a)\bbI_{\OFp}(a)$ lies in the Kirillov model $\cK(\pi,\psi)$, and hence there exists a unique Whittaker function $\localW{\frakp}^\dagger\in\cW(\pi,\addchar)$ 
\[\localW{\frakp}^\dagger(\DII{a}{1})=\mu_\frakp\Abs^\onehalf(a)\bbI_{\OFp}(a).\]
One can verify that $\localW{\frakp}^\dagger$ is invariant by $\Iwahori$ and is an $U_\frakp$-eigenfunction with eigenvalue $\al_\frakp=\mu_\frakp(\p)\abs{\p}^{-\onehalf}$.
If $\pi$ is unramified, then
\begin{align*}\newW_\pi(\DII{a}{1})=&\abs{a}^\onehalf\bbI_{\Zp}(a)\frac{\mu_\frakp(\p a)-\nu_\frakp(\p a)}{\mu_\frakp(\p)-\nu_\frakp(\p)};\\
\localW{\frakp}^\dagger(g)=&\newW_{\pi}(g)-\frac{1}{\al_\frakp}\cdot \newW_{\pi}(g\DII{1}{\p}).\end{align*}
If $\pi$ is special, then it is well known that \beq\label{E:special.W}\newW_\pi=\newW^\dagger_p\text{ and }\pi(\pMX{0}{1}{-\p}{0})\newW_\pi=-\al_\frakp\newW_\pi.\eeq
We have
\beq\label{E:midformulaII.W}\cP(\cmptv^\setn,\lnew,\chi)=\frac{\abs{D_\cK}^\onehalf}{\norm{\newForm}_p}\cdot \frac{1}{L(\onehalf,\pi_E\ot\chi)}\cdot\localP{\pi(\cmptv^\setn)\localW{\frakp}^\dagger}{\chi}.\eeq

\begin{defn}\label{D:Euler.W}Define the \padic multiplier $e_p(\pi,\chi)$ by
\[e_p(\pi,\chi)=\begin{cases}1&\text{if $\chi_p$ is ramified};\\
(1-\al_\frakp^{-1}\chi(\frakP))(1-\al_\frakp^{-1}\chi(\frakPbar))&\text{if $\chi_p$ is unramified, $\frakp=\frakP\ol{\frakP}$ is split};\\
1-\al_\frakp^{-2}&\text{if $\chi_p$ is unramified, $\frakp=\frakP$ is inert};\\
1-\al_\frakp^{-1}\chi(\frakP)&\text{if $\chi_p$ is unramified, $p=\frakP^2$ is ramified}
\end{cases}\]
(\cf \cite[\S 2.10]{BD:Heegner_Mumford}).
\end{defn}
\begin{prop}\label{P:5.W}Suppose that $\chi$ has conductor $\frakp^{s}$. Let $n=\max\stt{1,s}$. Then
\[\frac{1}{L(\onehalf,\pi_E\ot\chi)}\cdot\localP{\pi(\cmpt_\frakp^\setn)W_\frakp^\dagger}{\chi}=e_p(\pi,\chi)^{2-\Ord_p(N)}\cdot L(1,\tau_{E/F})^2\cdot\begin{cases}\al_\frakp^{2}p^{-2}&\cdots s=0,\\
p^{-s}&\cdots s>0.\end{cases}
\]
Therefore, by \eqref{E:midformulaII.W} we have
\[\cP(\cmpt^\setn_p,\vp_p,\chi)\cdot\frac{\norm{\newForm}_p}{\abs{D_\cK}^\onehalf}=e_p(\pi,\chi)^{2-\Ord_p(N)}\cdot L(1,\tau_{E/F})^2\cdot\begin{cases}\al_\frakp^{2}p^{-2}&\cdots s=0,\\
p^{-s}&\cdots s>0.\end{cases}\]
\end{prop}
\begin{proof}
For $t\in E$, we put
\[\iota_\cmpt^\setn(t):=(\cmpt_\frakp^\setn)^{-1}\iota(t)\cmpt_\frakp^\setn.\]
It is easy to see that $\iota_\cmpt^\setn(\cmJ)\in\Iwahori$ if $n\geq 1$.

Suppose that $p=\frakP\ol{\frakP}$ is split in $E$. Recall that $\Diff=\bftheta-\ol{\bftheta}\in \cO_{\cK_\frakP}^\x=\Zp^\x$. A direct computation shows that
\begin{align*}\iota_\cmpt^\setn(t)=&\pMX{1}{\Diff^{-1}\p^{-n}}{0}{1}\pMX{t}{}{0}{\ol{t}}\pMX{1}{-\Diff^{-1}\p^{-n}}{0}{1}.\\
\end{align*}
We find that
\begin{align*}
&\localP{\pi(\cmpt_\frakp^\setn)W_\frakp^\dagger}{\chi}\\
=&\int_{F^\x}\int_{F^\x}W_{\frakp}^\dagger(\DII{ax}{1}\pMX{1}{-\Diff^{-1}\p^{-n}}{0}{1})
W_{\frakp}^\dagger(\DII{-a}{1}\pMX{1}{-\Diff^{-1}\p^{-n}}{0}{1})\chi_{\frakP}(x) \dx a\dx x\\
=&\int_{\Zp-\stt{0}}\addchar(-\Diff^{-1}\p^{-n}x)\mu_\frakp\chi_\frakP\Abs^\onehalf(x)\dx x\cdot \int_{\Zp-\stt{0}}\addchar(\Diff^{-1}\p^{-n}a)\mu_\frakp\chi_\frakP^{-1}\Abs^\onehalf(a)\dx a.
\end{align*}
If $n=s\geq 1$, then $L(\onehalf,\pi_E\ot\chi)=1$, and as 
\[\int_{\Zp^\x}\psi(p^{-r}x)\chi_\frakP^\pm(x)\dx x=0\text{ for all }r<n,\]
we find that 
\begin{align*}\localP{\pi(\cmpt_\frakp^\setn)W_\frakp^\dagger}{\chi}=&\int_{\Zp^\x}\addchar(-\Diff^{-1}\p^{-n}x)\chi_\frakP(x)\dx x\cdot \int_{\Zp^\x}\addchar(\Diff^{-1}\p^{-n}a)\chi_\frakP^{-1}(a)\dx a\\
=&\ep(1,\chi_\frakP,\addchar)\ep(1,\chi^{-1}_\frakP,\addchar)\cdot\zeta_F(1)^2\\
=&\abs{\p^n}L(1,\tau_{E/F})^2.\end{align*}
If $s=0$ and $n=1$, then $\chi$ is unramified and
\begin{align*}
\localP{\pi(\cmpt_\frakp^\setn)W_\frakp^\dagger}{\chi}&= (\frac{-\abs{\p}}{1-\abs{\p}}+\frac{\mu_\frakp\chi_\frakP\Abs^\onehalf(\p)}{1-\mu_\frakp\chi_\frakP\Abs^\onehalf(\p)})\cdot (\frac{-\abs{\p}}{1-\abs{\p}}+\frac{\mu_\frakp\chi_\frakP^{-1}\Abs^\onehalf(\p)}{1-\mu_\frakp\chi_\frakP^{-1}\Abs^\onehalf(\p)})\\
&=\frac{1-\mu_\frakp^{-1}\chi_\frakP^{-1}\Abs^\onehalf(\p)}{1-\mu_\frakp\chi_\frakP\Abs^\onehalf(\p)}
\cdot\frac{1-\mu_\frakp^{-1}\chi_\frakP\Abs^\onehalf(\p)}{1-\mu_\frakp\chi_\frakP^{-1}\Abs^\onehalf(\p)}\cdot\zeta_F(1)^2\cdot \mu_\frakp^2(\p)\abs{\p}\\
&=\al_\frakp^2\abs{\p}^2\cdot\frac{L(\onehalf,\mu_\frakp\chi_\frakP)L(\onehalf,\mu_\frakp\chi_{\frakPbar})}{L(\onehalf,\nu_\frakp\chi_{\frakPbar})L(\onehalf,\nu_\frakp\chi_{\frakP})}
\cdot\zeta_F(1)^2\quad(\nu_\frakp=\mu_\frakp^{-1},\,\chi_{\frakPbar}=\chi_\frakP^{-1})\\
&=\al_\frakp^2\abs{\p}^2\cdot L(\onehalf,\pi\ot\chi)\cdot L(1,\tau_{E/F})^2\cdot\begin{cases}
(1-\al_\frakp^{-1}\chi_\frakP^{-1}(\p))^2(1-\al_\frakp^{-1}\chi_{\frakPbar}^{-1}(\p))^2&\cdots p\ndivides N,\\
(1-\al_\frakp^{-1}\chi_\frakP^{-1}(\p))(1-\al_\frakp^{-1}\chi_{\frakPbar}^{-1}(\p))&\cdots p\mid N.
\end{cases}.
\end{align*}
This proves the formula in the split case.

Now we assume that $\frakp$ is non-split. We introduce the matrix coefficient $\bfm^\dagger:G(\Qp)=\GL_2(\Qp)\to\C$ defined by \[\bfm^\dagger(g):=\bfb(\pi(g)W_\frakp^\dagger,W_\frakp^\dagger).\]
The function $\bfm^\dagger(g)$ only depends on the double coset $\Iwahori g \Iwahori$, and by definition
\[\bfm^\dagger(1)=\frac{1}{1-\mu_p^2(p)\abs{p}}.\] Put
\[\Prd^*:=\int_{E^\x/F^\x}\bfm^\dagger(\iota_\cmpt^\setn(t))\chi(t)dt.\]
It is clear that \beq\label{E:18.W}\localP{\pi(\cmpt_\frakp^\setn)W_\frakp^\dagger}{\chi}=\Prd^*\cdot\frac{\abs{D_\cK}^{-\onehalf}L(1,\tau_{E/F})}{\zeta_F(1)}.\eeq
To compute $\Prd^*$, we make some observations. Let $r\in\Z_{\geq 0}$. For $y\in \p^r\Zp^\x$, we have
\begin{align*}\iota_\cmpt^\setn(1+y\CMP)&\in \Iwahori\bfw\DII{\p^{n-r}}{\p^{r-n}}\Iwahori\text{ if }0\leq r<n\quad(\bfw=\pMX{0}{1}{-1}{0})\end{align*}
and $\iota_\cmpt^\setn(1+y\CMP)\in \Iwahori$ if $r\geq n$. For $y\in\p\Zp$, we have
\[\iota_\cmpt^\setn(y+\CMP)\in \Iwahori\bfw\DII{\p^{n+e-1}}{\p^{-n}}\Iwahori,\]
where $e=1$ if $p$ is inert and $e=2$ if $p$ is ramified. Put \[X_r:=\int_{p^r\Zp}\chi(1+y\OKbasis)d'y,\]
where $d'y$ is the Haar measure on $\Zp$ such that $\vol(\Zp,d'y)=L(1,\tau_{E/F})\abs{D_\cK}^\onehalf$.
Using the decomposition \[E^\x=F^\x(1+\Zp\CMP)\disjoint F^\x(p\Zp+\CMP),\]
we find that
\beq\label{E:long.W}
\begin{aligned}
\Prd^*=&\int_{\Zp}\chi(1+y\CMP)\bfm^\dagger(\iota_\cmpt^\setn(1+y\CMP))d'y+\int_{\p\Zp}\chi(y+\CMP)\bfm^\dagger(\iota_\cmpt^\setn (y+\CMP) )\abs{y+\CMP}_E^{-1}d'y\\
=&X_{n}\cdot\bfm^\dagger(1)+\sum_{r=0}^{n-1}(X_r-X_{r+1})\bfm^0(\bfw\DII{p^{n-r}}{p^{r-n}})+Y_0\cdot\bfm^\dagger(\bfw\DII{p^{n+e-1}}{p^{-n}}),
\end{aligned}
\eeq
where
\[Y_0:=\int_{\p\Zp}\chi(y+\CMP)d'y\cdot\abs{\uf}^{1-e}.\]
Suppose that $n=s\geq 1$. Then it is easy to verify \begin{itemize}\item $X_r=0$ if $0<r<n$, \item  $X_0+Y_0=0$ if $p$ is inert, and $X_0=Y_0=0$ if $p$ is ramified.\end{itemize} It follows from \eqref{E:long.W} that
\begin{align*}
\Prd^*&=X_n\cdot\bfm^\dagger(1)+(-X_n)\cdot \bfm^\dagger(\bfw\DII{\p}{\p^{-1}})\quad(X_n=\abs{p^n}L(1,\tau_{E/F})\abs{D_\cK}^\onehalf).
\end{align*}
If $\pi$ is a unramified principal series, then
\begin{align*}
\Prd^*&=\bfb(\newW_{\pi}-\pi(\DII{\p^{-1}}{\p})\newW_{\pi},W_\frakp^\dagger)\cdot X_n\\
&=(\frac{\mu_\frakp(\p)}{1-\mu_\frakp\nu_\frakp^{-1}\Abs(\p)}-\frac{\nu_\frakp(\p)}{1-\abs{\p}})\cdot\frac{1-\mu_\frakp\nu_\frakp^{-1}\Abs(\p)}{\mu_\frakp(\p)-\nu_\frakp(\p)}\cdot X_n\\
&=\zeta_F(1) \cdot \abs{\p^n} L(1,\tau_{E/F})\abs{D_\cK}^\onehalf.
\end{align*}
If $\pi$ is special, then by \eqref{E:special.W},
\begin{align*}\Prd^*=&\frac{1+\mu_\frakp\nu_\frakp^{-1}(\p)}{1-\mu_\frakp\nu_\frakp^{-1}\Abs(\p)}\cdot X_n\\
=&\zeta_F(1)\cdot \abs{\p^n}L(1,\tau_{E/F})\abs{D_\cK}^\onehalf\quad(\mu_\frakp\nu_\frakp^{-1}(\p)=\abs{\p}).\end{align*}
Suppose that $s=0$ and $n=1$. Then we have
\[\Prd^*=X_1\cdot\bfm^\dagger(1)+(X_0-X_1)\cdot \bfm^\dagger(\bfw\DII{\p}{\p^{-1}})+Y_0\cdot \bfm^\dagger(\bfw\DII{\p^{e}}{\p^{-1}}).\]
\begin{itemize}
\item If $p$ is inert, then $X_0=1-X_1=L(1,\tau_{E/F})$ and $Y_0=X_1$.
\item If $p$ is ramified, then $X_0=\abs{D_\cK}^\onehalf$ and $Y_0=\chi(\uf_E)\abs{D_\cK}^\onehalf$.
\end{itemize}

\emph{Case (1):} $p$ is inert and $\pi$ is unramified. Then
\begin{align*}
\Prd^*&=\bfm^\dagger(1)+X_0\cdot (\bfm^\dagger(\bfw\DII{\p}{\p^{-1}})-\bfm^\dagger(1))\\
&=\frac{1}{1-\mu^2_p(p)\abs{p}}-\frac{1}{1-\abs{p}^2}\\
&=\al_\frakp^2\abs{\p}^2(1-\al_\frakp^{-2})^2\cdot \zeta_F(1)L(1,\tau_{E/F})\cdot L(\onehalf,\pi_E\ot\chi)
\end{align*}

\emph{Case (2):} $p$ is inert and $\pi$ is special. Then $\al_p^2=1$, so $e_p(\pi,\chi)=0$ and
\begin{align*}\Prd^*=\frac{\abs{\p}-\al_\frakp\mu_\frakp(\p)\abs{\p}^\onehalf}{1-\abs{\p}^2}=0.\end{align*}

\emph{Case (3):} $p$ is ramified and $\pi$ is unramified. Note that
\[\bfm^\dagger(\bfw\DII{\p^2}{\p^{-1}})=\mu\Abs^\onehalf(\p)\bfm^\dagger(\bfw\DII{\p}{\p^{-1}}).\]
Let $\beta:=\chi(\uf_E)\mu_p\Abs^\onehalf(p)$. Then we have
\begin{align*}
\abs{D_\cK}^{-\onehalf}\Prd^*&=\abs{p}\cdot\bfm^\dagger(1)+(1-\abs{p})\cdot \bfm^\dagger(\bfw\DII{\p}{\p^{-1}})+\chi(\uf_E)\bfm^\dagger(\bfw\DII{\p^2}{\p^{-1}})\\
&=(1-\abs{p}+\beta)(\bfm^\dagger(\bfw\DII{\p}{\p^{-1}})-\bfm^\dagger(1))+(1+\beta)\bfm^\dagger(1)\\
&=-\frac{(1-\abs{p}+\beta)}{1-\abs{p}}+\frac{1+\beta}{1-\beta^2}\quad(\mu_p^2(p)\abs{p}=\beta^2=\al_p^2\abs{p}^2)\\
%&=\frac{\beta^2(1-\beta^{-1}\abs{p})}{(1-\abs{p})(1-\beta)}\\
&=\al_p^2\abs{p}^2\cdot (1-\chi(\uf_E)\al_p^{-1})^2\cdot \zeta_{F}(1)L(\onehalf,\pi_E\ot\chi).
\end{align*}

\emph{Case (4):} $p$ is ramified and $\pi$ is special. Then $\al_p^2=\chi(\uf_E)^2=1$ and
\begin{align*}
\abs{D_\cK}^{-\onehalf}\Prd^*&=\abs{p}\bfm^\dagger(1)+(1-\abs{p})(-\al_p)\bfm^\dagger(\DII{\p^{-1}}{1})+\chi(\uf_E)(-\al_p)\bfm^\dagger(\DII{\p^{-2}}{1})\\
&=\frac{\abs{p}^2(1-\chi(\uf_E)\al_p)}{1-\mu^2(p)\abs{p}}.
\end{align*}
We find that if $\chi(\uf_E)=\al_p$, then $\Prd^*=e_p(\pi,\chi)=0$ and if $\chi(\uf_E)=-\al_p$, then
\begin{align*}\abs{D_\cK}^{-\onehalf}\Prd^*=&\frac{2\abs{p}^2}{(1-\chi(\uf_E)\mu_p\Abs^\onehalf(p))(1+\chi(\uf_E)\mu_p\Abs^\onehalf(p))}\\
=&\abs{p}^2(1-\al_p^{-1}\chi(\uf_E))\cdot \zeta_F(1)L(\onehalf,\pi_E\ot\chi).\end{align*}
The above calculations together with \eqref{E:18.W} completes the proof in the inert or ramified case.
\end{proof}
\subsection{Central value formula}
We are ready to prove the central value formula connecting the toric period integral of the $p$-stabilized form $\Form^\dagger$ in \eqref{E:defnP.W} and the central $L$-value of $\pi_\cK$ twisted by anticyclotomic characters $\chi$ satisfying \eqref{crit}.
\begin{thm}\label{T:central.W} Suppose that $\chi$ has conductor $p^s$. Let $n=\max\stt{1,s}$. Then we have
\begin{align*}P(\cmpt^\setn,\Form^\dagger,\chi)^2\cdot\frac{\norm{\newForm}_{\Gamma_0(N)}}{\PetB{\vformB}{\vformB}}=&
\frac{2^{2-k}(-1)^mD_\cK^{k-2}}{\sqrt{D_\cK}}\cdot L(\onehalf,\pi_\cK\ot\chi)\cdot e_p(\pi,\chi)^{2-\Ord_p(N)}\cdot L(1,\tau_{\cK_\frakp/\Q_\frakp})^2 \\
&\times \chi(\frakN^+)\ep(\pi_p)\prod_{\pme|(D_\cK,N^-)}(1-\ep_\pme(\pi_\cK,\chi))\cdot\begin{cases}\al_p^2p^{-2}&\cdots s=0\\
p^{-s}&\cdots s>0.\end{cases}\end{align*}
\end{thm}
\begin{proof}This follows from \corref{C:1.W}, \propref{P:2.W} and \propref{P:5.W}.
\end{proof} 
%!TEX root =central.tex

\def\normalizedForm{\Form^{\dagger}}
\def\Thetam{\Theta^{[m]}}
\section{Theta elements and \padic $L$-functions}\label{S:ThetaElment}
\subsection{$\ell$-adic modular forms}\label{SS:l_adic_modular_forms}
\def\pvformB{\wh f_{\pi'}}
Let $\ell\ndivides N^-$ be a rational prime in \subsecref{SS:setup}. We briefly review $\ell$-adic modular forms on $B^\x$. Let $A$ be a $\cO_{\cK_\frakl}$-algebra. For an open compact subgroup $U\subset \hatR^\x$, we define the space of $\ell$-adic modular forms of weight $\wt$ and level $U$ by
\begin{align*}
\padicMF_\wt(U,A)=&\stt{\wh f:\hatB\to L_k(A)\mid \wh f(\al g u)=\rho_{k}(u_\ell^{-1})\wh f(g),\,\al\in B^\x,\,u\in U\hatQ}.
\end{align*}
We write $\padicMF_\wt(N_B,A):=\padicMF_\wt(\hatR^\x,A)$. Recall that we have fixed an embedding $\iota_\ell:\Qbar\hookto \C_\ell$. Let $\lam$ and $\frakl$ be the primes of $\Qbar$ and $\cK$ induced by $\iota_\ell$ respectively.
We let $i_{\cK_\frakl}:B\hookto M_2(\cK_\frakl)$ be the composition $i_{\cK_\frakl}:=\iota_\ell\circ i_\cK$ defined in \eqref{E:2.W}. Define $\rho_{k,\ell}:B_\ell^\x\to\Aut L_{k}(\C_\ell)$ by
\beq\label{E:gamma.W}\rho_{k,\ell}(g):=\rho_k(i_{\cK_\frakl}(g)).\eeq
By definition, $\rho_{k,\ell}$ is compatible with $\rho_{k,\infty}$ in the sense that $\rho_{k,\ell}(g)=\rho_{k,\infty}(g)$ for every $g\in B^\x$, and one checks that %$\rho_{k,\ell}(R_\ell^\x)\subset \Aut L_k(\cO_{\C_\ell})$. When $\ell=p$ and $\frakl=\frakP$, it is easy to see that
\[\rho_{k,\ell}(g)=\rho_k(\gamma_\frakl i_\ell(g)\gamma_\frakl^{-1}),\,\gamma_\frakl:=\pMX{\sqrtb}{-\sqrtb\ol{\CMP}}{-1}{\CMP}\in\GL_2(\cK_{\frakl}).\]
Here $i_\ell:B_\ell\iso M_2(\Q_\ell)$ is the isomorphism fixed in \eqref{E:embedding.W}. If $A=A[\frac{1}{\ell}]$, there is an isomorphism:\[\MF_{\wt}(N_B,A)\isoto\padicMF_\wt(N_B,A),\,f\mapsto \wh f(g):=\rho_{k}(\gamma_\frakl^{-1})\rho_{k,\ell}(g_\ell^{-1})f(g).\] We call $\wh f$ the \emph{$\ell$-adic avatar of $f\in\MF_{\wt}(N_B,A)$}.

Let $\Q(\pi)$ be the Hecke field of $\pi$. In other words, $\Q(\pi)$ is the finite extension of $\Q$ generated by Fourier coefficients of the elliptic new form $f_\pi$. Let $\cO_{\pi,\ell}\subset\C_\ell$ be the completion of the ring of integers of $\Q(\pi)$ with respect to $\lam$. The $\cO_{\pi,\ell}$-module $\padicMF_\wt(N_B,\cO_{\pi,\ell})[\pi'_f]:=\padicMF_\wt(N_B,\cO_{\pi,\ell})\cap \padicMF_\wt(N_B,\C_\ell)[\pi'_f]$ has rank one. We say $\vformB\in\MF_\wt(N_B,\C)[\pi'_f]$ is \emph{$\lam$-adically normalized} if its $\ell$-adic avatar $\pvformB$ is a generator of $\padicMF_\wt(N_B,\cO_{\pi,\ell})[\pi'_f]$ over $\cO_{\pi,\ell}$. This is equivalent to the following condition:
\[\pvformB(g_0)\not\con 0\pmod{\lam}\text{ for some }g_0\in G(\Af).\]

\def\whmForm{\wh\vp^{[m]}_{\pi'}}
\subsection{Theta elements}\label{SS:thetalets}
Let $n\geq 1$ be a positive integer. Let $\cG_n=\cK^\x\bksl \hatK/\wh\cO_n^\x$ be the Picard group of the order $\cO_n$. We identify $\cG_n$ with the Galois group of the ring class field of conductor $p^n$ over $\cK$ via geometrically normalized reciprocity law. Denote by $[\cdot ]_n:\wh\cK^\x\to \cG_n,\,a\mapsto [a]_n$ the natural projection map. We consider the automorphic form $\mForm=\Psi(\bfv_m^*\ot \vformB)$ in \eqref{E:defn.W} and define the function   
\[\whmForm:\wh\cK^\x\to \C,\quad a\mapsto\whmForm(a):=\mForm(x_n(a))\iota_p^{-1}(\ol{a_p}/a_p)^m.\] 
Replacing $\vformB$ by $\vformB^\dagger$, we can define the $p$-stabilizations $(\mForm)^\dagger:=\Psi(\bfv_{m}^*\ot\vformB^\dagger)$ and $(\whmForm)^\dagger$ in a similar manner. By \eqref{E:8.W}, we can verify that
$\whmForm$ factors through $K^\x\bksl\wh K^\x$.
\begin{defn}\label{D:Theta.W}Fix a set $\Xi_n$ of representatives of $\cG_n$ in $K^\x\bksl\wh K^\x$, define 
the $n$-th theta element $\Thetam_n(\vformB^\dagger)\in\C[\cG_n]$ of weight $m$ is defined by
\begin{align*}\Thetam_n(\vformB^\dagger):=&\al_\frakp^{-n}\cdot \sum_{a\in \Xi_n}(\whmForm)^\dagger(a)\cdot [a]_n\\
=&\al_\frakp^{-n}\cdot \sum_{a\in \Xi_n}(\mForm)^\dagger(x_n(a))\iota_p^{-1}(\ol{a_p}/a_p)^m\cdot [a]_n.\end{align*}
\end{defn}
We consider theta elements of weight zero. The function $\wh\varphi_{\pi'}^ {[0]}$ factors through $\cG_n$, so we can extend $\wh\varphi_{\pi'}^ {[0]}$ linearly to be a function $\wh\varphi_{\pi'}^ {[0]}:\C[\cG_n]\to\C$, and the definition of $\Theta_n^{[0]}(\vformB^\dagger)$ does not depend on the choice of $\Xi_n$. Let $P_n:=[1]_n\in\cG_n$ be the distinguished Gross point of conductor $p^n$. Then
\[\wh\varphi_{\pi'}^ {[0]} (\sg(P_n))=\varphi_{\pi'}^ {[0]} (x_n(a))\text{ if }\sg=[a]_n\in\cG_n.\]
Define the \emph{regularized} Gross point $P_n^\dagger$ as follows. If $p\ndivides N$, we define $P_n^\dagger$ by the formal sum\[P_n^\dagger:=\frac{1}{\al_p^n}\cdot P_n-\frac{1}{\al_p^{n+1}}\cdot P_{n-1},\]
and if $p\mid N$, we define $P_n^\dagger=\al_p^{-n}\cdot P_n$. We have 
\[\Theta^{[0]}_n(f_{\pi'}^\dagger)=\sum_{\sg\in\cG_n}\wh\varphi_{\pi'}^ {[0]} (\sg(P_n^\dagger))\cdot \sg,
\]
and $\stt{\Theta_n^{[0]}(\vformB^\dagger)}_n$ satisfy the following compatible relation.
\begin{lm}\label{L:7.W}
Let $\pi_{n+1,n}:\cG_{n+1}\to \cG_n$ be the natural quotient map. We have
\[\pi_{n+1,n}(\Theta^{[0]}_{n+1}(\vformB^\dagger))=\Theta^{[0]}_n(\vformB^\dagger).\]
\end{lm}
\begin{proof}This is standard. For $n'> n$, let $K_{n',n}:=\Ker(\cG_{n'}\to \cG_n)$.
Using the description \[K_{n',n}=[(\cO_{n}\ot\Zp)^\x]_{n'}=\stt{[1+\p^{n} u\CMP]_{n'}\mid u\in\Z/p^{n'-n}\Z},\]
we find that
\[\sum_{u\in K_{n+1,n}}(\varphi_{\pi'}^{[0]})^\dagger(x_{n+1}(au))=(\varphi_{\pi'}^{[0]})^\dagger\mid U_p(x_n(a))=\al_\frakp\cdot (\varphi_{\pi'}^{[0]})^\dagger(x_n(a)).\]
The lemma follows.
\end{proof}
Let $f_{\pi}\in S_k(\Gamma_0(N))$ be the elliptic new form corresponding to $\newForm$. The Fourier coefficients of the $q$-expansion $f_{\pi}(q)=\sum_{n>0}\bfa_n(f_\pi) q^n$ at the infinity cusp are given by
\[\bfa_n(f_\pi)=W_{\pi,f}(\DII{n}{1})n^\frac{k}{2}\quad(W_{\pi,f}=\prod_{\pme<\infty}W_{\pi_\pme}).\]
Let $A_p:=p^{\frac{k}{2}-1}\al_p$ be a root of the Hecke polynomial of $f_\pi$ at $p$ \ie $X^2-\bfa_p(f_\pi)X+p^{k-1}$ if $p\ndivides N$; $X-\bfa_p(f_\pi)$ if $p\mid N$. Let $\chi$ be an anticyclotomic Hecke character of conductor $p^s$ and weight $(m,-m)$ for an integer $-k/2< m <k/2$. Recall that $\wh\chi:\cK^\x\bksl\wh\cK^\x\to\cO_{\Cp}^\x$ denotes the \padic avatar of $\chi$ defined by $\wh\chi(a)=\chi(a)(a_p/\ol{a_p})^m$. We are going to give the interpolation formula of the square of
\begin{align*}\wh\chi(\Thetam_n(\vformB^\dagger)):=&\al_\frakp^{-n}\cdot \sum_{a\in \Xi_n}(\mForm)^\dagger(x_n(a))\iota_p^{-1}((\ol{a_p}/a_p)^m\cdot\wh\chi(a))\\
=&\al_\frakp^{-n}\cdot \sum_{[a]_n\in \cG_n}(\mForm)^\dagger(x_n(a))\chi(a)
\end{align*} for every $n\geq \max\stt{1,s}$ in terms of the central value of the Rankin-Selberg $L$-function $L(f_{\pi}/\cK,\chi,s)$  attached to $f_{\pi}$ and the theta series attached to $\chi$. Recall that connection between the automorphic $L$-function $L(s,\pi_\cK\ot\chi)$ and the Rankin $L$-series $L(f_{\pi}/\cK,\chi,s)$ is given by
\beq\label{E:19.W}L(s,\pi_\cK\ot\chi)=\Gamma_\C(s+\frac{k-1}{2}+m)\Gamma_\C(s+\frac{k-1}{2}-m)\cdot L(f_{\pi}/\cK,\chi,s+\frac{k-1}{2}).\eeq
 Define the \emph{period} $\Omega_{\pi,N^-}$ of $\pi'$ by
\beq\label{E:periodV.W}\Omega_{\pi,N^-}:=\frac{(4\pi)^k\norm{\newForm}_{\Gamma_0(N)}}{\PetB{\vformB}{\vformB}}.\eeq
\begin{prop}\label{P:evaluation.W}Suppose that $\chi$ has the conductor of $p^s$. For every $n\geq \max\stt{s,1}$, we have the interpolation formula
\begin{align*}\wh\chi(\Thetam_n(\vformB^\dagger)^2)=& \Gamma(\frac{k}{2}+m)\Gamma(\frac{k}{2}-m)\cdot \frac{L(f_{\pi}/\cK,\chi,\frac{k}{2})}{\Omega_{\pi,N^-}} \cdot e_p(\pi,\chi_t\nu)^{2-\Ord_p(N)}\cdot A_p^{-2s}(p^sD_\cK)^{k-1}\\
&\times  \frac{u_\cK^2}{\sqrt{D_\cK}}\cdot\chi(\frakN^+)\ep(\pi_p)(-1)^m\prod_{\pme|(D_\cK,N^-),\,\pme=\w^2}(1-\ep(\pi_\pme)\chi(\w)).\end{align*}
\end{prop}
\begin{proof}
We may assume $n=\max\stt{1,s}$, using the argument in \lmref{L:7.W}. By \eqref{E:8.W} and the definition of theta elements, we have \[P(\pi'(\cmpt^\setn),\Form^\dagger,\chi)=\vol(\wh\cO_n^\x)\al_\frakp^{n}\cdot \wh\chi(\Thetam_n(\vformB^\dagger))\quad(\Form^\dagger=(\mForm)^\dagger),\]
where $\vol(\wh\cO_n^\x)$ denotes the volume of the image of $\C^\x\wh\cO_n^\x$ in $K^\x\A^\x\bksl \AK^\x$ with respect to the measure $dt$. Recall that $dt$ is chosen so that $\vol(K^\x\A^\x\bksl \AK^\x,dt)=2L(1,\tau_{K/\Q})$. Using the class number formula, we have
\[\vol(\wh\cO_n^\x)=\vol(\wh\cO_\cK^\x)\cdot L(1,\tau_{\cK_\frakp/\Q_\frakp})\abs{\p}^{n}=\frac{4}{\sqrt{D_\cK}\cdot u_\cK}L(1,\tau_{\cK_\frakp/\Q_\frakp})\p^{-n}.\]
Combining these equations, we find that
\[\wh\chi(\Thetam_n(\vformB^\dagger))=\frac{\sqrt{D_\cK}\cdot u_\cK}{4}\cdot\frac{\al_\frakp^{-n}\p^n}{L(1,\tau_{\cK_\frakp/\Q_\frakp})}\cdot P(\cmpt^\setn,\Form^\dagger,\chi).\]
Thus, the proposition follows from \thmref{T:central.W}, \eqref{E:19.W} and the formula:
\[\ep_\pme(\pi_\cK,\chi)=\ep(\pi_\pme)\chi(\w)\text{ for }\pme|(D_\cK,N^-),\,\pme=\w^2.%-\pme^{1-\frac{k}{2}}a_\pme(f_\pi)\chi(\w)
\qedhere\]
\end{proof}
\begin{remark}This proposition verifies \cite[Conjecture\,2.17]{BD:AJM_exceptional_zero}, and hence removes the assumption in Theorem 3.4 \loccit
\end{remark}

\subsection{\padic $L$-functions}We shall use theta elements to construct anticyclotomic \padic $L$-functions attached to $f$ and derive the evaluation formulae. We begin with a key observation.
\begin{lm}\label{L:2.W}Let $A\subset\Cp$ be an $\cO_{\cK_\frakP}$-algebra. Let $\vform\in\MF_{\wt}(N_B,\Cp)$ such that the \padic avatar $\wh\vform\in \padicMF_\wt(N_B,A)$ is a \padic modular form over $A$. Let \[f_{\bfv_m}:=\Psi(\bfv^*_m\ot\vform)\in\cA(G)\] be defined as in \eqref{E:1.W}. For $a\in\hatK$, we have \begin{mylist}\item $p^{n(\frac{k-2}{2})}f_{\bfv_m}(x_n(a))(\ol{a_p}/a_p)^{m}\in \frac{1}{(k-2)!}A;$
\item the congruence relation
\[p^{n(\frac{k-2}{2})}f_{\bfv_m}(x_n(a))(\ol{a_p}/a_p)^{m}\con \sqrtb^{\frac{2-k}{2}}\Lpair{X^{k-2}}{\wh\vform(x_n(a))}\pmod{\frac{p^n}{(k-2)!}A}.\]
\end{mylist}
\end{lm}
\begin{proof}We write $a=(a^\setp,a_p)\in (\wh\cK^\setp)^\x\x \cK_p^\x$. By definition, we have
\beq\label{E:6.W}\vform_{\bfv_m}(x_n(a))(\ol{a_p}/a_p)^{m}=\Lpair{\rho_k(\gamma_\frakP)\rho_{k,p}((\cmpt_\frakp^\setn)^{-1})\bfv^*_m}{\wh\vform(x_n(a))}.\eeq
Here we are making use of the fact that $\rho_{k,p}(t)$ acts on $\bfv_m$ by $(\ol{t}/t)^{m}$ for $t\in(\cK\ot\Qp)^\x$. A direct computation shows that $ \rho_k(\gamma_\frakP^{-1})(\rho_{k,p}((\cmpt_\frakp^\setn)^{-1})=\rho_k(Z_p)$, where
% \beq\label{E:16.W}
\begin{align*}Z_p=&\pMX{1}{\sqrtb}{0}{p^n\sqrtb \delta}\text{ if $p$ is split in $\cK$}\\
\intertext{ and}
%\beq\label{E:17.W}
Z_p=&\pMX{1}{\sqrtb}{-p^n\CMP}{-p^n\sqrtb\ol{\CMP}}\text{ if $p$ is non-split in $\cK$.}
\end{align*}
Note that $\det Z_p=\sqrtb p^n \delta$,
\[Z_p\in M_2(\cO_{\cK_\frakp}),\quad Z_p\con \pMX{1}{\sqrtb}{0}{0}\pmod{p^nM_2(\cO_{\cK_\frakp})}. \]
For $P(X,Y)\in L_k(A)$, we find that
\begin{align*}p^{\frac{n(k-2)}{2}}D_\cK^\frac{k-2}{2}\cdot \rho_k(\gamma_\frakP^{-1})\rho_{k,p}((\cmpt_\frakp^\setn)^{-1})P(X,Y)=&p^{\frac{n(k-2)}{2}}D_\cK^\frac{k-2}{2}P((X,Y)Z_p)(\det Z_p)^\frac{2-k}{2}\\
=&\sqrtb^{\frac{2-k}{2}}P((X,Y)Z_p)\in L_k(A).
\end{align*}
In particular, \[p^{\frac{n(k-2)}{2}}\rho_k(\gamma_\frakP^{-1}) \rho_{k,p}((\cmpt_\frakp^\setn)^{-1})\bfv_m^*\con \sqrtb^\frac{2-k}{2}X^{k-2}\pmod{p^nL_k(A)}.\]
In view of \eqref{E:6.W}, the assertions of the proposition follow immediately.
\end{proof}
We make the following ordinary hypothesis:
\beqcd{ord}
\text{$\vformB$ is $p$-adically normalized, and the \padic valuation of the $U_p$-eigenvalue $\Ord_p(\al_p)=\frac{2-k}{2}$.}
\eeqcd
\begin{cor}\label{C:congruence.W}Suppose \eqref{ord} holds. Then $\Thetam_n(\vformB^\dagger)\in [(k-2)!]^{-1}\cO_{\pi,p}[\cG_n]$. Moreover,
\[\Thetam_n(\vformB^\dagger)\con \Theta^{[0]}_n(\vformB^\dagger)\pmod{p^n[(k-2)!]^{-1}\cO_{\pi,p}[\cG_n]}.\]
\end{cor}
\begin{proof}By definition, for $a\in\wh K^\x$ we have
\[\al_p^{-n}(\whmForm)^\dagger(a)=\al_p^{-n}\mForm(x_n(a))(\ol{a_p}/a_p)^{m}-\al_p^{-2}(\al_p^{-(n-1)}\mForm(x_{n-1}(a))(\ol{a_p}/a_p)^{m}).\]
Applying \lmref{L:2.W} to $f=\vformB$, we find that
\begin{itemize}\item $\al_p^{-n}(\whmForm)^\dagger(a)\in\cO_{\pi,p}$,
\item $\al_p^{-n}(\whmForm)^\dagger(a)\pmod{p^n[(k-2)!]^{-1}}$ is independent of $m$.
\end{itemize}
The corollary follows immediately.
\end{proof}

Let $\cG_\infty:=\prolim_n \cG_n$. Let $\Gamma^-\iso\Zp$ be the maximal $\Zp$-free quotient group of $\cG_\infty$ and let $\Delta$ be the torsion subgroup of $\cG_\infty$. We have an exact sequence
\[\exact{\Delta}{\cG_\infty}{\Gamma^-}.\]
Fix a non-canonical isomorphism $\cG_\infty\iso\Delta\x\Gamma^-$ once and for all. If $n\geq 1$, the map $\Delta\to \cG_\infty\to \cG_n$ is injective, and hence \[\cG_n\iso \Delta\x \Gamma^-_n,\,\Gamma^-\surjto\Gamma^-_n:=\cG_n/\Delta.\] Let $\chi_t$ be a branch character, \ie  a character $\chi_t:\Delta\to\Qbar^\x$ and let $\cO=\cO_{\pi,p}[\chi_t]$.
Define the $\chi_t$-branch of $\Theta^{[0]}_n(\vformB^\dagger)$ by \[\Theta_n(\vformB^\dagger,\chi_t):=\chi_t(\Theta^{[0]}_n(\vformB^\dagger))\in [(k-2)!]^{-1}\cO[\Gamma^-_n].\]
We define
\[\Theta_\infty(\pi):=\stt{\Theta^{[0]}_n(\vformB^\dagger)}_n\in[(k-2)!]^{-1}\cO\powerseries{\cG_\infty}\,;\quad \Theta_\infty(\pi,\chi_t):=\stt{\Theta_n(\vformB^\dagger,\chi_t)}_n=\chi_t(\Theta_\infty(\pi))\in[(k-2)!]^{-1}\cO\powerseries{\Gamma^-}.\]
Let $\wtsp$ be the set of critical specializations defined in the introduction. Note that $\wtsp$ consists of the \padic avatars of Hecke characters $\chi$ of $p$-power conductor satisfying \eqref{crit} and trivial on $\Delta$.
\begin{thm}\label{T:Thetaevaluation.W} Let $\wh\nu\in\wtsp$ be a \padic character of weight $(m,-m)$ and conductor $p^s$. We have the interpolation formula
\begin{align*}\wh\nu(\Theta_\infty(\pi,\chi_t)^2)=& \Gamma(\frac{k}{2}+m)\Gamma(\frac{k}{2}-m)\cdot \frac{L(f_{\pi}/\cK,\chi_t\nu,\frac{k}{2})}{\Omega_{\pi,N^-}} \cdot e_p(\pi,\chi_t\nu)^{2-\Ord_p(N)}\cdot p^sA_p^{-2s}(p^sD_\cK)^{k-2}\\
&\times u_\cK^2\sqrt{D_\cK}\cdot\ep(\pi_p)(-1)^m\prod_{\pme|(D_\cK,N^-),\,\pme=\w^2}(1-\ep(\pi_\pme)\chi_t(\w))\cdot \chi_t\nu(\frakN^+).\end{align*}
\end{thm}
\begin{proof}Let $n_0=\max\stt{s,1}$. For each integer $r>n_0$, we choose $n>r$ such that $\wh\nu\pmod{p^r}$ is trivial on $\wh\cO_{n}^\x$. Let $\chi=\chi_t\nu$. By \corref{C:congruence.W}, we have
\[\wh\nu(\Theta_\infty(\pi,\chi_t))\con \wh\chi(\Theta^{[0]}_n(\vformB^\dagger))\con \wh\chi(\Thetam_n(\vformB^\dagger))=\wh\chi(\Thetam_{n_0}(\vformB^\dagger))\pmod{p^r[(k-2)!]^{-1}}.\]
This congruence relation holds for all $r>n_0$. Therefore, $\wh\nu(\Theta_\infty(\pi,\chi_t))=\wh\chi(\Thetam_{n_0}(\vformB^\dagger))$, and the theorem follows from \propref{P:evaluation.W}.
\end{proof}
\begin{Remark}
The theta element $\Theta_\infty(\pi,\chi_t)$ is the \emph{square root} of the anticyclotomic \padic $L$-function associated to $(\pi,\chi_t)$. In view of the evaluation formula, we assume the following local root number condition in the remainder of this article:
\beqcd{ST}\ep(\pi_\pme)\chi_t(\w)=-1  \text{ for every }\pme\mid (D_\cK,N^-)\text{ with }\pme=\w^2.\eeqcd
Note that \eqref{ST} is always satisfied if $(D_\cK,N^-)=1$.
\end{Remark}
Let $*:\cO\powerseries{\cG_\infty}\to\cO\powerseries{\cG_\infty}$ be the involution defined by $\sg\mapsto \sg^{-1}$. We show that $\Theta_\infty(\pi)$ satisfies the functional equation in the following sense.
\begin{thm}\label{T:functionaleq}Let $r_0$ be the number of prime divisors of $(D_\cK,N^-)$. Let \[\ep':=(-1)^{r_0+\frac{k}{2}}\prod_{\pme\ndivides pD_\cK}\ep(\pi_\pme)\in\stt{\pm 1}\] and let $\sg_{\frakN^+}=\stt{[\frakN^+]_n}$ be the image of $\frakN^+$ in $\cG_\infty$. We have the functional equation:
\[\Theta_\infty(\pi)^*=\ep'\cdot \Theta_\infty(\pi)\cdot \sg_{\frakN^+}^{-1}.\]
\end{thm}
\begin{proof}Let $\Form=\Form^{[0]}$. Using the automorphy of $\Form$, we have \[\Form(x_n(a^{-1}))=\Form^\dagger(x_n(\ol{a}))=\Form(x_n(a)\cmJ_{\cmpt,n})\quad(\cmJ_{\cmpt}^\setn=(\cmpt^\setn)^{-1}\cmJ\cmpt^\setn)).\]
By the choice of $\cmJ$ and \lmref{L:10.W}, it is straightforward to show that
\[\pi(J_{\cmpt,n})\Form(Q)=\ep'\cdot \Form(\sg_{\frakN^+}(Q))\text{ for }Q\in \cG_n.\]
Therefore,
\[\Theta_n(\pi)^*=\sum_{\sg\in\cG_n}\Form(\sg^{-1}(P_n^\dagger))\sg=\ep'\cdot \Theta_n(\pi)\cdot [{\frakN^+}]^{-1}_n.\]
This proves the theorem.
\end{proof}
\begin{remark}Note that $\ep(\pi_\infty)=(-1)^\frac{k}{2}$. If $\chi_t=\bfone$ is the trivial character, then $\ep(\pi_\pme)=-1$ for all $\pme|N^-$ by \eqref{ST}. It follows that
\[\ep'=(-1)^\frac{k}{2}\prod_{\pme\not =p}\ep(\pi_\pme)=\ep(\pi)\ep(\pi_p).\]
\end{remark}
\section{The non-vanishing of theta elements modulo $\ell$}\label{S:NV}
\subsection{}
%For every $n\in\Z_+$ we have an exact sequence
%\[\exact{\Delta}{\cG_n}{\Gamma_n},\,\Gamma_n=\Gamma/\Gamma^{p^{n+\delta}}.\]
We retain the notation in the previous section. Throughout, we suppose that $\vformB$ is $\lam$-adically normalized. The purpose of this section is to study the non-vanishing properties of theta elements $\stt{\Theta_n(\vformB^\dagger,\chi_t)}_n$ modulo $\lam$.
Let $\Delta^\alg$ be the subgroup of $\cG_\infty$ generated by the image of $\cK^\x_{\ram}:=\prod_{\pme|D_{\cK}}\cK^\x_\pme$. It is clear that $\Delta^\alg$ is a $(2,\cdots,2)$-subgroup of $\Delta$. Let $\cD_0$ be a set of representatives of $\Delta^\alg$ in $\cK^\x_{\ram}$. Choose an arbitrary set $\cD_1$ of representatives of $\Delta/\Delta^\alg$ in $\hatK$. Then $\cD:=\cD_1\cD_0$ be a set of representatives of $\Delta$ in $\hatK$.
By definition, we can write
\beq\label{E:42.W}\Theta_n(\vformB^\dagger,\chi_t)=\sum_{[u]_n\in\Gamma^-_n}\sum_{\tau\in\cD_1}\chi_t(\tau)\left(\al_p^{-n}\sum_{d\in\cD_0}\normalizedForm(x_n(\tau u)d)\chi_t(d)\right)\cdot[u]_n\quad(\normalizedForm=\Psi(\bfv_0^*\ot\vformB^\dagger)).\eeq
\subsection{Uniform distribution of CM points}\label{SS:Uniform_CM}We recall a crucial result in \cite{Vatsal_Cornut:Documenta}. Let $\ol{\cK}^\x$ the closure of $\cK^\x$ in $\hatK$ and let $\ol{B}^\x$ be the closure of $B^\x$ in $\hatB$.
Let $\CMspace:=\ol{\cK}^\x\bksl \hatB$, $\GrossC:=\ol{B}^\x\bksl \hatB$ and $\cZ:=\ol{\Q}_+\bksl \hatQ$. The group $\hatB$ acts on these spaces by the right translation and $\hatK$ acts on $\CMspace$ and $\cZ$ by the left multiplication.

Let $\Red:\CMspace\to \GrossC$ be the natural quotient map and let $c:\GrossC\to\cZ$ be the map induced by the reduced norm $\rmN:B^\x\to \Q^\x$. For $g\in \hatB$, let $[g]$ denote the image of $g$ in $\CMspace$. Let $\cU$ be an open compact subgroup of $\hatB$. Put
\[\cX(\cD_1,\cU)=\prod_{\tau\in\cD_1}\cX/\cU\text{ and }\cZ(\cD_1,\cU)=\prod_{\tau\in\cD_1}\cZ/\rmN(\cU).\]
Define
\begin{align*}\Red_{\cD_1}:&\CMspace\longrightarrow \cX(\cD_1,\cU),\quad x\mapsto (\Red(\tau\cdot x)\cU)_{\tau\in\cD_1}\\
\intertext{ and}
c_{\cD_1}:&\cX(\cD_1,\cU)\longto\cZ(\cD_1,\cU),\quad(x_{\tau})_{\tau\in\cD_1}\mapsto (\rmN(x_\tau))_{\tau\in\cD_1}.\end{align*}

The following proposition is a special case of \cite[Corollary\,2.10]{Vatsal_Cornut:Documenta}.
\begin{prop}\label{P:Vatsal_Cornut.W} Let $\cH$ be a $B^\x_p$-orbit in $\CMspace$ and let $\ol{\cH}$ be the image of $\cH$ in $\CMspace/\cU$. Then for all but finitely many $x\in\ol{\cH}$, we have
\[\Red_{\cD_1}(\wh\cO_\cK^\x\cdot x)=c_{\cD_1}^{-1}(\wh\cO_\cK^\x\cdot \ol{x}),\]
 where $\overline{x}=c_{\cD_1}\circ \Red_{\cD_1} (x)$.
\end{prop}
\begin{proof}This is \cite[Corollary\,2.10]{Vatsal_Cornut:Documenta} with $\mathfrak{S} =\{ \emptyset \}$ and
$\mathfrak{R}=\cD_1$.
\end{proof}
The following corollary is a immediate consequence of the above proposition.
\begin{cor}\label{C:Vatsal_Cornut} Let $\stt{\beta_\tau}_{\tau\in\cD_1}$ be a sequence in $A$ such that $\beta_{\tau_1}\in A^\x$ for some $\tau_1$. Let $f\in\cM_2(\cU,A)$. Suppose that
\begin{mylist}
\item  $f$ is not Eisenstein;
\item  $\rmN(\cU)\supset \prod_{\pme|D_\cK}\Z_\pme^\x$.
\end{mylist}
Then there exists an integer $n_0$ such that for every $n>n_0$,  we have
\[\sum_{\tau\in \cD_1}\beta_\tau\cdot f(x_n(a\tau))\not =0\text{ for some }a\in \hatK.\]
\end{cor}
\begin{proof}Let $P_0:=[\cmpt^{(0)}]\in\CMspace$. Let $\cH=P_0\cdot B^\x_p$ be the $B^\x_p$-orbit of $P_0$. Then $P_n:=P_0\cdot\DII{\p^n}{1}\in \cH$, and from \eqref{E:EichlerOrder.W} we find that the image of $\stt{P_n}_{n=1,2,\cdots}$ are distinct in $\hatK\bksl \hatB/\hatR^\x$. By \propref{P:Vatsal_Cornut.W}, there exists $n_0$ such that \beq\label{E:20.W}\Red_{\cD_1}(\wh\cO_\cK^\x P_n)=c_{\cD_1}^{-1}(\wh\cO_\cK^\x\ol{P_n})\text{ for every }n>n_0.\eeq Fix $n>n_0$. Since $f$ is not Eisenstein, $f(y)\not =f(z)$ for some $y,z\in\cX$ with $c(y)=c(z)$. The assumption (2) implies that the norm map $\rmN:\hatK\to \cZ/\rmN(\cU)$ is surjective as the class number of $\Q$ is one. Hence,  replacing $\cD_1$ by $a'\cD_1$ for some $a'\in\hatK$ if necessary, we may assume \[c(y)=c(z)=c(\Red(P_n))\pmod{\rmN(\cU)}.\] Take $(w_\tau)_{\tau\in\cD_1}\in c^{-1}_{\cD_1}(\ol{P_n})$. By \eqref{E:20.W}, there exist $a_1,a_2\in\wh\cO_\cK^\x$ such that \[\Red_{\cD_1}(a_1P_n)=(y,w_{\tau_2},\cdots)\,;\,\Red_{\cD_1}(a_2P_n)=(z,w_{\tau_2},\cdots).\] %(\Red([x_n(a_1\tau_1)]),\Red([x_n(a_1\tau_2)]),\cdots )_{\tau_i\in\cD_1} $\Red_{\cD_1}(a_2P_n)=(z,w_{\tau_2},\cdots)$.
It follows that
\[\sum_{\tau\in \cD_1}\beta_\tau\cdot f(x_n(a_1\tau))-\sum_{\tau\in \cD_1}\beta_\tau\cdot f(x_n(a_2\tau))=\beta_{\tau_1}(f(y)-f(z))\not =0.\]
It is clear that either $a_1$ or $a_2$ does the job. \end{proof}
\subsection{Eisenstein functions}
Let $A$ be a $\Z$-algebra. Let $U$ be an open-compact subgroup of $\hatB$. Denote by $\cM_2(U,A)$ the set of functions $h:B^\x\bksl \hatB\to A$ such that $h$ is right invariant by $U$. Let $\cM_2(A):=\dirlim_{U\subset \hatB}\cM_2(U,A)$ be the space of smooth $A$-valued functions on $B^\x\bksl \hatB$. We write $\varrho:\hatB\to\Aut\cM_2(A)$ for the right translation of $\hatB$.
\begin{defn}Let $B^1=\stt{g\in B^\x\mid \rmN(g)=1}$ be an algebraic group over $\Q$. Put
\[\cM_2(A)_{\Eis}:=\stt{h\in\cM_2(A)\mid \varrho(g_1)h=h\text{ for all $g_1\in B^1(\Af)$}}.\]
It is clear that $\cM_2(A)_{\Eis}$ is a $\hatB$-invariant subspace of $\cM_2(A)$. Let
\[\cS_2(A):=\cM_2(A)/\cM_2(A)_{\Eis}.\]
Let $\cS_2(U,A)$ denote the image of $\cM_2(U,A)$ in $\cS_2(A)$.

A function $h\in\cM_2(A)$ is called \emph{Eisenstein} if $h\in\cM_2(A)_{\Eis}$. Equivalently, $h$ is Eisenstein if and only
$h(g)=h_1(\rmN(g))$ for some smooth function $h_1:\Q_+^\x\bksl \hatQ\to A$.
\end{defn}
We make the following observations in the flavor of Ihara's lemma. The first one is taken from \cite[Proposition\,5.3]{Vatsal:nonvanishing}.
\begin{lm}\label{L:Ihara1.W}Let $\pme\ndivides N^-$ be a finite place. Let $t_\pme\in B^\x_\pme$ such that $\rmN(t_\pme)=\pme$. Let $\cR'\in\End\cM_2(A)$ be the endomorphism defined by
\[\cR'=1+\beta \cdot\varrho(t_\pme)\quad(\beta\in A).\]
Suppose that $U\supset R_\pme^\x:=(R\ot_\Z\Z_\pme)^\x=\GL_2(\Z_\pme)$. Then $\cR':\cS_2(U,A)\to\cS_2(A)$ is injective.
\end{lm}
\begin{proof}Let $h\in\cM_2(U,A)$. If $\cR(h)\in \cM_2(A)_{\Eis}$, then it is easy to see that $h$ is right invariant by $\SL_2(\Z_\pme)$ and $t_\pme\SL_2(\Z_\pme)t_\pme^{-1}$. By a theorem of Ihara, $h$ is right invariant by $\SL_2(\Q_\pme)$ and hence by $B_1(\Af)$ in virtue of the strong approximation theorem for $B_1$.
\end{proof}

\begin{lm}\label{L:Ihara2.W}Let $\pme\ndivides N^-$ be a finite place. Let $\beta_1,\cdots ,\beta_s\in A$ and let $\cR\in\End(\cM_2(A))$ be the endomorphism defined by
\[\cR=1+\sum_{i=1}^s\beta_i\cdot \varrho(\DII{\pme^{-i}}{1}).\]
Then $\cR:\cS_2(U,A)\to\cS_2(A)$ is injective.
\end{lm}
\begin{proof}
Let $h\in\cM_2(U,A)$. We need to show that if \[\cR(h)=h+\sum_{i=1}^s\beta_i\cdot \varrho(\DII{\pme^{-i}}{1})h\in\cM_2(A)_{\Eis},\]
then $h\in\cM_2(A)_{\Eis}$. Let $S_h\supset U$ be the stabilizer of $h$ in $\hatB$. Namely,
\[S_h:=\stt{g\in\hatB\mid \varrho(g)h=h}.\]
Let $N(\Q_\pme)$ be the unipotent radical of the upper triangular subgroup in $\GL_2(\Q_\pme)$ and let \[N':=N(\Q_\pme)\cap U\subset S_h.\] Write $u:=\DII{\pme^{-1}}{1}$ and
\[P(u)=-\sum_{i=1}^{s}\beta_i \varrho(u)^{i-1}\in\End\cM_2(A)^{N'}.\]
By the assumption that $h-\varrho(u)P(u)h\in\cM_2(A)_{\Eis}$, we find that for every positive integer $n$,
\beq\label{E:43.W}h-\varrho(u^n)P(u)^n\cdot h\in \cM_2(A)_{\Eis}.\eeq
Since $h,\, P(u)^n h\in\cM_2(A)^{N'}$ and $N'$ is a proper subgroup, using the identity $u^{-n}\pMX{1}{x}{0}{1}u^n=\pMX{1}{\pme^nx}{0}{1}$, we deduce from \eqref{E:43.W} that
\[\pMX{1}{x}{0}{1}\in S_h\text{ for all }x\in\Q_\pme.\]
On the other hand, $\pMX{1}{0}{y}{1}\in S_h$ for some $y\in\Q_\pme^\x$. By the relation
\[\pMX{1}{0}{y}{1}=\pMX{1}{y^{-1}}{0}{1}\pMX{0}{y^{-1}}{-y}{0}\pMX{1}{y^{-1}}{0}{1},\]
we find $w_0=\pMX{0}{y^{-1}}{-y}{0}\in S_h$, and hence
\[\pMX{1}{0}{x}{1}=w_0^{-1}\pMX{1}{-y^{2}x}{0}{1}w_0\in S_h.\]  It follows that $S_h$ contains $\SL_2(\Q_\pme)\cdot U$. By the strong approximation for $B^1$, we find that $B^1(\Af)\subset S_h$.
\end{proof}
Let $\rho_{\pi,\lam}:\Gal(\Qbar/\Q)\to\GL_2(\cO_{\pi,\ell})$ be the $\ell$-adic Galois representation attached to $f_\pi$.
\begin{lm}\label{L:Eisensteinfunction.W}Suppose that the residual Galois representation $\rhobar_{\pi,\lam}$ is irreducible and $\ell>k-2$. Let $\bfv\in L_k(\cO_{\pi,\ell})$. If $\bfv\not\con 0\pmod{\lam}$, the function $f_\bfv(g):=\Lpair{\bfv}{\pvformB(g)}\pmod{\lam}\in\cM_2(\cO_{\pi,\ell}/\lam)$ is not Eisenstein.
\end{lm}\begin{proof}We note that $f_{\bfv}$ is not a zero function by the irreducibility of $L_\wt(\cO_{\pi,\ell}/\lam)$ as $\GL_2(\cO_{\pi,\ell}/\lam)$-module when $\ell>k-2$. Therefore, if $f_\bfv$ is Eisenstein, then $\rhobar_{\pi,\lam}$ is reducible. \end{proof}

\subsection{The vanishing of $\mu$-invariants}\label{SS:padicL}
For each positive integer $s$, define the open compact subgroup $\cI_1(p^s)$ of $\wh R^\x$ by \[\cI_1(p^s):=\stt{g\in \wh R^\x\mid g_p\con\pMX{1}{*}{0}{1}\pmod{p^s}}.\]
Let $\uf$ be a generator of the maximal ideal of $\cO:=\cO_{\pi,\ell}[\chi_t]$. Suppose that $\ell=p$. We follow the approach of Vatsal to study the $\mu$-invariant of $\Theta_\infty(\pi,\chi_t)\in\cO\powerseries{\Gamma^-}$.
 \begin{thm}\label{T:1.W} Let $r_0$ be as in \thmref{T:functionaleq}. In addition to \eqref{ST} and \eqref{ord}, we assume that
\begin{mylist}
\item $p>k-2$ and $p\ndivides 2^{r_0}$,
\item $\rhobar_{\pi,\lam}$ is absolutely irreducible. \end{mylist}Then Iwasawa $\mu$-invariant of $\Theta_\infty(\pi,\chi_t)$ vanishes.
\end{thm}
\begin{proof}
 We define a function $\bff_p:B^\x\bksl \hatB/\hatQ\to \cO/\uf\cO$ by
\beq\label{E:new_fucn.W}\bff_p(g)=\sqrtb^\frac{2-k}{2}\cdot \Lpair{X^{k-2}}{\pvformB(g)}\pmod{\uf}.\eeq
A direct computation shows that $\bff_p\in\cM_2(\cI_1(p),\cO/\uf\cO)$. Let $\cU_{\cD_0}$ be the open-compact subgroup given by
\[\cU_{\cD_0}=\stt{g\in \wh R^\x\mid g_\pme\in R_\pme^\x\cap \uf_{\cK_\pme}R_\pme^\x\uf_{\cK_\pme}^{-1}\text{ for all }\pme|D_\cK},\]
where $\uf_{\cK_\pme}$ is a uniformizer of $\cK_\pme$. It is easy to see that
\beq\label{E:13.W}\rmN(\cU_{\cD_0})\supset \prod_{\pme|D_\cK}\Z_\pme^\x.\eeq
Let $\bff_{\cD_0}\in\cM_2(\cU_{\cD_0}\cap \cI_1(p),\cO/\uf\cO)$ be the function defined by
\[\bff_{\cD_0}(g)=\sum_{d\in\cD_0}\bff_p(gd)\chi_t(d)\pmod{\uf}.\]
From \lmref{L:2.W}, we can deduce that
\[\Theta_n(\vformB^\dagger,\chi_t)\pmod{\uf}=A_p^{-n}\sum_{[u]_n\in\Gamma^-_n}\left(\sum_{\tau\in\cD_1}\bff_{\cD_0}^\dagger(x_n(u\tau))\chi_t(\tau)\right)\cdot[u]_n,\]
where $\bff_{\cD_0}^\dagger\in\cM_2(\cU,\cO/\uf\cO)$ is given by
\[\bff_{\cD_0}^\dagger:=\bff_{\cD_0}-p^\frac{k-2}{2}A_p^{-1}\cdot \varrho(\DII{\frakp^{-1}}{1})\bff_{\cD_0},\,\cU:=\cU_{\cD_0}\cap \cI_1(p^2).\]
 To prove the vanishing of the $\mu$-invariant, we need to show that for $n\gg 0$, there exists $a\in\hatK$ such that
\[\sum_{\tau\in\cD_1}\bff_{\cD_0}^\dagger(x_n(a\tau))\chi_t(\tau)\not \con0\pmod{\uf},\]
and in turn, it suffices to verify the assumptions for $\bff_{\cD_0}^\dagger$ in \corref{C:Vatsal_Cornut}. By \eqref{E:13.W}, $\rmN(\cU)\supset \prod_{\pme|D_\cK}\Z_\pme^\x$. By \lmref{L:Eisensteinfunction.W}, $\bff_p$ is not Eisenstein, which implies that $\bff_{\cD_0}^\dagger\in\cM_2(\cU,\cO/\uf\cO)$ is not Eisenstein by the following \lmref{L:6.W}.
\end{proof}

\begin{lm}\label{L:6.W} Suppose that $\bff_p$ is not Eisenstein. Then $\bff_{\cD_0}^\dagger$ is not Eisenstein.
\end{lm}
\begin{proof}Let $\pme|D_\cK$ be a ramified place and $\uf_{\cK_\pme}$ be a uniformizer of $\cK_\pme^\x$. Put
\[\cR'_\pme:=1+\chi_t(\uf_{\cK_\pme})\varrho(\uf_{\cK_\pme})\in\End(\cM_2(\cO/\uf\cO)).\]
Let $\stt{\pme_i}_{i=1,\cdots s}$ are the set of prime divisors $\pme$ of $D_\cK$ with $\pme\ndivides N^-$. By the assumption \eqref{ST}, we have
\[\bff_{\cD_0}=2^{r_0}\cdot \cR'_{\pme_1}\circ\cR'_{\pme_2}\cdots\circ\cR'_{\pme_s} (\bff_p).\]
Applying \lmref{L:Ihara1.W} and \lmref{L:Ihara2.W}, we conclude that $\bff_{\cD_0}^\dagger=\cR_p(\bff_{\cD_0})$ with $\cR_p:=1-\al_p^2A_p\cdot \varrho(\DII{\p^{-1}}{1})$ is not Eisenstein if $\bff_p\not\in\cM_2(\cO/\uf\cO)_{\Eis}$.
\end{proof}

\subsection{The non-vanishing modulo $\ell$ with anticyclotomic twists}

Suppose that $\ell\not =p$. We prove the non-vanishing of central $L$-values modulo $\ell$ with anticyclotomic twists, using a Galois average trick of Sinnot in \cite{Sinnott:W_Theorem}.
\begin{thm}\label{T:2.W} Let $\chi$ be an anticyclotomic Hecke character of conductor $p^{s_0}$ and weight $(m,-m)$ with $-k/2<m<k/2$. Suppose that
\begin{mylist}
\item $(\pi,\chi)$ satisfies \eqref{ST},
\item $\ell\ndivides 2^{r_0}pND_\cK$ and $\ell>k-2$,
\item $\rhobar_{\pi,\lam}$ is absolutely irreducible.
\end{mylist}
 %Then there exists an integer $n_0$ such that for all finite order characters $\nu:\Gamma^-\to\bbmu_{p^\infty}$ of conductor $p^n$ with $n>n_0$, we have
%\beq\label{E:11.W}
%\[\wh\chi\nu(\Theta_n^{[m]}(\vformB^\dagger))\not\con 0\pmod{\lam}\]
%Therefore, it follows from \propref{P:evaluation.W} that
Then for all but finitely many $\nu:\Gamma^-\to\bbmu_{p^\infty}$, we have
\[\frac{L(f_{\pi}/\cK,\chi\nu,\frac{k}{2})}{\Omega_{\pi,N^-}}\not\con 0\pmod{\lam}.\]
\end{thm}
\begin{proof} Choose a finite extension $\cO$ of $\Z_\ell$ in $\C_\ell$ so that $\cO$ contains $\cO_{\pi,\ell}$ and the values of $\chi$ on $\AKf^\x$ and let $\uf$ be a uniformizer of $\cO$. Let $\pvformB^\dagger$ be the $\ell$-adic avatar of the $p$-stabilization $\vformB^\dagger$. Define a function $F_\ell:B^\x\bksl \hatB\to \cO$ by
\[F_\ell(g)=\Lpair{\rho_k(\gamma_\frakl^{-1})\bfv_m^*}{\wh f_{\pi'}^\dagger(g)}\quad(\bfv_m^*=D_\cK^\frac{2-k}{2}\cdot X^{\frac{k-2}{2}-m}Y^{\frac{k-2}{2}+m}).\]
 Note that $\gamma_\frakl\in\GL_2(\cO_{\cK_\frakl})$ as $\ell\ndivides D_\cK$. For each integer $n\geq s_0$, we put \[\Theta_n^{\chi}:=\sum_{[a]_n\in\cG_n}F_\ell(x_n(a))\wh\chi([a]_n)\cdot [a]_n\in \cO[\cG_n],\]
where $\wh\chi:\cG_\infty\to\cO^\x$ is the $\ell$-adic avatar of $\chi$. One checks by definition that
\[(\vp_{\pi'}^{[m]})^\dagger(x_n(a))\chi(a)=F_\ell(x_n(a))\wh\chi(a),\,a\in\hatK,\]
and hence for each $\nu:\Gamma^-\to \bbmu_{p^\infty}$,
\[\nu(\Theta_n^{\chi})=\iota_\ell\iota_p^{-1}(\wh\chi\nu(\Theta_n^{[m]}(\vformB^\dagger))).\]
In view of \propref{P:evaluation.W}, it suffices to show $\nu(\Theta_n^\chi)\not\con 0\pmod{\lam}$ for all but finitely many $\nu$.

Let $\bfk_\ell=\cO/\uf\cO[\bbmu_{p}]$ be the finite extension of the finite field $\cO/\uf\cO$ generated by the values of $\bbmu_p$. Put
\[\cI'=\stt{g\in \wh R^\x\mid g\con 1\pmod{\ell},\,g_p\in \Iwahori}.\]
Then $F_\ell\pmod{\lam}\in\cM_2(\cI',\bfk_\ell)$. Define $F_{\cD_0}\in\cM_2(\cU'_{\cD_0},\bfk_\ell)$ by
\[
F_{\cD_0}(g):=\sum_{d\in\cD_0}\wh\chi(d)F_\ell(gd)\pmod{\lam}\quad(\cU'_{\cD_0}=\cI'\cap \cU_{\cD_0}).\]
It is clear that $F_{\cD_0}$ is invariant by the Iwahori subgroup $\Iwahori$ and is an $U_p$-eigenform with eigenvalue $\al_p$.
% It is easy to see that \beq\label{E:9.W}
%\begin{aligned}&\Tr_{\bfk_\ell(\nu)/\bfk_\ell}(\chi_t\nu(a^{-1})\nu(\Theta_n(\vformB^\dagger,\chi_t))\pmod{\lam})\\
%=&\al_p^{-n}\sum_{[u]_n\in \Gamma^-_n}\Tr_{\bfk_\ell(\nu)/\bfk_\ell}(\nu([u]_n)\cdot \sum_{\tau\in\cD_1}F_{\cD_0}(x_n(a\tau u))\chi_t(\tau).\end{aligned}\eeq
Let $p^s$ be the order of the Sylow $p$-subgroup of $\bfk_\ell^\x$. Let $\nu:\Gamma^-_n\to\bbmu_{p^\infty}$ be a character of conductor $p^n$ with $n>\max\stt{s,s_0}$ (so $\nu:\Gamma^-_n\hookto \bbmu_{p^\infty}$ is injective). Put
\[C_n=\stt{\gamma\in\Gamma^-_n\mid \nu(\gamma)\in \bfk_\ell^\x}.\]
Then we have $C_n=\Ker(\cG_{n}\to \cG_{n-s})$. Let $\bfk_\ell(\nu)$ be the field generated by the values of $\nu$ over $\bfk_\ell$. Since $\bfk_\ell$ contains $\bbmu_p$, $d_\nu:=[\bfk_\ell(\nu):\bfk_\ell]$ is a $p$-power, and for a $p$-power root of unity $\zeta\in \bfk_\ell(\nu)$, we have
\[\Tr_{\bfk_\ell(\nu)/\bfk_\ell}(\zeta)=\begin{cases}0&\cdots \zeta\not\in \bfk_\ell,\\
d_\nu&\cdots \zeta\in \bfk_\ell.\end{cases}\]
It follows from the above that for each $a\in \AKf^\x$,
\beq\label{E:9.W}\begin{aligned}
&\Tr_{\bfk_\ell(\nu)/\bfk_\ell}(\al_p^n\wh\chi\nu(a^{-1})\cdot\nu(\Theta_n^\chi)\pmod{\lam})\\
=&d_\nu\cdot\sum_{[u]_n\in C_n}\sum_{\tau\in\cD_1}F_{\cD_0}(x_n(a\tau u))\wh\chi(\tau)\\
=&d_\nu\cdot \sum_{\tau\in\cD_1}\sum_{y\in\Z/p^s\Z}F_{\cD_0}(x_n(a\tau)\pMX{1}{\frac{y}{p^{s}}}{0}{1})\wh\chi(d)\zeta_\nu^y
\end{aligned}\eeq
for some primitive $p^s$-th root of unity $\zeta_\nu$. Define $\wtd F_{\cD_0}\in\cM_2(\bfk_\ell(\nu))$ by
\begin{align*}
%F_{\cD_0}(g):=&\sum_{d\in\cD_0}\chi_t(d)F(gd),\\
\wtd F_{\cD_0}(g):=&\sum_{y\in\Z/p^s\Z}\zeta_\nu^y\varrho(\pMX{1}{\frac{y}{p^{s}}}{0}{1})F_{\cD_0}(g).\end{align*}
Then $\wtd F_{\cD_0}\in\cM_2(\cU',\bfk_\ell(\nu))$ for $\cU'=\cI_1(p^{2s})\cap \cU_{\cD_0}'$.
We can rephrase \eqref{E:9.W} as
\beq\label{E:10.W} \Tr_{\bfk_\ell(\nu)/\bfk_\ell}(\al_p^n\wh\chi\nu(a^{-1})\cdot\nu(\Theta_n^\chi)\pmod{\lam})=d_\nu\cdot\sum_{\tau\in\cD_1}\wtd F_{\cD_0}(x_n(a\tau))\wh\chi(\tau).\eeq
We proceed to show that $\wtd F_{\cD_0}$ is not Eisenstein. Under our assumptions, $F_{\cD_0}$ is not Eisenstein by \lmref{L:Eisensteinfunction.W} and \lmref{L:6.W}. A simple computation shows that
\begin{align*}
\sum_{a\in(\Z/p^s\Z)^\x}\varrho(\DII{a}{1})\wtd F_{\cD_0}=&\sum_{a\in(\Zp/p^s\Z)^\x}\sum_{y\in\Z/p^s\Z}\zeta_\nu^{ay}\varrho(\pMX{1}{\frac{y}{p^s}}{0}{1})F_{\cD_0}\\
=&p^s\cdot F_{\cD_0}-p^{s-1}\sum_{y\in\Z/p\Z}\varrho(\pMX{1}{\frac{y}{p}}{0}{1})F_{\cD_0}\\
%=&p^s\cdot(F-p^{-1}\varrho(\DII{p^{-1}}{1})U_p\mid F_{\cD_0})\\
=&p^s\cdot (1-p^{-1}\al_p\cdot \varrho(\DII{\p^{-1}}{1}))F_{\cD_0}.
\end{align*}
The above equation implies that $\wtd F_{\cD_0}$ is not Eisenstein by \lmref{L:Ihara2.W}. On the other hand, it is clear that $\rmN(\cU')\supset\prod_{\pme|D_\cK}\Z_\pme^\x$ by \eqref{E:13.W}. Hence, we can deduce the theorem from \corref{C:Vatsal_Cornut} in view of \eqref{E:10.W}.
\end{proof}
%Let $E_{/\Q}$ be an elliptic curve over $\Q$ of conductor $N$. Let $\Sel_\ell(E/K_n)$ be the $\ell$-primary part of the Selmer group of $E$ over $K_n$, where $K_n$ is the ring class field of $\cK$ of conductor $p^n$. Combined with the result in \cite{Longo_Vigni:Selmer_groups}, we obtain the following corollary:
%\begin{cor}\label{C:Longo_Vigni} Suppose that $(E,\cK)$ satisfies the Heegner condition. Let $\ell\not =p$ be a rational prime such that \begin{mylist}\item $\ell\ndivides (p-1)h_\cK\cdot ND_\cK$,
%\item the residual Galois representation $\rho_{E,\ell}:G_\Q\to\Aut _{\bbF_\ell}E[\ell]$ is surjective.
%\item $\ell$ does not divide the minimal degree of a modular parametrization $X_0(N)\to E$.
%\end{mylist}
% Then there exists $n_0$ such that
%\[\#\Sel_\ell(E/K_n)=\#\Sel_\ell(E/K_{n_0})\text{ for all }n.\]
%\end{cor}

%!TEX root = central.tex

\newtheorem*{hypCR}{Hypothesis (CR${}^+$)}
\section{The comparison between periods}\label{S:periods}
In this section, we compare the periods $\Omega_{\pi,N^-}$ defined in \eqref{E:periodV.W} and Hida's canonical period $\Omega_\pi$. Henceforth, we assume $\ell=p$ and $\vformB^\dagger$ is $p$-ordinary. Let $\bbT(\Gamma_0(N))$ be the Hecke algebra over $\cO$ of the space of elliptic modular forms $S_k(\Gamma_0(N))$. Then $\pi$ gives rise to an $\cO$-algebra homomorphism $\lam_\pi:\bbT(\Gamma_0(N))\to\cO$ such that $\lam_\pi(T_\pme)=\Tr\rho_{\pi,p}(T_\pme)$ for all $\pme\ndivides pN$. Let $\eta_\pi(N)$ be an $\cO$-generator of the congruence ideal $I_{\pi}(N):=\lam_{\pi}(\Ann_{\bbT(\Gamma_0(N))}\Ker\lam_{\pi})\subset \cO$. Then Hida's canonical period $\Omega_\pi$ is defined by
\[\Omega_\pi:=\frac{(4\pi)^k\norm{\newForm}_{\Gamma_0(N)}}{\eta_\pi(N)}.\]
Here we are making use of the fixed embedding $\cO\hookto\Cp\iso\C$. In general, the ratio $\Omega_{\pi,N^-}/\Omega_{\pi}$ lies in $\cO$. This section is devoted to showing that $\Omega_{\pi,N^-}/\Omega_{\pi}\in\cO^\x$ in certain favorable situations. When $k=2$ and $N$ is square-free, under mild assumptions, Pollack and Weston \cite[Theorem\,6.8]{Pollack_Weston:AMU} even give the formula of the ratio $\Omega_{\pi,N^-}/\Omega_{\pi}$ in terms of local Tamagawa components at primes dividing $N^-$. They do not need to assume the ordinary hypothesis, but it is not clear to us if their approach is applicable if $k>2$. Nonetheless, it is pointed out in \cite{Pollack_Weston:AMU} that the statement $\Omega_{\pi,N^-}/\Omega_{\pi}\in\cO^\x$ is equivalent to the freeness of spaces of modular forms on $B$ over the associated Hecke algebra and the vanishing of these local Tamagawa components. Therefore, it is natural to study the comparison between periods $\Omega_{\pi,N^-}$ and $\Omega_{\pi}$ by the standard techniques developed by Wiles, Taylor-Wiles, Diamond and Fujiwara in the proof of "$R=T$" theorems.

Let $G_\Q=\Gal(\Qbar/\Q)$. For each place $\pme$, we fix a decomposition of group $G_\pme$ in $G_\Q$ and let $I_\pme$ be the inertia group in $G_\pme$. Let $\rho_0:=\rhobar_{\pi,p}$ denote the residual Galois representation and let $N_{\rho_0}$ be the prime-to-$p$ part of the Artin conductor of $\rho_0$. Throughout, we assume the following:
%\beqcd{R}\rho_0\text{ is ramified at every }\pme|N^-.\eeqcd
\begin{hypCR}\label{A:1.W}
\begin{mylist}
\item The prime $p>k+1$ and $p\ndivides N$.
\item The restriction of $\rho_0$ to the absolute Galois group of $\Q(\sqrt{(-1)^\frac{p-1}{2}p})$ is absolutely irreducible.
\item If $\pme\mid N^-$ and $\pme\con \pm 1\pmod{p}$, then $\pme\mid N_{\rho_0}$.
\item If $\pme\parallel N^+$ and $\pme\con 1\pmod{p}$, then $\pme\mid N_{\rho_0}$.
\item $N_{\rho_0}$ and $N/N_{\rho_0}$ are co-prime.
\end{mylist}
\end{hypCR}
We will prove the following proposition in \subsecref{SS:proof} after preparing some notation and recall basic facts in the first two subsections.
\begin{prop}\label{P:congruence}Suppose the hypothesis (CR${}^+$) holds. If $\rho_0$ is ramified at every prime dividing $N^-$ $($\ie $N^-\mid N_{\rho_0})$, then the congruence ideal $I_\pi(N)$ is generated by $\PetB{\vformB}{\vformB}$. In other words, $\Omega_{\pi,N^-}=u\cdot\Omega_{\pi}$ for some unit $u\in\cO^\x$.
\end{prop}

\subsection{Hecke algebras and congruence ideals}
For an open compact subgroup $U\subset \hatB$, put $\cS(U):=\padicMF_{\wt}(U,\cO)$. For $g\in\hatB$, Let $[U_1gU_2]\in\Hom_\cO(\cS(U_2),\cS(U_1))$ be the Hecke operator defined by
\[[U_1 g U_2]f(g)=\sum_{i}\rho_{k,p}(g_i)f(gg_i)\quad(U_1 g U_2= \disjoint_{i}g_i U_2).\]
Let $M^+$ be a positive integer with $(M^+,N^-)=1$ and let $M=N^-M^+$. Recall that $R_{M^+}$ denotes the Eichler order of level $M^+$. We put $\cU_M:=\wh R_{M^+}^\x$ and $\cS(M)=\cS(\cU_M)$. If $\pme\ndivides M$, let $T_\pme$ denote the operator
\[[\cU_M \DII{\pme}{1} \cU_M],\] and if $\pme\mid M$, let
\[U_\pme= \left[\cU_M \DII{\pme}{1} \cU_M\right] \text{ for }\pme\ndivides N^-\text{ and }
U_\pme=\left[\cU_M \uf_\pme \cU_M\right]\text{ for } \pme\mid N^-,
\]
where $\uf_\pme\in B^\x_\pme$ such that $\rmN(\uf_\pme)=\pme$.
Let $\bbT(M)$ be the Hecke algebra generated over $\cO$ by Hecke operators $T_\pme\text{ for }\pme\ndivides M$ and $U_\pme\text{ for }\pme|M$ in $\End_\cO\cS(M)$. Define the perfect pairing $\pairing_{M}:\cS(M)\x \cS(M)\to\cO$ by
 \beq\label{E:pairing}\pair{f_1}{f_2}_{M}:=\sum_{[g]} \Lpair{f_1(g)}{f_2(g\tau^{M^+})}\cdot(\#(B^\x\cap  g \cU_M g^{-1}\hatQ)/\Q^\x))^{-1},\eeq
 where $[g]$ runs over $B^\x\bksl \hatB/\cU_M\hatQ$. It is easy to verify that
 \[\pair{t f_1}{f_2}_{M}=\pair{f_1}{tf_2}_M\text{ for all }t\in\bbT(M).\]
%Then $\bbT(M)$ are self-adjoint operators with respect to $\pairing_{M}$.

Let $\lam_{\pi'}:\bbT(N)\to\cO$ be the $\cO$-algebra homomorphism induced by $\pi'$. Let $N^-_1$ be the product of prime factors of $N^-$ but not dividing $N_{\rho_0}$. Note that for each $\pme\mid N^-_1$ we have $\pme\not\con\pm 1\pmod{p}$ and $\Tr\rho_0(\Frob_\pme)^2\con (1+\pme)^2\pmod{p}$ by (CR${}^+$). Put $N_\emptyset:=N_{\rho_0}\cdot N^-_1$.
By the level lowering/raising (\cf\cite{Jarvis:level_lowering} and \cite{Diaomond_Taylor:Inventione}), there exists a modular lift $\lam_\emptyset:\bbT(N_\emptyset)\to\cO$ such that $\lam_\emptyset(T_\pme)\con \lam_{\pi'}(T_\pme)\pmod{\frakm_\cO}$ for all $\pme\ndivides N$ $(\frakm_\cO=\cO\cap\lam)$. We write
\[N=N_\emptyset\prod_{\pme}\pme^{m_\pme}.\]
Under the hypothesis (CR${}^+$), it is known (\cf\cite[pp.435-436]{Diaomond_Taylor:Inventione}) that
\begin{itemize}
\item $m_\pme\leq 2$;
\item $m_\pme=0$ unless $\pme\mid N^+$ and $\pme\ndivides N_\emptyset$;
\item If $m_\pme=1$, then $\pme\not\con 1\pmod{p}$.
\end{itemize}
 Let $\Sigma$ be a subset of prime factors of $N/N_\emptyset$. Set $N_\Sigma:=N_\emptyset\cdot \prod_{\pme\in\Sigma}\pme^{m_\pme}$.
 Let $\frakm_\Sigma$ be the maximal ideal of $\bbT(N_\Sigma)$ generated by
\[\frakm_\cO,\,T_\pme-\lam_\emptyset(T_\pme)\text{ for }\pme\ndivides N_\Sigma,\,U_\pme-\lam_{\pi'}(U_\pme)\text{ for }\pme \mid N_{\Sigma}.\]
Let $\bbT_{\Sigma}:=\bbT(N_\Sigma)_{\frakm_\Sigma}$ be the localization at $\frakm_\Sigma$.
\begin{lm}\label{L:9.W}\begin{mylist}\item If $\pme^2\mid N_\Sigma$, then the Hecke operator $U_\pme=0$ in $\bbT_\Sigma$. \item The Hecke algebra $\bbT_\Sigma$ is reduced.\end{mylist}
\end{lm}
\begin{proof}
Part (1) is clear if $\pme^2\mid N_\emptyset$, and if $\pme^2\mid N_\Sigma/N_\emptyset$, it is proved in \cite[Corollary\,1.8]{Taylor:Ihara_avoidance} (\cf\cite[Proposition 2.15]{Wiles:FLT}). To prove part (2), it suffices to show that $U_\pme$ are semisimple elements in $\bbT_\Sigma$ for $\pme\mid N_\Sigma$. This is clear by part (1) if $\pme\mid N_\emptyset$ or $\pme^2\mid N_\Sigma$, and it follows from (CR${}^+$) (4) if $\pme\parallel N_\Sigma/N_\emptyset$.
\end{proof}
Let $\e_p:G_\Q\to\Zp^\x$ be the \padic cyclotomic character. It is well known that there exists a Galois representation \[\rho_{\Sigma}:G_\Q\to\GL_2(\bbT_\Sigma)\]
such that
\begin{itemize}
\item $\rho_{\Sigma}$ is unramified outside $pN_\Sigma$.
\item $\Tr\rho_\Sigma(\Frob_\pme)=T_\pme\text{ for all }\pme\ndivides pN_\Sigma$ and $\det\rho_\Sigma=\e_p$.
\item There exists a character $\delta_p:G_p\to \bbT_\Sigma^\x$ such that
\[\rho_\Sigma|_{G_p}\sim \pMX{\delta_p^{-1}\e_p}{*}{0}{\delta_p}\text{ and }\delta_p|_{I_p}=\e^{(2-k)/2}.\]
\item For each $\pme\parallel N_\Sigma/N^-_1$, there exists a character $\delta_{\Sigma,\pme}:G_\pme\to\bbT_\Sigma^\x$ such that \[\rho_\Sigma|_{G_\pme}\sim \pMX{\delta_{\Sigma,\pme}^{-1}\e_p}{*}{0}{\delta_{\Sigma,\pme}}\text{ and }\delta_{\Sigma,\pme}(\Frob_\pme)=U_\pme.\]
    Here $\Frob_\pme$ is the arithmetic Frobenius at $\pme$.
\item For $\pme\mid N^-_1$,
\[\rho_\Sigma|_{G_\pme}\sim\pMX{\pm \e_p}{*}{0}{\pm 1},\,*\in\frakm_\Sigma.\]
     \end{itemize}

Let $\pme\mid N/N_\emptyset$ such that $\pme\not\in\Sigma$. We define an element $ u_{\Sigma,\pme}\in\bbT_\Sigma$ and a level-raising map $L_\pme: \cS_{\Sigma}\to \cS_{\Sigma\cup\stt{\pme}}$ as follows.
If $m_\pme=2$, set $ u_{\Sigma,\pme}:=0$ and define \[L_\pme(f)=\pme f-\DII{1}{\pme}T_\pme f+\DII{1}{\pme^2}f.\]
If $m_\pme=1$, then the Hecke polynomial $P_\pme(X)=X^2-T_\pme X+\pme\in\bbT_\Sigma[X]$ is congruent to $(X-\ep_\pme\pme)(X-\ep_\pme)$ modulo $\frakm_\Sigma$ for some $\ep_\pme\in\stt{\pm 1}$. Since $\pme\not\con 1\pmod{p}$, there exists a unique root $ u_{\Sigma,\pme}\in \bbT_\Sigma$ of $P_\pme(X)$ such that $ u_{\Sigma,\pme}\con \ep_\pme\pmod{\frakm_{\Sigma}}$ by Hensel's lemma. Define
\beq\label{E:15.W}L_\pme(f)= u_{\Sigma,\pme} f-\DII{1}{\pme}f.\eeq
In either case, it is easy to verify that $U_\pme\circ L_\pme=L_\pme \circ u_{\Sigma,\pme}$. Moreover, $L_\pme$ induces a surjective map $\bbT_{\Sigma\cup\stt{\pme}}\to\bbT_{\Sigma}$, sending $U_\pme$ to $ u_{\Sigma,\pme}$ by the following lemma:
\begin{lm}\label{L:8.W}The map $L_\pme$ is injective, and $\cS_{\Sigma\cup\stt{\pme}}/L_\pme(\cS_\Sigma)$ is a free $\cO$-module.
\end{lm}
\begin{proof}This is \cite[Lemma\,3.1]{Taylor:Ihara_avoidance} (\cf\lmref{L:Ihara2.W}).
\end{proof}

The above construction gives rise to a homomorphism $\bbT_{\Sigma}\surjto\bbT_{\emptyset}$. If $\lam:\bbT_\Sigma\to\cO$ is an $\cO$-algebra homomorphism, we write $I_\lam$ for the kernel of $\lam$ and put
\[\cS_\Sigma[\lam]:=\stt{x\in\cS_\Sigma\mid I_{\lam}x=0}.\]
Let $\lam_\Sigma:\bbT_\Sigma\to\bbT_\emptyset\stackrel{\lam_\emptyset}\longto \cO$ be the composition. %We put \[\cS_\Sigma[\lam_\Sigma]=\stt{x\in\cS_\Sigma\mid I_{\lam_\Sigma} x=0}.\]
The $\cO$-module $\cS_\Sigma[\lam_\Sigma]$ is free of rank one by the strong multiplicity one theorem and the reducedness of $\bbT_\Sigma$. Let $\cS_\Sigma[\lam_\Sigma]^\perp$ be the $\cO$-module defined by
\[\cS_\Sigma[\lam_\Sigma]^\perp=\stt{x\in \cS_\Sigma[\lam_\Sigma]\ot_\cO E\mid \pair{x}{y}_{N_\Sigma}\in\cO\text{ for all }y\in \cS_\Sigma[\lam_\Sigma]},\]
where $E$ is the fraction field of $\cO$.
Then $\cS_\Sigma[\lam_\Sigma]^\perp\supset \cS_\Sigma[\lam_\Sigma]$. We let $C(N_\Sigma):=\cS_\Sigma[\lam_\Sigma]^\perp/\cS_\Sigma[\lam_\Sigma]$ be the congruence module of $\lam_\Sigma$ and let $\eta_\Sigma=\lam_\Sigma(\Ann_{\bbT_\Sigma}I_{\lam_\Sigma})$ be the congruence ideal of $\lam_\Sigma$. It is known that \beq\label{E:14.W}\# C(N_\Sigma)\leq \#(\cO/\eta_{\Sigma})\eeq and the equality holds if $\cS_\Sigma$ is free over $\bbT_\Sigma$ (so $\bbT_\Sigma$ is Gorenstein).
\begin{lm}\label{L:congruencenumber}If $\pme\parallel N/N_\emptyset$ and $\pme\not\in\Sigma$, then we have \[\#C(N_{\Sigma\cup\stt{\pme}})=\#C(N_{\Sigma})\cdot \#(\cO/(\lam_\emptyset(u_{\emptyset,\pme})^2-1)\cO).\]
\end{lm}
\begin{proof}Note that $\lam_\emptyset(u_{\emptyset,\pme})\not =1$ as $\lam_\emptyset$ is unramified outside $N_\emptyset$. Let $L_\pme^*:\cS_{\Sigma}\to\cS_{\Sigma\cup\stt{\pme}}$ be the adjoint map of $L_\pme$ with respect to $\pairing_{N_{\Sigma}}$ and $\pairing_{N_{\Sigma\cup\stt{\pme}}}$. It follows from \lmref{L:8.W} that \beq\label{E:11.W}L_\pme(\cS_\Sigma[\lam_\Sigma])=\cS_{\Sigma\cup\stt{\pme}}[\lam_{\Sigma\cup\stt{\pme}}]\,;\,
L_\pme^*(\cS_{\Sigma\cup\stt{\pme}}[\lam_{\Sigma\cup\stt{\pme}}]^\perp)=\cS_\Sigma[\lam_\Sigma]^\perp.\eeq
Let $\cU=\cU_{N_{\Sigma}}$ and $\cU_1=\cU_{N_{\Sigma\cup\stt{\pme}}}$. A direct computation shows that
\[L_\pme^*= u_{\Sigma,\pme}\cdot [\cU\DII{\pme}{1}\cU_1]-[\cU\cU_1].\]
%\begin{lm}
%\begin{align*}
%\al^*=&[K \al^{-1}\tau K_1]\cdot\frac{\vol(K\cap {}^{\al^{-1}\tau_\pme}K_1)}{\vol(K_1)}\\
%[K\al K_1]\circ \beta=&[K\al\beta K]\cdot \frac{\vol(K\cap {}^{\al\beta}K)}{\vol(K\cap {}^\al K_1)}.
%\end{align*}
%Thus
%\[\al^*\circ\beta=[KgK]\cdot \frac{\vol(K\cap {}^gK)}{\vol(K_1)},\,g=\al^{-1}\tau_\pme\beta.\]
%\end{lm}
Therefore, we find that
\begin{align*}
L^*_\pme\circ L_\pme%&=( u_{\Sigma,\pme}\cdot 1^*-\DII{1}{\pme}^*)(1\cdot  u_{\Sigma,\pme}-\DII{1}{\pme})\\
=& u_{\Sigma,\pme}\cdot[\cU\DII{\pme}{1}\cU]\cdot  u_{\Sigma,\pme}- u_{\Sigma,\pme}\cdot [\cU\DII{\pme}{\pme}\cU]\cdot (1+\pme)-[\cU\cU]\cdot (1+\pme) u_{\Sigma,\pme}+[\cU\DII{1}{\pme}\cU]\\
=&(1+ u_{\Sigma,\pme}^2)\cdot T_\pme -2 u_{\Sigma,\pme}\cdot (1+\pme)\\%\quad(T_\pme =[K\DII{\pme}{1}K]\cdot \pme^{k/2-1},\,S_\pme=[K\DII{\pme}{\pme}K]\cdot \pme^{k-2})\\
=& u_{\Sigma,\pme}^{-1}( u_{\Sigma,\pme}^2-1)( u_{\Sigma,\pme}^2-\pme).
\end{align*}
Since $ u_{\Sigma,\pme}^2-\pme\con 1-\pme\not\con 0\pmod{\frakm_\Sigma}$, by \eqref{E:11.W} we find that
\begin{align*}
\#C(N_{\Sigma\cup\stt{\pme}})&=\#(\cS_\Sigma[\lam_\Sigma]^\perp/L^*_\pme\circ L_\pme(\cS_\Sigma[\lam_\Sigma]))\\
&=\#C(N_\Sigma))\cdot \#\cO/(\lam_\emptyset(u_{\emptyset,\pme})^2-1)\cO.\qedhere
\end{align*}
\end{proof}

\subsection{Deformation rings and Selmer groups}
We introduce a certain deformation ring. Recall that $\rho$ is a deformation of $\rho_0$ if $\rho\pmod{\frakm_A}\iso \rho_0$. Consider the functor $\frakD_\Sigma$ from local Noetherian complete $\cO$-algebras with the residual field $\bfk$ to sets which sends $A$ with the maximal ideal $\frakm_A$ to the isomorphism classes of deformations $\rho:G_\Q\to\GL_2(A)$ of $\rho_0$ satisfying:\begin{itemize}
\item[(D1)] $\det \rho=\e_p$;
\item[(D2)] $\rho$ is minimally ramified outside $N^-_1\Sigma$ in the sense of \cite[Definition\,3.1]{Diamond:Boston};
\item[(D3)] There exists a character $\delta_p:G_\pme\to A^\x$ such that \[\rho|_{G_p}\sim \pMX{\delta_p^{-1}\e_p}{*}{0}{\delta_p}\text{ and }\delta_p|_{I_p}=\e_p^{(2-k)/2};\]
\item[(D4)] For each $\pme\parallel N_\Sigma/N_\emptyset$, there exists a unramified character $\delta_\pme:G_\pme\to A^\x$ such that
\[\rho|_{G_\pme}\sim\pMX{\delta_\pme^{-1}\e_p}{*}{0}{\delta_\pme}\text{ and }\delta_\pme(\Frob_\pme)\con 1\pmod{\frakm_A}.\]
\item[(D5)] If $\pme\mid N^-_1$, then $\rho|_{G_\pme}$ satisfies the sp-condition in \cite[Definition\,2.2]{Terracini:TW}. Namely,
\[\rho|_{G_\pme}\sim\pMX{\pm \e_p}{*}{0}{\pm 1},\,*\in\frakm_A.\]
\end{itemize}
Under (CR${}^+$), it is a standard fact that $\frakD_\Sigma$ is represented by the universal deformation
\[\rho_{R_\Sigma}:G_\Q\to\GL_2(R_\Sigma).\]
The universal property of $R_\Sigma$ gives rise to $\cO$-algebra homomorphisms $R_\Sigma\to R_\emptyset$ and $R_\Sigma\to \bbT_\Sigma$ under which $\rho_{R_\Sigma}$ pushes forward to $\rho_{R_\emptyset}$ and $\rho_{\Sigma}$ respectively.
\begin{lm}\label{L:5.W}The map $R_\Sigma\to \bbT_\Sigma$ is a surjection.
\end{lm}
\begin{proof}By \lmref{L:9.W}, $\bbT_\Sigma$ is generated by $T_\pme$ for $\pme\ndivides N_\Sigma$ and $U_\pme$ for $\pme\parallel N_\Sigma$. If $\pme\ndivides N_\Sigma$, we have $\Tr\rho_{R_\Sigma}(\Frob_\pme)\mapsto T_\pme$. If $\pme| N^-_1$, then $U_\pme=\pm 1$ by the sp-condition. If $\pme\parallel N_\Sigma/N^-_1$, then \[\rho_{R_\Sigma}|_{G_\pme}\iso \pMX{\delta_{R_\Sigma,\pme}^{-1}\e_p}{*}{}{\delta_{R_\Sigma,\pme}}\text{ for a unramified character }\delta_{R_\Sigma,\pme}:G_\pme\to R_\Sigma^\x\]
such that $\delta_{R_\Sigma,\pme}(\Frob_\pme)\con 1\pmod{\frakm_{R_\Sigma}}$, and we have $\delta_{R_\Sigma,\pme}(\Frob_\pme)\mapsto U_\pme$.
\end{proof}
Let $\wp_{\Sigma}$ be the kernel of the $\cO$-algebra morphism:
\[R_{\Sigma}\to R_{\emptyset}\to \bbT_\emptyset\stackrel{\lam_\emptyset}\longto \cO.\]
Let $\Adjrep$ denote the discrete Galois module $\ad^0\rho_{\lam_\emptyset}\ot E/\cO$. Define the subspace $\Adjrep^+$ by
\[\Adjrep^+=\stt{\pMX{a}{b}{0}{-a}\mid a,b\in E/\cO}\subset \Adjrep=\stt{\pMX{a}{b}{c}{-a}\mid a,b,c\in E/\cO}.\]
We define the subgroup $L_\pme\subset H^1(\Q_\pme,\Adjrep)$ as follows.
We denote
$L_\pme=H^1(\Q_\pme,\Adjrep^+)$ if $\pme\parallel  N_\Sigma/N_{\emptyset}$,
$L_\pme=\ker\stt{H^1(\Q_\pme,\Adjrep)\to H^1(\Dmd{F},\Adjrep)}$, where $F$ is a lifting of $\Frob_\pme$ in $G_\pme$ if $\pme\mid N^-_1$, and $L_\pme=H^1_f(\Q_\pme,\Adjrep)$ be the local Bloch-Kato group otherwise. Define the Selmer group by
\begin{align*}
%\Sel_{\Sigma}(\Adjrep):=&\ker\stt{H^1(\Q,\Adjrep)\to\prod_{\pme\not\in\Sigma}\frac{H^1(\Q_\pme,\Adjrep)}{H^1_f(\Q_\pme,\Adjrep)}};\\
\Sel_{\Sigma}(\Adjrep):=&\ker\stt{H^1(\Q,\Adjrep)\to\prod_{\pme}\frac{H^1(\Q_\pme,\Adjrep)}{L_\pme}}.
\end{align*}
It is not difficult to show that we have an $\cO$-module isomorphism:
\beq\label{E:12.W}\begin{aligned}
%\Hom_{\cO}(\wp_{\Sigma}/\wp_{\Sigma}^2, E/\cO)=\Sel_{\Sigma}(\Adjrep);\\
\Hom_{\cO}(\wp_{\Sigma}/\wp_{\Sigma}^2,E/\cO)\iso\Sel_{\Sigma}(\Adjrep).
\end{aligned}
\eeq
Let $\Sgsc$ be the set of prime factors $\pme\mid N/N_\emptyset$ with $m_\pme=2$.
\begin{lm}\label{L:1.W}We have a natural inclusion map \[H^1_f(\Q_\pme,\Adjrep)\hookto H^1(\Q_\pme,\Adjrep^+).\]In particular, if $\Sigma\supset \Sgsc$, then we have
\[\#(\Sel_{\Sigma}(\Adjrep))/\#(\Sel_{\Sgsc}(\Adjrep))\mid\prod_{\pme\parallel N_\Sigma/N_\emptyset}\#(H^1(I_\pme,\Adjrep^+)^{G_\pme}).\]\end{lm}
\begin{proof}
Let $\pme\parallel N/N_\emptyset$. Let $\Adjrep^-:=\Adjrep/\Adjrep^+$ be a discrete $G_\pme$-module of $\cO$-corank one. Consider the following diagram:
\[\xymatrix{ 0\ar[r]&H^1_f(G_\pme,\Adjrep)\ar[r]\ar[d]&H^1(G_\pme,\Adjrep)\ar[r]^{\pi}\ar[d]^{\pi_1}& H^1(I_\pme,\Adjrep)^{G_\pme}\ar[d]^{\pi_2}\ar[r]&0\\
0\ar[r]&H^1(G_\pme/I_\pme,\Adjrep^-)\ar[r]&H^1(G_\pme,\Adjrep^-)\ar[r]&H^1(I_\pme,\Adjrep^-)^{G_\pme}\ar[r]&0.
}\]
Since $\pme\not\con 1\pmod{p}$, $H^0(G_\pme,\Adjrep^-)=H^1(G_\pme/I_\pme,\Adjrep^-)=0$. It follows that $\ker\pi_1=L_\pme=H^1(G_\pme,\Adjrep^+)$ and $H^1(G_\pme/I_\pme,\Adjrep^+)\iso H^1_f(G_\pme,\Adjrep)$. By the snake lemma, $H^1_f(G_\pme,\Adjrep)$ is a submodule of $L_\pme$. The second assertion follows from the exact sequence \[0\to\Sel_{\Sgsc}(\Adjrep)\to \Sel_{\Sigma}(\Adjrep)\to\prod_{\pme\parallel N_\Sigma/N_\emptyset}\frac{L_\pme}{H^1_f(G_\pme,\Adjrep)}\]
and the isomorphism
\[H^1(I_\pme,\Adjrep^+)^{G_\pme}\iso H^1(G_\pme,\Adjrep^+)/H^1(G_\pme/I_\pme,\Adjrep^+)\iso L_\pme/H^1_f(G_\pme,\Adjrep).\qedhere\]
\end{proof}
\begin{cor}\label{C:2.W}If $\Sigma$ is the set of prime factors of $N/N_\emptyset$, then we have \[\#(\wp_{\Sigma}/\wp_{\Sigma}^2)\mid \#(\wp_{\Sgsc}/\wp_{\Sgsc}^2)\cdot
\prod_{\pme\parallel N/N_\emptyset}\#(\cO/(\lam_\emptyset(u_{\emptyset,\pme})^2-1)\cO).\]\end{cor}
\begin{proof}
By \eqref{E:12.W} and \lmref{L:1.W},
\[\#\ker(\wp_{\Sigma}/\wp_{\Sigma}^2\to\wp_{\Sgsc}/\wp_{\Sgsc}^2)=\#(\Sel_{\Sigma}(\Adjrep)/\Sel_{\Sgsc}(\Adjrep))\]
divides
\begin{align*}
\prod_{\pme\parallel N/N_\emptyset}\#(H^1(I_\pme,\Adjrep^+)^{G_\pme})
=&\prod_{\pme\parallel N/N_\emptyset}\#H^0(G_\pme,\Adjrep^+(-1))\\
=&\prod_{\pme\parallel N/N_\emptyset}\#(\cO/((\pme-1)(\lam_\emptyset(u_{\emptyset,\pme})^2-1)\cO)).\qedhere
\end{align*}
\end{proof}
\subsection{Proof of \propref{P:congruence}}\label{SS:proof}Let $\Sigma$ be the set of prime factors of $N/N_\emptyset$. We begin with the following proposition on the freeness of the Hecke module $\cS_\Sigma$.
\begin{prop}\label{P:3.W} With the hypothesis $(\mathrm{CR}{}^+)$, $\cS_\Sigma$ is a free $\bbT_\Sigma$-module of rank one.\end{prop}
\begin{proof}
First consider the case $\Sigma=\emptyset$. We have
\[\Hom(\frakm_\emptyset/\frakm_\emptyset^2,\cO/\frakm_\cO)=\ker\stt{H^1(\Q,\ad^0\rho_0)\to \prod_\pme\frac{H^1(\Q_\pme,\ad^0\rho_0)}{\cL_\pme}},\]
where $\cL_\pme=H^1_f(\Q_\pme,\ad^0\rho_0)$ if $\pme\ndivides N^-_1$ and $\cL_\pme=\ker\stt{H^1(\Q_\pme,\ad^0\rho_0)\to  H^1(\Dmd{F},\ad^0\rho_0)}$ for a lifting $F$ of $\Frob_\pme$ in $G_\pme$ if $\pme\mid N^-_1$. This is the \emph{minimal} case in the sense that $\#(\cL_\pme)=\#(H^0(G_\pme,\ad^0\rho_0))$ for all $\pme$ (\cite[\S3.4]{Terracini:TW}). Using the Taylor-Wiles system constructed in \cite[\S2]{Taylor:Ihara_avoidance}, we deduce that $\cS_\emptyset$ is a free $\bbT_\emptyset$-module of rank one and
\[\#(\wp_{\emptyset}/\wp_{\emptyset}^2)=\# C(N_{\emptyset})=\#(\cO/\eta_{\emptyset}).\]
Furthermore, the argument in \cite[\S3]{Taylor:Ihara_avoidance} shows that
\[\#(\wp_{\Sgsc}/\wp_{\Sgsc}^2)=\# C(N_{\Sgsc})=\#(\cO/\eta_{\Sgsc}).\]
Combined with \lmref{L:congruencenumber} and \corref{C:2.W}, the above equation yields that
\[\#(\wp_{\Sigma}/\wp_{\Sigma}^2)\mid \#C(N)\mid \#(\cO/\eta_{\Sigma}).\]
The proposition follows from \cite[Theorem\,2.4]{Diamond:TW}.
\end{proof}
Now we are ready to prove \propref{P:congruence}. The Jacquet-Langlands correspondence induces a surjective $\cO$-algebra homomorphism $JL^*:\bbT(\Gamma_0(N))_\frakm\twoheadrightarrow\bbT_{\Sigma}$ such that $\lam_{\pi}=JL^*\circ\lam_{\pi'}$, where $\frakm$ is the maximal ideal containing $\ker\lam_\pi$. The assumption $N^-\mid N_{\rho_0}$ implies $JL^*$ is an isomorphism. On the other hand, by definition $\cS_\Sigma[\lam_{\pi'}]=\cO\cdot \pvformB$ and $\PetB{\vformB}{\vformB}=\pair{\pvformB}{\pvformB}_{N_B}$. Therefore, by \propref{P:3.W} we conclude that
\[\#(\cO/I_\pi(N))=\#(\cS_\Sigma[\lam_{\pi'}]^\perp/\cS_\Sigma[\lam_{\pi'}])=\#(\cO/\PetB{\vformB}{\vformB}\cO).\]
This completes the proof. 
\bibliographystyle{amsalpha}
\bibliography{mybib}

\providecommand{\bysame}{\leavevmode\hbox to3em{\hrulefill}\thinspace}
\providecommand{\MR}{\relax\ifhmode\unskip\space\fi MR }
\providecommand{\MRhref}[2]{%
  \href{http://www.ams.org/mathscinet-getitem?mr=#1}{#2}
}
\providecommand{\href}[2]{#2}
\begin{thebibliography}{BDIS02}

\bibitem[BD96]{BD:Heegner_Mumford}
M.~Bertolini and H.~Darmon, \emph{Heegner points on {M}umford-{T}ate curves},
  Invent. Math. \textbf{126} (1996), no.~3, 413--456.

\bibitem[BD05]{Bertolini_Darmon:IMC_anti}
\bysame, \emph{Iwasawa's main conjecture for elliptic curves over
  anticyclotomic {$\Bbb Z\sb p$}-extensions}, Ann. of Math. (2) \textbf{162}
  (2005), no.~1, 1--64.

\bibitem[BDIS02]{BD:AJM_exceptional_zero}
Massimo Bertolini, Henri Darmon, Adrian Iovita, and Michael Spiess,
  \emph{Teitelbaum's exceptional zero conjecture in the anticyclotomic
  setting}, Amer. J. Math. \textbf{124} (2002), no.~2, 411--449.

\bibitem[Cas73]{Casselman:Atkin-Lehner}
W.~Casselman, \emph{On some results of {A}tkin and {L}ehner}, Math. Ann.
  \textbf{201} (1973), 301--314.

\bibitem[CV05]{Vatsal_Cornut:Documenta}
C.~Cornut and V.~Vatsal, \emph{C{M} points and quaternion algebras}, Doc. Math.
  \textbf{10} (2005), 263--309.

\bibitem[Dia97a]{Diamond:Boston}
F.~Diamond, \emph{An extension of {W}iles' results}, Modular forms and
  {F}ermat's last theorem ({B}oston, {MA}, 1995), Springer, New York, 1997,
  pp.~475--489.

\bibitem[Dia97b]{Diamond:TW}
\bysame, \emph{The {T}aylor-{W}iles construction and multiplicity one}, Invent.
  Math. \textbf{128} (1997), no.~2, 379--391.

\bibitem[DT94]{Diaomond_Taylor:Inventione}
F.~Diamond and R.~Taylor, \emph{Nonoptimal levels of mod {$l$} modular
  representations}, Invent. Math. \textbf{115} (1994), no.~3, 435--462.

\bibitem[Hid81]{Hida:congruence_number}
H.~Hida, \emph{Congruence of cusp forms and special values of their zeta
  functions}, Invent. Math. \textbf{63} (1981), no.~2, 225--261.

\bibitem[Jac72]{Jacquet:GLtwoPartII}
H.~Jacquet, \emph{Automorphic forms on {${\rm GL}(2)$}. {P}art {II}}, Lecture
  Notes in Mathematics, Vol. 278, Springer-Verlag, Berlin, 1972.

\bibitem[Jar99]{Jarvis:level_lowering}
Frazer Jarvis, \emph{Level lowering for modular mod {$l$} representations over
  totally real fields}, Math. Ann. \textbf{313} (1999), no.~1, 141--160.

\bibitem[JL70]{Jacquet_Langlands:GLtwo}
H.~Jacquet and R.~P. Langlands, \emph{Automorphic forms on {${\rm GL}(2)$}},
  Lecture Notes in Mathematics, Vol. 114, Springer-Verlag, Berlin, 1970.

\bibitem[LV10]{Longo_Vigni:Selmer_groups}
M.~Longo and S.~Vigni, \emph{On the vanishing of {S}elmer groups for elliptic
  curves over ring class fields}, J. Number Theory \textbf{130} (2010), no.~1,
  128--163.

\bibitem[PTB11]{Prasad:Bessel}
D.~Prasad and R.~Takloo-Bighash, \emph{Bessel models for {GS}p(4)}, J. Reine
  Angew. Math. \textbf{655} (2011), 189--243.

\bibitem[PW11]{Pollack_Weston:AMU}
R.~Pollack and T.~Weston, \emph{On anticyclotomic {$\mu$}-invariants of modular
  forms}, Compos. Math. \textbf{147} (2011), no.~5, 1353--1381.

\bibitem[Sch02]{Schmidt:newform}
Ralf Schmidt, \emph{Some remarks on local newforms for {$\rm GL(2)$}}, J.
  Ramanujan Math. Soc. \textbf{17} (2002), no.~2, 115--147. \MR{1913897
  (2003g:11056)}

\bibitem[Sin87]{Sinnott:W_Theorem}
W.~Sinnott, \emph{On a theorem of {L}. {W}ashington}, Ast\'erisque (1987),
  no.~147-148, 209--224, 344, Journ{\'e}es arithm{\'e}tiques de Besan{\c{c}}on
  (Besan{\c{c}}on, 1985).

\bibitem[SU10]{SU:IMC_GL(2)}
C.~Skinner and E.~Urban, \emph{The {I}wasawa main conjectures for {$\GL_2$}},
  Preprint, 226 p., November, 2010.

\bibitem[Tay06]{Taylor:Ihara_avoidance}
R.~Taylor, \emph{On the meromorphic continuation of degree two
  {$L$}-functions}, Doc. Math. (2006), no.~Extra Vol., 729--779 (electronic).

\bibitem[Ter03]{Terracini:TW}
L.~Terracini, \emph{A {T}aylor-{W}iles system for quaternionic {H}ecke
  algebras}, Compositio Math. \textbf{137} (2003), no.~1, 23--47.

\bibitem[Vat02]{Vatsal:Uniform_CM}
V.~Vatsal, \emph{Uniform distribution of {H}eegner points}, Invent. Math.
  \textbf{148} (2002), no.~1, 1--46.

\bibitem[Vat03]{Vatsal:nonvanishing}
\bysame, \emph{Special values of anticyclotomic {$L$}-functions}, Duke Math. J.
  \textbf{116} (2003), no.~2, 219--261.

\bibitem[VO12]{VanOrder}
Jeanine Van~Order, \emph{On the quaternionic {$p$}-adic {$L$}-functions
  associated to {H}ilbert modular eigenforms}, Int. J. Number Theory \textbf{8}
  (2012), no.~4, 1005--1039.

\bibitem[Wal85]{Waldsupurger:Central_value}
J.-L. Waldspurger, \emph{Sur les valeurs de certaines fonctions {$L$}
  automorphes en leur centre de sym\'etrie}, Compositio Math. \textbf{54}
  (1985), no.~2, 173--242.

\bibitem[Wil95]{Wiles:FLT}
A.~Wiles, \emph{Modular elliptic curves and {F}ermat's last theorem}, Ann. of
  Math. (2) \textbf{141} (1995), no.~3, 443--551.

\bibitem[Yua05]{YuanHP:Thesis}
H.~Yuan, \emph{Special {V}aule {F}ormulae of {R}ankin-{S}elberg
  {$L$}-functions}, Ph.D. thesis, University of Pennsylvania, 2005.

\bibitem[Zha04]{Zhang:GZII}
Shou-Wu Zhang, \emph{Gross-{Z}agier formula for {$\rm GL(2)$}. {II}}, Heegner
  points and {R}ankin {$L$}-series, Math. Sci. Res. Inst. Publ., vol.~49,
  Cambridge Univ. Press, Cambridge, 2004, pp.~191--214. \MR{2083213
  (2005k:11121)}

\end{thebibliography}
\end{document}